\documentclass[11pt,letterpaper]{amsart}
\usepackage[utf8]{inputenc}
\usepackage{amsfonts, amsmath, amssymb, amsthm, color, bookmark,graphicx,subfig, mathabx, tikz-cd, enumitem}

\usepackage{hyperref}
\hypersetup{
  pdfborder={0 0 0},
  colorlinks   = true,
  urlcolor     = blue,
  linkcolor    = blue,
  citecolor   = red
}

\allowdisplaybreaks

\usepackage[centertags]{amsmath}

\usepackage[highstructure]{accessibility}

\usepackage{mathrsfs}

\IfFormatAtLeastTF{2026-06-01}{}{%
  \usepackage{aliascnt}%
}
\usepackage{cleveref}

\newcommand{\newsharedtheorem}[4]{%
  \IfFormatAtLeastTF{2026-06-01}
    {\newtheorem{#1}[theorem]{#2}}
    {\newaliascnt{#1}{theorem}%
     \newtheorem{#1}[#1]{#2}%
     \aliascntresetthe{#1}}%
  \crefname{#1}{#1}{#3}%
  \Crefname{#1}{#2}{#4}%
}

\theoremstyle{plain}
\newtheorem{theorem}{Theorem}[section]
\crefname{theorem}{theorem}{theorems}
\Crefname{theorem}{Theorem}{Theorems}

\newsharedtheorem{lemma}{Lemma}{lemmas}{Lemmas}
\newsharedtheorem{corollary}{Corollary}{corollaries}{Corollaries}
\newsharedtheorem{proposition}{Proposition}{propositions}{Propositions}
\newsharedtheorem{conjecture}{Conjecture}{conjectures}{Conjectures}

\newtheorem{claim}{Claim}[theorem]
\crefname{claim}{claim}{claims}
\Crefname{claim}{Claim}{Claims}

\theoremstyle{remark}
\newsharedtheorem{remark}{Remark}{remarks}{Remarks}

\theoremstyle{definition}
\newsharedtheorem{definition}{Definition}{definitions}{Definitions}
\newsharedtheorem{notation}{Notation}{notations}{Notations}
\newsharedtheorem{example}{Example}{examples}{Examples}
\newsharedtheorem{question}{Question}{questions}{Questions}

\newenvironment{claimproof}[1][Proof of the claim]
  {%
   \begin{proof}[#1]}
  {\end{proof}}
                        
\renewcommand{\ll }{\langle\hspace{-.7mm}\langle }
\newcommand{\rr }{\rangle\hspace{-.7mm}\rangle }

\newcommand{\bg}{\overline{G}}
\newcommand{\bh}{\overline{H}}
\newcommand{\Z}{\mathbb{Z}}
\newcommand{\N}{\mathcal{N}}

\newcommand{\C}{\mathbb{C}}
\newcommand{\Q}{\mathbb{Q}}
\newcommand{\fin}{\mathcal{FIN}}
\newcommand{\h}{\mathcal H}
\newcommand{\cH}{\mathscr{H}}

\DeclareMathOperator{\cd}{cd}
\DeclareMathOperator{\mg}{MG}
\DeclareMathOperator{\rk}{rk}
\DeclareMathOperator{\im}{im}

\DeclareMathOperator{\defi}{defi}
\DeclareMathOperator{\vcd}{vcd}

\newcommand{\F}{\mathcal F}
\newcommand{\fS}{\mathcal S}

\usepackage[textsize=tiny]{todonotes}

\begin{document}

\title{$L^2$-Betti numbers of Dehn fillings}
\author{Nansen Petrosyan}
\address{School of Mathematics, University of Southampton, Southampton SO17~1BJ, UK}
\email{n.petrosyan@soton.ac.uk}
\thanks{}
\author{Bin Sun}
\address{Department of Mathematics, Michigan State University, East Lansing, MI, 48824, USA}
\email{sunbin1@msu.edu}

\begin{abstract} 
We initiate the study of the $L^2$-Betti numbers of group-theoretic Dehn fillings. 
For  a broad class of virtually special  groups $G$, we prove that sufficiently deep Dehn fillings $\bg$ have the same $L^2$-Betti numbers as $G$. As applications, we verify the Singer Conjecture for certain Einstein manifolds, establish a virtual fibering criterion for $\bg$, obtain bounds on deficiency of $\bg$, and provide new examples of hyperbolic groups with exotic subgroups that arise as Dehn fillings of any cusped arithmetic hyperbolic manifold of dimension at least four.
\end{abstract}

\maketitle

\section{Introduction}

Dehn filling is a process that produces a closed manifold from one with boundary. Let $M$ be an $n$-manifold with boundary an $(n-1)$-torus. A \textit{Dehn filling} of $M$ is a manifold obtained by gluing an $n$-dimensional solid torus $D^2\times S^{n-1}$ to $M$ along their boundaries via a homeomorphism.  It is particularly central to the study of $3$-manifolds.  Notably, Thurston's Hyperbolic Dehn Filling Theorem \cite{thurston1983three} states that if $M$ is a $3$-dimensional hyperbolic knot complement, then most Dehn fillings of $M$ admit a hyperbolic Riemannian metric, producing examples of non-Haken hyperbolic $3$-manifolds and giving evidence to his Geometrization Conjecture, which has been solved by Perelman \cite{Perelman2002, Perelman2003}.

There is an analogous construction in group theory, called
\emph{(group-theoretic) Dehn filling}, which was introduced by \cite{osin2007peripheral, groves2008dehn} and has important applications
in group theory and topology. A key ingredient in Agol's proof
\cite{agol2013virtual} of the Virtual Haken Conjecture is the theorem
that every hyperbolic group acting properly and cocompactly on a
$\mathrm{CAT}(0)$ cube complex is virtually compact special, in whose
proof Dehn filling plays a crucial role. Other representative
applications include the construction of free, purely pseudo-Anosov
normal subgroups of mapping class groups
\cite{dahmani2017hyperbolically} and of acylindrically hyperbolic
quotients of mapping class groups obtained by imposing suitable
power relations on Dehn twists \cite{dahmani2021dehn}. Dehn filling
has also been used to establish the Farrell--Jones conjecture for
certain relatively hyperbolic groups \cite{antolin2017farrell} and
to solve the isomorphism problem for certain classes of relatively
hyperbolic groups \cite{dahmani2018recognizing}. Further applications
include results on the preservation of the absence of elementary
splittings and of certain connectivity properties of Bowditch
boundaries under sufficiently long fillings \cite{groves2018dehn},
as well as the construction of the first $W^{*}$-superrigid groups
with Kazhdan's property~$(T)$ \cite{chifan2021wreath}. Finally, Dehn filling was used to show that the Cannon Conjecture implies the Relative Cannon Conjecture for toral relatively hyperbolic groups \cite{groves2016boundaries}.

In its most general form, group-theoretic Dehn filling involves a triple \((G, H, N)\), where \(G\) is a group, \(H \leqslant G\) is a subgroup, and \(N \lhd H\) is a normal subgroup. The Dehn filling process produces the quotient group \(G / \ll N \rr\) as its output, where \(\ll N \rr\) is the normal closure of $N$ in $G$. This algebraic framework mirrors the topological notion of Dehn filling, as seen through the Seifert--van Kampen theorem. Specifically, let \(G = \pi_1(M)\), \(H\) be the image of \(\pi_1(\partial M)\) in \(\pi_1(M)\), and \(N\) the image of \(\pi_1(\partial D^2)\). The gluing process in topology then corresponds to the group-theoretic construction.

The general setting of Dehn filling is overly broad to allow for meaningful conclusions. A typical strategy in practice is to assume that $H$ is a hyperbolically embedded subgroup of $G$, which is a generalization of $G$ being hyperbolic relative to $H$, and prove that for many choices of $N\lhd H$, the quotient $G/\ll N \rr$ will have certain properties. For example, a critical ingredient of Agol's work \cite{agol2013virtual} is Wise's Malnormal Special Quotient Theorem \cite{wise2021structure,agol2016alternate}, which states that if $G$ is a virtually compact special hyperbolic group, and $G$ is hyperbolic relative to its subgroup $H$, then for many choices of $N\lhd H$, the quotient $G/\ll N \rr$ will be hyperbolic virtually compact special.

In this paper, we consider the case where $H$ is a hyperbolically embedded subgroup of $G$ (denoted $H\hookrightarrow_h G$) and compute the $L^2$-Betti numbers of $G/\ll N \rr$ for many choices of $N\lhd H$. To simplify notation, we will often write $\bh := H/N$ and $\bg := G/\ll N \rr$.

For simplicity of exposition, most of our results in the introduction are stated for a single peripheral subgroup, but they hold more generally for multiple peripherals.

\subsection*{$L^2$-Betti numbers of Dehn filling.}For a discrete group $G$, the $L^2$-Betti numbers
$b_n^{(2)}(G)$, $n\geq 0$, are homological invariants defined using
von Neumann dimension over the group von Neumann algebra
$\mathcal{N}(G)$ in place of ordinary vector-space dimension.
They satisfy the finite-index scaling formula
$b_n^{(2)}(H)=[G:H]b_n^{(2)}(G)$ whenever $H$ has finite index in
$G$, motivating their interpretation as normalized Betti numbers.
This interpretation is further illustrated by L\"uck's
Approximation Theorem \cite{luck1994approximating}: if $G$ is
residually finite and admits a classifying space with finitely
many cells in each dimension, then for every nested sequence
$G=G_0\geq G_1\geq\cdots$ of finite-index normal subgroups with
$\bigcap_{i\geq 0}G_i=\{1\}$, we have
\[
b_n^{(2)}(G)
=
\lim_{i\to\infty}
\frac{b_n(G_i;\mathbb{Q})}{[G:G_i]}.
\]
Here $b_n(G_i;\mathbb{Q})$ denotes the ordinary rational Betti
number. We refer the reader to L\"uck's monograph
\cite{luck2002l2} for a systematic treatment of the definitions,
basic properties, and applications of $L^2$-invariants.

Our main result is the following theorem, which computes the $L^2$-Betti numbers of sufficiently deep Dehn fillings for virtually compact cubical groups that are hyperbolic relative to virtually abelian subgroups. 

In what follows, a group is called \textit{compact cubical} if it is the fundamental group of a compact locally $\mathrm{CAT}(0)$ cube complex, and a group is said to \textit{virtually} have property $\mathcal P$ if it has a finite-index subgroup that has property $\mathcal P$.

\begin{theorem}[\Cref{thm. sparse rel. hyp.}]
    \label{thm. intro rel hyp}
    Let $G$ be a virtually compact cubical group that is hyperbolic relative to a virtually abelian subgroup $H\leqslant G$. Then there exists a finite-index torsion-free normal subgroup $K\lhd H$ such that for every sufficiently deep normal subgroup $N\lhd H$ that satisfies $N\leqslant K$ and $K/N$ is torsion-free, we have
    \[b^{(2)}_\ast(\bg)=b^{(2)}_\ast(G).\]
\end{theorem}

 The next theorem provides an approximation result for $L^2$-Betti numbers of locally indicable groups and serves as a foundation for the proof of \Cref{thm. intro rel hyp}.  Recall that a group is \textit{locally indicable} if every non-trivial finitely generated subgroup of it has a homomorphism onto $\Z$. 

\begin{theorem}[see also \Cref{cor. gen li}]
    \label{thm. intro li}
    Let $G$ be a virtually locally indicable group of type $F_\infty$ that has a hyperbolically embedded amenable subgroup $H\hookrightarrow_h G$. Then for every $\delta>0$ and $n\geq 0$, there exists a finite subset $\mathcal{F}_{n,\delta}\subset H\smallsetminus\{1\}$ such that if a normal subgroup $N\lhd H$ satisfies $N\cap\mathcal{F}_{n,\delta}=\emptyset$ and $\bh$ is of type $F_\infty$, then
    \[| b^{(2)}_n(\bg)-b^{(2)}_n(G)|< \delta.\]
\end{theorem}

 For groups that are virtually compact cubical and hyperbolic relative to arbitrary subgroups, we also obtain the following inequality.

\begin{theorem}[\Cref{thm. hyp.} (see also \Cref{thm. gen hyp.})]
    \label{thm. intro hyp}
    Let $G$ be a virtually compact cubical hyperbolic group, and suppose that $G$ is hyperbolic relative to a subgroup $H$. Then there exists a finite-index torsion-free normal subgroup $K\lhd H$ such that for every sufficiently deep normal subgroup $N\lhd H$ that satisfies
    \begin{itemize}
        \item $N\leqslant K$, and
        \item $K/N$ is torsion-free virtually compact cubical hyperbolic,
    \end{itemize}
    we have
    \[b^{(2)}_1(\bg)\leqslant b^{(2)}_1(G).\]
\end{theorem}

\begin{remark}
    The torsion-free assumption  on $K/N$ in \Cref{thm. intro rel hyp} and  \Cref{thm. intro hyp} cannot be removed. We refer to \Cref{rm. torsion-free necessary} for an example. The passage from $H$ to its finite-index subgroup $K$ is more subtle -- it can be omitted if the Atiyah Conjecture is preserved by sufficiently deep Dehn fillings (see \Cref{rm. why finite-index}).
\end{remark}

In several applications, we compute the $L^2$-Betti numbers of Dehn fillings associated with cusped arithmetic hyperbolic manifolds. While \Cref{thm. intro rel hyp} is a key input, it is itself insufficient for this purpose, as it is still open whether non-uniform arithmetic hyperbolic lattices are virtually compact cubical in general. To address this, we extend \Cref{thm. intro rel hyp} to include manifolds with virtually special sparse fundamental groups (see \Cref{thm. sparse rel. hyp.}). Importantly, Bergeron and Wise \cite{bergeron2012boundary} showed that all cusped arithmetic hyperbolic manifolds have virtually sparse special fundamental groups (see \Cref{thm. lattice sparse}). Furthermore, we note that any virtually compact cubical group that is hyperbolic relative to virtually abelian subgroups is itself virtually sparse special (see \Cref{prop. properties of vir special groups} and \Cref{rem:sparse compact}).

\subsection*{Applications.}  We now present some applications of the above theorems.  The first application concerns the algebraic fibering of groups, which is intimately related to $L^2$-Betti numbers \cite{kielak2020RFRS,fisher2021improved}. Recall that a group $G$ is said to \textit{$FP_k(\Q)$-fiber} if there is an epimorphism $G\twoheadrightarrow\Z$ whose kernel is of type $FP_k(\Q)$. In particular, $FP_1(\Q)$-fibering means that the kernel is finitely generated.

This notion of algebraic fibering originally emerged in the study of 3-manifolds that fiber over the circle. More recently, Kielak \cite{kielak2020RFRS} and Fisher \cite{fisher2021improved} characterized the virtual algebraic fibering of residually finite rationally solvable (RFRS) groups in terms of the vanishing of the first $L^2$-Betti number. Their results offer an algebraic counterpart to Agol’s theorem, which was instrumental in proving Thurston’s Virtual Fibering Conjecture. 

\begin{corollary}[\Cref{cor. 1 fiber}]\label{cor. intro 1 fiber}
    Let $G$ be a virtually compact cubical hyperbolic group that virtually $FP_1(\Q)$-fibers, and suppose that $G$ is hyperbolic relative to its subgroup $H$. Then there exists a finite-index torsion-free normal subgroup $K\lhd H$ such that for every sufficiently deep normal subgroup $N\lhd H$ that satisfies $N\leqslant K$ and $K/N$ is torsion-free virtually compact cubical hyperbolic, we have that the group $G/\ll N \rr$ virtually $FP_1(\Q)$-fibers.
\end{corollary}

\begin{corollary}[\Cref{cor. k fiber}]\label{intro. cor. k fiber}
    Let $G$ be a virtually compact cubical group that is hyperbolic relative to a virtually abelian subgroup $H\leqslant G$, and let $k\in\mathbb N^+$. Then there exists a finite-index torsion-free normal subgroup $K\lhd H$ such that for every sufficiently deep normal subgroup $N\lhd H$ such that 
    
    \begin{itemize}
        \item $N\leqslant K$, and
        \item $K/N\cong \Z$,
    \end{itemize}
    we have that the group $G/\ll N \rr$ virtually $FP_k(\Q)$-fibers if and only if so does $G$.
\end{corollary}

Our next application concerns the Singer Conjecture for Einstein manifolds. Note that Thurston's Hyperbolic Dehn Filling Theorem only holds in dimension $3$. This is because if $M$ is a manifold with torus boundary and $\dim M\geqslant 4$, then for any Dehn filling $\overline M$ of $M$, the fundamental group $\pi_1(\overline M)$ contains a $\Z^2$ subgroup, and thus $\overline M$ cannot admit any hyperbolic Riemannian metric. Nonetheless, Anderson \cite{anderson2006dehn} proved that most Dehn fillings of $M$ still admit Einstein metrics.

\begin{corollary}[\Cref{cor. constructing Einstein}] 
For every cusped arithmetic hyperbolic manifold $M$ of dimension at least four, there exists a sequence of pairwise non-homeomorphic Anderson-type closed Einstein manifolds $\{W_n\}_{n\geqslant 1}$ such that each $W_n$ satisfies the Singer Conjecture and $\{W_n\}_{n\geqslant 1}$ converges to a finite-sheeted cover $W$ of $M$ under the pointed Gromov--Hausdorff metric with based point any $y\in W$.
\end{corollary}
Recall that the Singer Conjecture predicts that the $L^2$-Betti numbers of the fundamental group of any closed aspherical manifold are concentrated in the middle dimension. To the best of our knowledge, the above corollary gives the first examples of Anderson-type Einstein manifolds that satisfy the Singer Conjecture.

In recent years, there has been significant progress in understanding the finiteness properties of subgroups of hyperbolic groups. Italiano--Martelli--Migliorini, Isenrich--Martelli--Py, and Isenrich--Py in \cite{IMM23, IMBP24, IP23} have provided new examples of subgroups of hyperbolic groups with unusual finiteness properties, offering counterexamples to long-standing conjectures. By applying \Cref{cor. k fiber}, we present a new construction of hyperbolic groups with exotic subgroups. An advantage of this approach is that every cusped arithmetic hyperbolic manifold of dimension at least four gives rise to infinitely many such examples.

\begin{theorem}[\Cref{thm. exotic quotient}] \label{intro: thm. exotic quotient} 
Let $G$ be the fundamental group of a cusped arithmetic hyperbolic manifold of dimension $n$.  Then there is an infinite sequence $\{\bg_k\}^\infty_{k=1}$ of pairwise non-isomorphic hyperbolic Dehn filling quotients of $G$, each of which has the same $L^2$-Betti numbers as $G$. In particular,
\begin{itemize}
  \item[(i)] If $n=2q\geqslant 4$,  then for each $k$, there exist a finite-index subgroup  $\bg'_k\leqslant \bg_k$ and an epimorphism $\psi:\bg'_k\rightarrow \Z$ such that $\ker \psi$ is of type $FP_{q-1}(\Q)$ but not of type $FP_{q}(\Q)$.\\
  \item[(ii)] If $n=2q+1\geqslant 5$, then there exist a finite-index subgroup $\bg'_k\leqslant \bg_k$ and an epimorphism $\psi:\bg'_k\rightarrow \Z$ such that $\ker \psi$ is of type $FP(\Q)$ but not hyperbolic.
\end{itemize}
\end{theorem}

In \cite{IMP24}, Isenrich, Martelli and Py constructed the first examples of hyperbolic groups which contain subgroups that are of type $F_3$ but are not of type $F_4$. These groups are obtained by Dehn filling a certain non-uniform lattice in $PO(8,1)$. They state that an alternative approach would be to compute the $L^2$-Betti numbers of Dehn fillings and pose the following:

\begin{question}[Isenrich--Martelli--Py, \cite{IMP24}]\label{ques. IMP} Let $\Gamma < PO(2k, 1)$ be a torsion-free non-uniform lattice with purely unipotent
parabolic subgroups. Is $b_k^{(2)}(\overline{\Lambda})\ne 0$ for a Dehn filling $\overline{\Lambda}$ of a deep enough  finite-index subgroup $\Lambda < \Gamma$?
\end{question}

We answer this question in the affirmative for arithmetic lattices and sufficiently deep fillings by establishing:

\begin{corollary}[\Cref{thm. lattice summarize} \ref{item. lattice li}]
    Let $\Gamma< PO(2k,1)$ be a non-uniform arithmetic lattice. Then for every $\delta>0$, all sufficiently deep fillings $\overline\Gamma$ of $\Gamma$ satisfy
    \[\left|b^{(2)}_k(\overline\Gamma)-|\chi(\Gamma)|\right|<\delta.\]
In particular,  if  $\delta$ is sufficiently small, then $b^{(2)}_k(\overline\Gamma)\ne 0$.
\end{corollary}

Finally, we use $L^2$-Betti numbers to give an upper bound for deficiency of groups. For a finitely presented group $G$, given by a finite presentation, the \textit{deficiency} of the presentation is the number of generators minus the number of relators, and the \textit{deficiency} of $G$, denoted as $\defi(G)$, is the maximum of the deficiency of all finite presentations of $G$. 
\begin{corollary}[\Cref{cor. defi lattice}]\label{intro.cor. defi}
    Let $G$ be a non-uniform arithmetic lattice in $SO(n,1)$, $n\geq 3$. Denote by $\{H_i\}_{i=1}^k$  the collection of peripheral subgroups. Then for every family of sufficiently deep normal subgroups  $\{N_i\lhd H_i\}_{i=1}^k$, we have
    $$\defi(\bg)\leqslant 1.$$
    In addition, if $G< SO(4,1)$, then 
    $$\defi(\bg)\leqslant 1-\frac{3}{4\pi^2}\mbox{vol}(\mathbb H^4/G).$$
\end{corollary}

We note that the bounds obtained in \Cref{intro.cor. defi} do not depend on the chosen Dehn filling and coincide with the known bounds for the deficiency of the lattice $G$. Lubotzky in \cite[Proposition 6.2]{Lubotzky83} showed that any lattice in $SO(n,1)$, $n\geq 3$ has deficiency at most one. Lott in \cite{Lott99} improved Lubotzky's result and in particular showed that any lattice in $SO(4,1)$ has a deficiency bound as above. 

\subsection*{Classifying space and finiteness properties.}
A key step in our computation of $L^2$-Betti numbers involves constructing a tailored classifying space.

\begin{theorem}[\Cref{cor. topology} (see also \Cref{cor. gen topology})]\label{thm. intro topology}
    Let $G$ be a group with a hyperbolically embedded subgroup $H$. Let $\phi\colon BH\rightarrow BG$ be a cellular map induced by the inclusion $H\hookrightarrow G$. For every normal subgroup $N\lhd H$, let $\psi_N\colon BH\rightarrow B\bh$ be a cellular map induced by $H\twoheadrightarrow \bh$. Let $X_N$ be the CW-complex obtained by gluing the mapping cylinders $M_\phi$ and $M_{\psi_N}$ along their common subcomplex $BH$. Then for sufficiently deep $N\lhd H$, the complex $X_N$ is a $K(\bg,1)$-space.
\end{theorem}

In addition to the computation of $L^2$-Betti numbers, the above theorem allows us to deduce geometric finiteness properties of Dehn fillings, which are analogs of the algebraic finiteness properties obtained by the authors in \cite[Corollary 2.2 (iii)]{sun2019cohomologyii}.

\begin{corollary}[\Cref{cor. finiteness conditions} (see also \Cref{cor. gen finiteness conditions})]
Let $G$ be a group with a hyperbolically embedded subgroup $H\hookrightarrow_h G$.
\begin{enumerate}
    \item[(i)] If $G$ is of type $F_n$ for some $n\in \mathbb N^{+}\cup \{\infty\}$, then for all sufficiently deep normal  $N\lhd H$, we have that $\bg$ is of type $F_n$  if and only if $\bh$ is of type $F_n$.\\

    \item[(ii)] If both $G$ and $H$ are of type $F$, then for all sufficiently deep normal  $N\lhd H$ such that $\bh$ is of type $F$, we have that $\bg$ is of type $F$.
\end{enumerate}
\end{corollary}

\begin{remark}  Analogs of other results of \cite{sun2019cohomologyii} with geometric finiteness properties in place of algebraic finiteness properties can be deduced using the above corollary. Since the proofs of those results are almost identical with their counterparts in \cite{sun2019cohomologyii}, we will not include them in this paper.
\end{remark}

\subsection*{Our approach on computing $L^2$-Betti numbers.} For any group $G$ and a $G$-CW-complex $X$, we denote the $L^2$-Betti numbers associated with the action of $G$ on $X$ by $b^{(2)}_\ast(X;G)$. Let $EG$ be the universal cover of a $K(G,1)$-CW-complex. A special case of L\"uck's Approximation Conjecture for $L^2$-Betti numbers, proved by Jaikin-Zapirain and L\'opez-\'Alvarez
in \cite{jaikin2020strong}, asserts that, if $\{M_i\}_{i\geqslant 1}$ is a nested sequence of normal subgroups of $G$ with trivial intersection, then under certain conditions, one obtains
\[b^{(2)}_\ast(EG ;G)=\lim_{i\rightarrow \infty}b^{(2)}_\ast (M_i\backslash EG; G/M_i).\]
Note that the above result deals with a sequence of normal subgroups rather than a single one. To apply this in the Dehn filling setting, we first slightly reformulate their approximation result by showing that, for every $\delta>0$, there exists a finite subset $\fS_\delta\subset G\smallsetminus\{1\}$ such that if a normal subgroup $M\lhd G$ satisfies $M\cap \fS_\delta=\emptyset$, then 

\[|b^{(2)}_\ast(M\backslash EG ; G/M)-b^{(2)}_\ast(EG ;G)|<\delta.\]
Two problems still remain. First, the above only gives an estimate rather than a computation of $L^2$-Betti numbers. Secondly, $b^{(2)}_\ast(M\backslash EG ; G/M)$ may not equal $b^{(2)}_\ast(G/M)$ as $M\backslash EG$ is not contractible unless $M=\{1\}$. We use the properties of Dehn fillings to resolve these as follows:

\begin{enumerate}[label=(\roman*)]
    \item[1)] Assuming $G$ is virtually compact cubical, and is either hyperbolic or relatively hyperbolic with respect to abelian subgroups, we use different versions of the Malnormal Special Quotient Theorem, proved by Agol--Groves--Manning \cite{agol2016alternate} and Wise \cite{wise2021structure}, to ensure that $G/M$ satisfies the Atiyah Conjecture, which asserts that $b^{(2)}_\ast(M\backslash EG ; G/M)$ takes value in a discrete subgroup of $\Q$, and thus must be equal to $b^{(2)}_\ast(EG ;G)$ provided $\delta$ is sufficiently small.\\

    \item[2)]\label{item. problem 2} By using the Cohen--Lyndon property for sufficiently deep Dehn fillings of hyperbolically embedded subgroups of $G$, established by the second author in \cite{sun2018cohomologyi}, we construct a $K(G/M,1)$-space suitable for our applications. We refer to this construction as a {\it Dehn filling space}, which is summarized in Theorem \ref{thm. intro topology}. An excision argument on this space relates $b^{(2)}_\ast(M\backslash EG ; G/M)$ with $b^{(2)}_\ast(G/M)$.
\end{enumerate}

\noindent Finally, it is worth highlighting that the choice of finite subsets that the sufficiently deep normal subgroups must avoid in the main results depends not only on the Cohen--Lyndon property but also on the quantitative version of L\"{u}ck's Approximation Conjecture.

\subsection*{Ingredients in the proof of \Cref{thm. intro rel hyp} and its generalisation \Cref{thm. sparse rel. hyp.}:}
\begin{itemize}
    \item L\"uck's Approximation Conjecture for virtually locally indicable groups \cite{jaikin2020strong} and its quantitative version  \Cref{thm. matrix approximation ver 2}.
    \item Cohen--Lyndon property for relatively \cite{groves2016boundaries} and acylindrically \cite{sun2018cohomologyi} hyperbolic groups.
    \item Construction of the Dehn filling space for Dehn filling quotients, \Cref{thm. topology}.
    \item Geometric finiteness properties of Dehn filling quotients, \Cref{cor. finiteness conditions}.
    \item Inequalities for $L^2$-Betti numbers of Dehn filling quotients of acylindrically hyperbolic virtually locally indicable groups, \Cref{thm. li dehn filling}.
   \item  Embeddability of finitely generated special groups in finitely generated right-angled Artin groups, \cite[Theorem 1.1]{haglund2008special}.
    \item The Malnormal Special Quotient Theorem \cite{wise2021structure, agol2016alternate} and its relative version  \cite{wise2021structure}.
    \item A weaker version of the Atiyah Conjecture for certain virtually locally indicable groups, \Cref{lem. det and atiyah ver 2}. 
       \item The Atiyah Conjecture for virtually compact special groups, \cite{schreve2015l2}.
    \item Separability of abelian subgroups in virtually special groups, \cite[Corollary 6.8]{wise2021structure}.
    \item Structure of sparse virtually special cube complexes whose fundamental group is hyperbolic relative to virtually abelian groups, \cite[Lemma 7.51, Theorem 7.54]{wise2021structure}.
\end{itemize}

\subsection*{Structure of the paper.} We recall the necessary background and notation in \Cref{sec. prelim}, and then,  in \Cref{sec. approximation}, we derive a quantitative version of  the result in \cite{jaikin2020strong} on L\"uck's Approximation Conjecture. In \Cref{sec. df and l2 betti}, we  first prove \Cref{thm. intro topology}, and then use it to deduce \Cref{thm. intro li} and \Cref{thm. intro hyp}. The proof of \Cref{thm. intro rel hyp} proceeds in two steps. First, we prove \Cref{thm. rel hyp.}, which is a special case of \Cref{thm. intro rel hyp} where the peripheral subgroup is abelian. Then, in \Cref{sec. multiple peripherals}, we state the general versions of \Cref{thm. intro li}, \Cref{thm. intro hyp} and \Cref{thm. rel hyp.} for a finite family of peripheral subgroups. In \Cref{sec:sparse}, we consider virtually sparse special groups and use the general results of \Cref{sec. multiple peripherals} to prove a general version of \Cref{thm. intro rel hyp}. Finally, we deduce our applications in \Cref{sec. applications}.

\subsection*{Acknowledgments.} First, we thank Kevin Li, who asked us whether the algebraic excision formula obtained in \cite{sun2019cohomologyii} could be derived using a topological excision.  The pursuit of an answer to this question led us to the construction of the Dehn filling space, which became the catalyst for the remaining results of this paper.  We are indebted to Mark Hagen for explaining to us how virtually sparse special groups may be used to understand cubulations of hyperbolic arithmetic lattices. We are grateful to Robert Bell, Luca Di Cerbo,  Chris Hruska, Wolfgang L\"uck, Jason Manning, Ashot Minasyan and Alan Reid for their helpful conversations. Finally, we thank the referees for the very helpful comments and suggestions.

\section{Preliminaries}\label{sec. prelim}

\subsection{Finiteness properties}\label{sec:finiteness_properties}

Throughout this paper, all actions will be left actions. Let $G$ be a group. We use
the terminology {\it $G$-CW-complex} for a $G$-space $X$ as defined by L\"uck \cite[Definition 1.25]{luck2002l2}.
This is equivalent to the notion of an {\it admissible $G$-complex} in the sense of \cite[Chapter IX, \S 10]{brown1982cohomology}, characterised by the property that for any cell $\sigma$ of $X$, if an element of $G$ stabilises $\sigma$, then it must fix $\sigma$ pointwise.

The subspace $X_n$ is called the \textit{$n$-skeleton} of $X$. We say that $X$ is a \textit{free $G$-CW complex} if in addition the action $G\curvearrowright X$ is free. The $G$-CW complex $X$ is \textit{of type $F_n$} for some $n\in\mathbb N$ (resp. \textit{type $F_{\infty}$}) if $G\backslash X$ has finite $n$-skeleton (resp. has finitely many cells in each dimension).

Next, we define the necessary finiteness notions of groups and mention their relevant properties. For more details, we refer the reader to Chapter VIII of \cite{brown1982cohomology}.

Let $R$ denote a commutative ring with unit and $G$ denote a discrete group. The \textit{cohomological dimension of $G$ over $R$} can be defined by
$$\cd_R \, G=\sup\{n\in\mathbb{N}\mid  H^{n}(G,M)\neq 0\text{ for some } R G\text{-module }M\}.$$
When $R=\Z$, we omit it from the notation and denote the cohomological dimension by $\cd \, G$. A group $G$ is said to have finite {\it virtual cohomological dimension} equal to $n$, denoted $\vcd \, G$, if $G$ has a subgroup $H$ of finite index with $\cd \, H =n$. This is a well-defined notion by \cite[VIII.2.4]{brown1982cohomology}. A group $G$ is said to be of \textit{type} $FP_n(R)$ for some $n\in\mathbb{Z}_{\geqslant0}\cup\{\infty\}$ if there is a projective $RG$-resolution 
\[\cdot\cdot\cdot\rightarrow P_2\rightarrow P_1\rightarrow P_0\twoheadrightarrow R\]
 such that $P_i$ is finitely generated and projective as an $RG$-module for all $i\leqslant n$. A projective $RG$-resolution $P_*\twoheadrightarrow R$ is said to be of {\it length $n$}, if 
 $P_n\ne 0$ and $P_i=0$ for all $i> n$.
 A group $G$ is said to be of {\it type $FP(R)$}, if there is a projective $RG$-resolution of $R$ by finitely generated projective $RG$-modules of finite length. It is immediate from the definitions that if $G$ is of type $FP(R)$, then it is of type $FP_{\infty}(R)$. 

Denote $\cd_R \, G=n$. By \cite[VIII.2.1]{brown1982cohomology}, it follows that  $G$ is of type $FP(R)$ if and only if $n<\infty$  and $G$ is of type $FP_n(R)$. 

 There are analogous topological finiteness properties which we explain next.

 An Eilenberg-MacLane space or $K(G, 1)$-complex is an aspherical CW complex with fundamental group equal to $G$. A group $G$ is of {\it type} $F_n$ if there is a $K(G, 1)$-complex $X$ with finite $n$-skeleton and it is of {\it type}
$F$ if $X$ is finite. If $G$ is of type $F_n$ or $F$, then it is of type $FP_n(\Z)$ or $FP$, respectively. A group $G$ is finitely generated if and only if it is of type $F_1$ or $FP_1(R)$  \cite[1.2.1]{bieri_book}. A group $G$ is finitely presented if and only if it is of type $F_2$.

For the rest of this paper, by $EG$ we mean a contractible free $G$-CW complex, and by $BG$ we mean the quotient $G\backslash EG$. If we do not specify the choice of $EG$, then we simply choose an arbitrary contractible free $G$-CW complex as $EG$. Such a complex always exists for any group $G$ and is a terminal object in the homotopy category of free $G$-CW complexes. In particular, any two choices for $EG$ are $G$-homotopy equivalent \cite{luck05_survey}.  In certain situations we may also specify the choice so that $EG$ (resp. $BG$) has a certain property, say $\mathcal P$, in which case we will say that let $EG$ (resp. $BG$) be a contractible free $G$-CW complex (resp. $K(G,1)$-space) satisfying $\mathcal P$.

\subsection{Von Neumann rank of a matrix}
Let $G$ be a discrete group. Recall that $\ell^2(G)$ denotes the Hilbert space of square-summable complex-valued functions on $G$ and  there is a left $G$-action on $\ell^2(G)$, called the \textit{left regular representation}. The \textit{group von Neumann algebra} of $G$, denoted $\N(G)$, is the set of all bounded $G$-equivariant linear operators on $\ell^2(G)$.  The \textit{von Neumann dimension $\dim_G$} of $G$ is a function from the set of $\N(G)$-modules to $[0,\infty]$. We refer to \cite[Chapter 6]{luck2002l2} for the precise definition of $\dim_G$. For the purpose of this paper, we only collect some properties of $\dim_G$.

\begin{proposition}\label{prop. properties of dimension} Let $G$ be a group. The following properties hold.
\begin{enumerate}[label=(\roman*)]
    \item\label{item. finite dimension} If $M$ is a finitely generated $\N(G)$-module, then $\dim_G M<\infty$.

    \item\label{item. only depend on iso class} If $M$ and $N$ are isomorphic $\N(G)$-modules, then $\dim_G M=\dim_G N$.

    \item\label{item. additivity} If $0\rightarrow M\rightarrow N \rightarrow P\rightarrow 0$ is an exact sequence of $\N(G)$-modules, then 
    $$\dim_G N=\dim_G M+\dim_G P,$$
    where the laws of the summation of $[0,\infty)$ extends naturally to $[0,\infty]$.

    \item\label{item. cofinal} Let $\{M_i\mid i\in I\}$ be a cofinal system of $\N(G)$-submodules of an $\N(G)$-module $M$, i.e., $M=\bigcup_{i\in I}M_i$ and for every two indices $i,j\in I$ there exists an index $k\in I$ such that $M_i,M_j\subset M_k$. Then $\dim_G M=\sup\{\dim_G M_i\mid i\in I\}$. 
\end{enumerate}
    
\end{proposition}

Let $I,J$ be indexing sets. A \textit{tamed matrix} over $\C G$ of dimension $I\times J$ is a (possibly infinite) matrix $A=(a_{ij})_{i\in I,j\in J}$ such that for all $i,j$, we have $a_{ij}\in \C G$, and for all $i\in I$, there are only finitely many $j\in J$ such that $a_{ij}\neq 0$. If $|I|,|J|<\infty$, then we call $A$ a \textit{finite matrix}.

Consider the $\N(G)$-modules $(\N(G))^I\cong \bigoplus_{i\in I}\N(G)$ and $(\N(G))^J\cong \bigoplus_{j\in J}\N(G)$, whose elements are of the form $(x_i)_{i\in I}$ and $(y_j)_{j\in J}$, respectively, where $x_i,y_j\in \N(G)$ for all $i,j$ and $x_i=y_j=0$ for all but finitely many $i\in I,j\in J$.

The right-multiplication of a tamed matrix $A=(a_{ij})_{i\in I,j\in J}$ over $\C G$ induces a $\N(G)$-module homomorphism $r_A\colon (\N(G))^I\rightarrow (\N(G))^J$, i.e.,
\[r_A(x):=x\cdot A= \left(\sum_{i\in I}x_i\cdot a_{ij}\right)_{j\in J}.\]
The \textit{von Neumann rank} of $A$ is
\[\rk_G(A):=\dim_G \im (r_A)\in[0,\infty].\]
If $N\lhd G$ is a normal subgroup, then the quotient map $G\rightarrow G/N$ extends to a ring homomorphism $\C G\rightarrow \C[G/N]$, which, when applied to each entry of $A$, gives us a matrix $\epsilon_N(A)$. For simplicity, we will denote 
\[\rk_{G/N}(A):=\rk_{G/N}(\epsilon_N(A)).\]

\subsection{$L^2$-Betti numbers and (Strong) Atiyah Conjecture}

Let $G$ be a group, let $X$ be a $G$-CW complex, and let $C_\ast(X)$ be the cellular $\Z G$-chain complex of $X$, with boundary map $\partial_\ast$. Suppose $M$ is a $\Z G$-module. We can thus form the tensor product chain complex $M\otimes_{\Z G}C_\ast(X),$
whose boundary map is $\widetilde\partial_\ast=\mathrm{id}_{M}\otimes \partial_\ast$. The homology of this chain complex is denoted by $H^G_\ast(X ; M)$. Define
\[H_\ast(G ;M):=H^G_\ast(EG ;M).\]
If $Y\subset X$ is a $G$-CW subcomplex, then we can form a relative chain complex $M\otimes_{\Z G}C_\ast(X,Y)$,
where $C_{\ast}(X,Y)$ is the relative cellular chain complex of the pair $(X,Y)$. The homology of this chain complex will be denoted by $H^G_\ast(X,Y ;M)$. For any subgroup $H\leqslant G$, the inclusion $H\hookrightarrow G$ induces a classifying map $\phi\colon G\times_H EH \rightarrow EG$.
Let $M_\phi$ be the mapping cylinder of $\phi$. Define
\[H_\ast(G,H ; M):=H^G_n(M_\phi,G\times_H EH ; M).\]
This definition is equivalent to the original definition of relative group homology given by Bieri--Eckmann \cite{bieri1978relative}. The cohomology groups $H_G^\ast(X ;M)$ and $H_G^\ast(X,Y ;M)$ are defined similarly where instead of tensoring with $M\otimes_{\Z G}-$, one applies $\mathrm{Hom}_{\Z G}(-, M)$ to the corresponding chain complexes.

There is a natural embedding $\Z G\hookrightarrow \N(G)$, and thus the group von Neumann algebra $\N(G)$ becomes a right $\Z G$-module. If we let $M=\N(G)$, then the homology $H^G_\ast(X\; ;\N(G))$,
is naturally a left $\N(G)$-module. In this case, to emphasize the $\N(G)$-module structure,  we denote 
$$\cH_\ast(X ;\N(G)):=H^G_\ast(X ;\N(G)),$$
$$\cH_\ast(G ;\N(G)):=H_\ast(G ;\N(G)).$$
We denote
\[b^{(2)}_\ast(X ;G):=\dim_G \cH_\ast(X ;\N(G)),~~b^{(2)}_\ast(G):=\dim_G \cH_\ast(G ;\N(G))\]
and call them the \textit{$L^2$-Betti numbers} of $X$ and $G$, respectively.
Similarly, denote
\[b^{(2)}_\ast(X,Y ;G):=\dim_G \cH_\ast(X,Y ;\N(G)),~~b^{(2)}_\ast(G,H):=\dim_G \cH_\ast(G,H ;\N(G)).\]
\begin{remark} For any two choices of $EG$, there exists a natural $\N(G)$-module isomorphism between the corresponding $\cH_\ast(G ;\N(G))$, and for any two choices of $EG$ and any two choices of $EH$, there exists a natural $\N(G)$-module isomorphism between the corresponding $\cH_n(G,H ;\N(G))$.
\end{remark}

The next lemma follows easily from \cite[Theorem 6.54 (6)]{luck2002l2}. 
\begin{lemma}\label{lem. zero after vcd}
    Let $G$ be a group. Then for all $n\geqslant \vcd(G)+1$, we have $b^{(2)}_n(G)=0$.
\end{lemma}

For any group $G$, we will denote
\[\fin(G):=\{n\in\mathbb N^+\mid \exists \text{ subgroup } H\leqslant G \text{ with } |H|=n\}.\]
Let
\[\dfrac{1}{\fin(G)}\cdot \Z\]
be the subgroup of the additive group of $\Q$ generated by elements in the set
\[\left\{\dfrac{1}{n} \; \Big| \; n\in \fin(G)\right\}.\]
\begin{conjecture}[(Strong) Atiyah Conjecture]
For every free $G$-CW complex $X$ of type $F_{k+1}$ for some $k\in\mathbb N$ and every $n\leqslant k$, 
\[b^{(2)}_n(X ;G)\in \frac{1}{\fin(G)}\cdot \Z.\]    
\end{conjecture}

 We will exploit the specific cases for which the Atiyah conjecture  is already known to support some of our main arguments.

\subsection{Hyperbolically embedded subgroups}

Let $G$ be a group, $\{H_\lambda\}_{\lambda\in\Lambda}$ a family of subgroups of $G$, $X$ a subset of $G$ such that $G$ is generated by $X$ together with the union of all $H_{\lambda}$ (in which case $X$ is called a \textit{relative generating set} of $G$ with respect to $\{H_{\lambda}\}_{\lambda\in\Lambda}$), and $\mathcal{H}=\bigsqcup_{\lambda\in\Lambda}H_{\lambda}$. Consider the Cayley graph $\Gamma(G,X\sqcup \mathcal{H})$. Note that, for $\lambda\in\Lambda$ there is a natural embedding $\Gamma(H_{\lambda},H_{\lambda})\hookrightarrow \Gamma(G,X\sqcup\mathcal{H})$ whose image is the subgraph of $\Gamma(G,X\sqcup\mathcal{H})$ with vertices and edges labeled by elements of $H_{\lambda}$.

\begin{remark}
For distinct $\lambda,\mu\in\Lambda$, we allow $X\cap H_{\lambda}\neq \emptyset$ and $H_{\lambda}\cap H_{\mu}\neq\{1\}$, in which case there will be multiple edges between some pairs of vertices of $\Gamma(G,X\sqcup \mathcal{H})$.
\end{remark}

For $\lambda\in\Lambda$, an edge path in $\Gamma(G,X\sqcup \mathcal{H})$ between vertices of $\Gamma(H_{\lambda},H_{\lambda})$ is called $H_{\lambda}$\textit{-admissible} if it does not contain any edge of $\Gamma(H_{\lambda},H_{\lambda})$. Note that an $H_{\lambda}$-admissible path is allowed to pass through vertices of $\Gamma(H_{\lambda},H_{\lambda})$.

\begin{definition}\label{def. relative metric}
For every pair of elements $h,k\in H_{\lambda}$, let $\widehat{d}_{\lambda}(h,k)\in[0,\infty]$ be the length of a shortest $H_{\lambda}$-admissible path connecting the vertices labeled by $h$ and $k$. If no such path exists, set $\widehat{d}_{\lambda}(h,k)=\infty$. The laws of summation on $[0,\infty)$ extend naturally to $[0,\infty]$ and it is easy to verify that $\widehat{d}_{\lambda}:H_{\lambda}\times H_{\lambda}\rightarrow [0,+\infty]$ defines a metric on $H_{\lambda}$, which is called the \textit{relative metric on} $H_{\lambda}$ \textit{with respect to} $X$.
\end{definition}

\begin{definition}\label{hyperbolically embedded}
We say that the family $\{H_\lambda\}_{\lambda\in\Lambda}$ \textit{hyperbolically embeds into} $(G,X)$ (denoted by $\{H_\lambda\}_{\lambda\in\Lambda}\hookrightarrow_h(G,X)$) if the following hold:

\begin{enumerate}
    \item[(i)] $G$ is generated by the set $X$ together with the union of all $H_{\lambda},\lambda\in\Lambda$.

    \item[(ii)] The Cayley graph $\Gamma(G,X\sqcup \mathcal{H})$ is a Gromov hyperbolic space.

    \item[(iii)] For each $\lambda\in\Lambda$, the metric space $(H_{\lambda},\widehat{d}_{\lambda})$ is proper, i.e., every ball of finite radius contains only finitely many elements
\end{enumerate}

Further, we say that the family $\{H_\lambda\}_{\lambda\in\Lambda}$ \textit{hyperbolically embeds into} $G$, denoted by $\{H_\lambda\}_{\lambda\in\Lambda}\hookrightarrow_h G$, if there exists some subset $X\subset G$ such that $\{H_\lambda\}_{\lambda\in\Lambda}\hookrightarrow_h(G,X)$. We say that \textit{$G$ is hyperbolic relative to the family $\{H_\lambda\}_{\lambda\in\Lambda}$} or the family $\{H_{\lambda}\}_{\lambda\in\Lambda}$ is a \textit{peripheral structure} on $G$ if $|\Lambda|<\infty$ and there exists a finite set $X\subset G$ such that  $\{H_\lambda\}_{\lambda\in\Lambda}\hookrightarrow_h(G,X)$. If $\{H_\lambda\}_{\lambda\in\Lambda}$ consists of a single element $H$ we will drop the braces and write that $H\hookrightarrow_h (G,X)$ (resp. $H\hookrightarrow_h G$, $G$ is hyperbolic relative to $H$) instead of $\{H\}\hookrightarrow_h (G,X)$ (resp. $\{H\}\hookrightarrow_h G$, $G$ is hyperbolic relative to $\{H\}$).
\end{definition}

\begin{remark}
    Our definition of relative hyperbolicity is not one of the standard definitions, but is equivalent to those by \cite[Proposition 4.28]{dahmani2017hyperbolically}.
\end{remark}

\subsection{Dehn filling}

Dehn filling is a process which takes input a group $G$, a family of subgroups $\{H_\lambda\}_{\lambda\in\Lambda}$ of $G$, and a family of normal subgroups $\{N_\lambda\in H_\lambda\}_{\lambda\in\Lambda}$, and outputs the quotient group $G/\ll \bigcup_{\lambda\in\Lambda} N \rr_G$. Here and in the sequel, for any subset $S$ of a group $G$, we denote by $\ll S \rr_G$ the normal closure of $S$ in $G$. If $G$ is understood, we will simply write $\ll S \rr$.

In practice, it is often assumed that $G$ and $\{H_\lambda\}_{\lambda\in\Lambda}$ satisfy certain conditions, so that for most choices of $\{N_\lambda\lhd H_\lambda\}_{\lambda\in\Lambda}$, the quotient $G/\ll \bigcup_{\lambda\in\Lambda} N \rr_G$ satisfy certain properties. In this paper, we always assume that $\{H_\lambda\}_{\lambda\in\Lambda}$ hyperbolically embeds into $G$, in which case the Cohen--Lyndon property arises naturally. To simplify the statements of our results, we use the following terminology, which aligns with the terminology of $3$-manifold theory.

\begin{definition}\label{def. sufficiently deep}

Let $G\geqslant H$ be groups, and let $\mathcal P= \mathcal P(G,H,N),\mathcal Q=\mathcal Q(G,H,N)$ be properties formulated for all normal subgroups $N\lhd H$. We say that $\mathcal P$ \textit{holds for sufficiently deep normal subgroups $N\lhd H$} if there exists a finite subset $\F\subset H\smallsetminus\{1\}$ such that $\mathcal P$ holds whenever $N\cap \F=\emptyset$. We say that $\mathcal P$ \textit{holds for sufficiently deep normal subgroups $N\lhd H$ such that $\mathcal Q$ holds} if there exists a finite subset $\F\subset H\smallsetminus\{1\}$ such that, if $N\cap \F=\emptyset$ and $\mathcal Q$ holds, then $\mathcal P$ holds.    
\end{definition}

\begin{remark}\label{rm. sufficiently deep for finite groups}
    Let $G\geqslant H$ be groups. Note that, if $|H|<\infty$ then a property $\mathcal P$ holds for sufficiently deep normal subgroups $N\lhd H$ if and only if it holds for $N=\{1\}$.
\end{remark}

There are two key ingredients in our study of Dehn filling. The first one is the following result which roughly says that hyperbolicity is preserved by sufficiently deep Dehn filling:

\begin{theorem}[{Dahmani--Guirardel--Osin \cite{dahmani2017hyperbolically}}]\label{thm. simple Dehn filling}
Let $G$ be a group with a subgroup $H\hookrightarrow_h G$. Then for sufficiently deep $N\lhd H$, the natural homomorphism $H/N\rightarrow G/\ll N \rr $ maps $H/N$ injectively onto a hyperbolically embedded subgroup of $G/\ll N \rr$.
\end{theorem}

The second is a structural result about the kernel associated with a sufficiently deep Dehn filling, termed as a \textit{Cohen--Lyndon triple}, which was introduced by \cite{sun2018cohomologyi}, but the idea dates back to a paper by Cohen--Lyndon \cite{cohen1963free}, hence the name. This property is crucial for our construction of a classifying space of the quotient group arising from a Dehn filling process.

\begin{definition}
    Let $N\lhd H \leqslant G$ be a triple of groups. We say that the triple $(G,H,N)$ is a \textit{Cohen--Lyndon triple} if there exists a left transversal $T$ of $H\ll N \rr$ in $G$ such that $\ll N \rr$ decomposes as a free product of the following form $\ll N \rr = \Asterisk_{t\in T} tNt^{-1}$.
\end{definition}

\begin{theorem}~\cite[Theorem~2.5]{sun2018cohomologyi}\label{thm. cl}
Let $G$ be a group with a hyperbolically embedded subgroup $H\hookrightarrow_h G$. Then for all sufficiently deep normal subgroups $N\lhd H$, the triple $(G,H,N)$ is a Cohen--Lyndon triple.    
\end{theorem}

In the subsequent sections, we will often need to pass to a finite-index normal subgroup of a (sufficiently deep) Dehn filling quotient of a group $G$. \Cref{lem. df in fi subgroup} below roughly states that such a finite-index normal subgroup arises as a (sufficiently deep) \textit{$G$-equivariant Dehn filling} of a finite-index normal subgroup of $G$. Let us start with the related definitions.

\begin{definition}
    Let $G$ be a group that is hyperbolic relative to a family of subgroups $\{H_i\}^k_{i=1}$, let $G_0\lhd G$ be a finite-index normal subgroup. For all $i$, let $H_{0i}=G_0\cap H_i$, and let $\h_i$ be the family of subgroups of $G_0$ obtained by choosing one element from each $G_0$-conjugacy class of the set $\{gH_{0i}g^{-1}\mid g\in G\}$. Let $\h=\bigcup^k_{i=1} \h_i$.
Then $\h$ is a finite family of subgroups of $G_0$ and $G_0$ is hyperbolic relative to $\h$ by \cite[Theorem 9.1]{hruska2010relative}. We call $\h$ the \textit{peripheral structure induced by} $\{H_i\}^k_{i=1}$.
\end{definition}

\begin{definition}\label{def. equivariant filling}
Let $G$ be a group hyperbolic relative to a family of subgroups, and let $G_0\lhd G$ be a finite-index normal subgroup with the induced peripheral structure $\h$. We say that a family of normal subgroups $\{N_H\lhd H\mid H\in\h\}$ is \textit{$G$-equivariant} or the Dehn filling of $G_0$ induced by this family is \textit{$G$-equivariant} if, for all $H\in \h$ and $g\in G$ such that $gHg^{-1}\in \h$, we have $N_{gHg^{-1}}=gN_Hg^{-1}$.
\end{definition}

In the following lemma, as we will discuss normal closures in different groups, we will use $\ll\cdot\rr_G$ instead of $\ll\cdot\rr$.

\begin{lemma}\label{lem. df in fi subgroup}
    Let $G$ be a group that is hyperbolic relative to a family of infinite subgroups $\{H_i\}^k_{i=1}$, and let $G_0\lhd G$ be a finite-index normal subgroup, endowed with the induced peripheral structure $\{K_j\}^\ell_{j=1}$ from $G$. Then the following hold:
    
    \begin{enumerate}[label=(\roman*)]
        \item\label{item. bijection} There is a bijection $\varphi\colon \mathcal T\rightarrow \fS$, where $\fS$ is the set of all families of normal subgroups $\{N_i\lhd H_i\}^k_{i=1}$ such that $N_i\leqslant G_0\cap H_i$ for all $i$, and $\mathcal T$ is the set of all $G$-equivariant families of normal subgroups $\{M_j\lhd K_j\}^\ell_{j=1}$.

        \item\label{item. peripheral}
        For $k=1,2$, let $S_k=\{N_{k,i}\lhd H_i\}^k_{i=1}\in \fS$ such that $N_{1,i}\leqslant N_{2,i}$ for all $i$, and let $\varphi^{-1}(S)=\{M_{k,j}\lhd K_j\}^\ell_{j=1}$. Then for each $j\in\{1,\dots,\ell\}$, we have $M_{1,j}\leqslant M_{2,j}$ and there exists $i\in\{1,\dots,k\}$ and $g\in G$ such that
        \[gN_{k,i}g^{-1}=M_{k,j},~~~~ K_j/M_{k,j}\cong (G_0\cap H_i)/N_{k,i},~~~~ N_{2,i}/N_{1,i}\cong M_{2,j}/M_{1,j}.\]
        \item\label{item. embedding}
        For each $T\in\mathcal{T}$, the following diagram commutes
        \begin{center}
        \begin{tikzcd}
         G_0 \arrow[r] \arrow[d,hook]
         & G_0/\ll \bigcup_{M\in T} M \rr_{G_0} \arrow[d, hook] \\
         G \arrow[r]
         &  G/\ll \bigcup_{N\in \varphi(T)} N \rr_G
       \end{tikzcd}
       \end{center}
       and we have

      \begin{equation}\label{eq. index}
      [G/\ll \bigcup_{N\in \varphi(T)} N \rr_G:G_0/\ll \bigcup_{M\in T} M \rr_{G_0}]=[G:G_0].
      \end{equation}

      \item\label{item. deep}
      For each family of finite subsets $\{\F_i\subset H_i\smallsetminus\{1\}\}^k_{i=1}$, there exists a family of finite subsets $\{\F^{\prime}_j\subset K_j\smallsetminus\{1\}\}^\ell_{j=1}$ such that the following hold: let $S=\{N_i\lhd H_i\}^k_{i=1}\in \fS$, and let $\varphi^{-1}(S)=\{M_j\lhd K_j\}^\ell_{j=1}$. Then $M_j\cap \F^{\prime}_j=\emptyset$ for all $j$ implies $N_i\cap \F_i=\emptyset$ for all $i$.

      Conversely, for each family of finite subsets $\{\F^{\prime}_j\subset K_j\smallsetminus\{1\}\}^\ell_{j=1}$, there exists a family of finite subsets $\{\F_i\subset H_i\smallsetminus\{1\}\}^k_{i=1}$ such that the following hold:
      let $T=\{M_j\lhd K_j\}^\ell_{j=1}\in\mathcal{T}$ and let $\varphi(T)=\{N_i\lhd H_i\}^k_{i=1}$. Then $N_i\cap \F_i=\emptyset$ for all $i$ implies $M_j\cap \F^{\prime}_j=\emptyset$ for all $j$.

      \item\label{item. finite index} 
      For each family of finite-index subgroups $\{P_i\leqslant  H_i\}^k_{i=1}$, there exists a family of finite-index subgroups $\{Q_j\leqslant K_j\}^\ell_{j=1}$ such that the following hold: let $S=\{N_i\lhd H_i\}^k_{i=1}\in \fS$, and let $\varphi^{-1}(S)=\{M_j\lhd K_j\}^\ell_{j=1}$. Then $M_j\leqslant Q_j$ for all $j$ implies $N_i\leqslant P_i$ for all $i$.

      Conversely, for each family of finite-index subgroups $\{Q'_j\leqslant K_j\}$, there exists a family of finite-index subgroups $\{P'_i\leqslant H_i\}^k_{i=1}$ such that the following hold:
      let $T=\{M_j\lhd K_j\}^\ell_{j=1}\in\mathcal{T}$ and let $\varphi(T)=\{N_i\lhd H_i\}^k_{i=1}$. Then $N_i\leqslant P'_i$ for all $i$ implies $M_j\leqslant Q'_j$ for all $j$.
    \end{enumerate}
\end{lemma}

\begin{proof}

    For each $j\in\{1,\dots,\ell\}$, choose $\alpha(j)\in\{1,\dots,k\}$ and $g_j\in G$ such that $$K_j=g_j(G_0\cap H_{\alpha(j)})g^{-1}_j.$$
\noindent    For each $i\in\{1,\dots,k\}$, choose $\beta(i)\in\{1,\dots,\ell\}$ such that $\alpha\circ\beta(i)=i$.
    Therefore, $$K_{\beta(i)}=g_{\beta(i)}(G_0\cap H_i)g^{-1}_{\beta(i)}.$$

\noindent{\it (i).}  Let $\{M_j\lhd K_j\}^\ell_{j=1}$ be a $G$-equivariant family of normal subgroups. For each $i$, let
        \[N_i=g^{-1}_{\beta(i)} M_{\beta(i)} g_{\beta(i)}.\]
        Then $N_i\leqslant G_0\cap H_i$. We need to show that $N_i$ is normal in $H_i$. Let $h\in H_i$. As $G_0\lhd G$ we have $g^{-1}_{\beta(i)}K_{\beta(i)}g_{\beta(i)}=G_0\cap H_i\lhd H_i$ and thus $g_{\beta(i)}hg^{-1}_{\beta(i)}K_{\beta(i)}g_{\beta(i)}h^{-1}g^{-1}_{\beta(i)}=K_{\beta(i)}$. As the family $\{M_j\lhd K_j\}^\ell_{j=1}$ is $G$-equivariant, we have $g_{\beta(i)}hg^{-1}_{\beta(i)}M_{\beta(i)}g_{\beta(i)}h^{-1}g^{-1}_{\beta(i)}=M_{\beta(i)}$, which is equivalent to $hN_ih^{-1}=N_i$ as desired. 
        Next, define
        \[\varphi(\{M_j\lhd K_j\}^\ell_{j=1}):=\{N_i\}^k_{i=1}=\{g^{-1}_{\beta(i)} M_{\beta(i)} g_{\beta(i)}\}^k_{i=1}.\]
        We need to construct the inverse of $\varphi$. For each $i\in\{1,\dots,k\}$, let $N'_i\lhd H_i$ such that $N'_i\leqslant G_0\cap H_i$. For $j=1,\dots,\ell$, let    
        \[M'_j=g_j N'_{\alpha(j)} g^{-1}_j.\]
        We verify that $\{M'_j\}^\ell_{j=1}$ is a $G$-equivariant family. Let $j_1,j_2\in\{1,\dots,\ell\}$ such that there exists $g\in G$ with $gK_{j_1}g^{-1}=K_{j_2}$. We have $K_{j_1}\leqslant g_{j_1}H_{\alpha(j_1)}g^{-1}_{j_1}$, and thus $K_{j_2}=gK_{j_1}g^{-1}\leqslant g g_{j_1}H_{\alpha(j_1)}g^{-1}_{j_1} g^{-1}$. Therefore, $G_0\cap H_{\alpha(j_2)}= g^{-1}_{j_2}K_{j_2}g_{j_2}\leqslant g^{-1}_{j_2}g g_{j_1}H_{\alpha(j_1)}g^{-1}_{j_1} g^{-1}g_{j_2}$. Being a finite-index subgroup of the infinite group $H_{\alpha(j_2)}$, the group $G_0\cap H_{\alpha(j_2)}$ is infinite, and thus \cite[Proposition 4.33]{dahmani2017hyperbolically} implies that $\alpha(j_1)=\alpha(j_2)$ and $g^{-1}_{j_2}g g_{j_1}\in H_{\alpha(j_1)}$. Therefore,
\begin{align*}
    M'_{j_2}&= g_{j_2}N'_{\alpha(j_2)}g^{-1}_{j_2}=g_{j_2}N'_{\alpha(j_1)}g^{-1}_{j_2}=g_{j_2}(g^{-1}_{j_2}g g_{j_1})N'_{\alpha(j_1)}(g^{-1}_{j_2}g g_{j_1})^{-1}g^{-1}_{j_2}\\
    &=gg_{j_1}N'_{\alpha(j_1)}g^{-1}_{j_1}g^{-1}=g M'_{j_1} g^{-1},
\end{align*}
        as desired.\\ 
        
\noindent{\it (ii).}  For every $j\in\{1,\dots,\ell\}$, we have
        \begin{align*}
            &g_jN_{k,\alpha(j)}g^{-1}_j=M_{k,j},\\
            &M_{1,j}=g_j N_{1,\alpha(j)}g^{-1}_j\leqslant g_j N_{2,\alpha(j)}g^{-1}_j= M_{2,j},\\
            &M_{2,j}/M_{1,j}=(g_j N_{2,\alpha(j)}g^{-1}_j)/(g_j N_{1,\alpha(j)}g^{-1}_j)\cong N_{2,\alpha(j)}/N_{1,\alpha(j)},\\
            &K_j/M_{k,j}=(g_j(G_0\cap H_{\alpha(j)})g^{-1}_j)/(g_jN_{k,\alpha(j)}g^{-1}_j)\cong (G_0\cap H_{\alpha(j)})/N_{k,\alpha(j)}.
        \end{align*}
        
\noindent{\it (iii).}  Let $T=\{M_j\lhd K_j\}^\ell_{j=1}\in\mathcal{T}$ and let $\varphi(T)=\{N_i\lhd H_i\}^k_{i=1}$. For $j=1,\dots,\ell$, we have $M_j=g_jN_{\alpha(j)}g^{-1}_j$. In particular, we have $\ll \bigcup^\ell_{j=1} M_j \rr_{G_0}\leqslant \ll \bigcup^k_{i=1} N_i \rr_G$ and thus there is a natural homomorphism $\psi\colon G_0/\ll \bigcup^\ell_{j=1} M_j \rr_{G_0} \rightarrow G/\ll \bigcup^k_{i=1} N_i \rr_G$ that fits into a commutative diagram

    \begin{center}
    \begin{tikzcd}
     G_0 \arrow[r] \arrow[d,hook]
     & G_0/\ll \bigcup^\ell_{j=1} M_j \rr_{G_0} \arrow[d, "\psi"] \\
     G \arrow[r]
     &  G/\ll \bigcup^k_{i=1} N_i \rr_G
     \end{tikzcd}
     \end{center}

\noindent     To prove that $\psi$ is injective, it suffices to prove $\ll \bigcup^k_{i=1} N_i \rr_G\leqslant \ll \bigcup^\ell_{j=1} M_j \rr_{G_0}$. A general element $n\in \ll \bigcup^k_{i=1} N_i \rr_G$ has the form $n=\prod^q_{p=1}g'_p n_p (g'_p)^{-1}$. For each $p$, there exists $i(p)\in\{1,\dots,k\}$ such that $n_p\in N_{i(p)}$. We have
     \[N_{i(p)}=g^{-1}_{\beta(i(p))}M_{\beta(i(p))}g_{\beta(i(p))}.\]
\noindent       So,
     \[n\in\prod^q_{p=1} g'_p g^{-1}_{\beta(i(p))}M_{\beta(i(p))}g_{\beta(i(p))} (g'_p)^{-1}.\]
\noindent       By the definition of the induced peripheral structure, there exists $j(p)\in\{1,\dots,\ell\}$ and $g''_p\in G_0$ such that 
     \[g'_p g^{-1}_{\beta(i(p))}K_{\beta(i(p))}g_{\beta(i(p))} (g'_p)^{-1}= g''_p K_{j(p)}(g''_p)^{-1}.\]
\noindent       As $\{M_j\}^\ell_{j=1}$ is a $G$-equivariant family, we have
     \[g'_p g^{-1}_{\beta(i(p))}M_{\beta(i(p))}g_{\beta(i(p))} (g'_p)^{-1}= g''_p M_{j(p)}(g''_p)^{-1}.\]
\noindent       Therefore,
     \[n\in\prod^q_{p=1}g''_p M_{j(p)}(g''_p)^{-1}\subset \ll \bigcup^\ell_{j=1} M_j \rr_{G_0},\]
     which finishes the proof that $\psi$ is injective and also proves \eqref{eq. index}.\\

\noindent{\it (iv).}  Let $\{\F_i\subset H_i\smallsetminus\{1\}\}^k_{i=1}$ be a family of finite subsets. For $j\in\{1,\dots,\ell\}$, let 
     \[\F^{\prime}_j=g_j (\F_{\alpha(j)}\cap G_0) g^{-1}_j.\]
     Then the family $\{\F^{\prime}_j\}^\ell_{j=1}$ satisfies the requirement.

     Conversely, let $\{\F^{\prime}_j\subset K_j\smallsetminus\{1\}\}^\ell_{j=1}$ be a family of finite subsets. For $i\in\{1,\dots,k\}$, let
     \[\F_i=\bigcup_{j\in \alpha^{-1}(\{i\})} g^{-1}_j\F^{\prime}_j g_j.\]
     Then the family $\{\F_i\}^k_{i=1}$ satisfies the requirement.\\

\noindent{\it (v).}  Let $\{P_i\leqslant H_i\}^k_{i=1}$ be a family of finite-index subgroups. For $j\in\{1,\dots,\ell\}$, let   
     \[Q_j=g_j (P_{\alpha(j)}\cap G_0) g^{-1}_j.\]
     Then the family $\{Q_j\}^\ell_{j=1}$ satisfies the requirement.

     Conversely, let $\{Q'_j\leqslant K_j\}^\ell_{j=1}$ be a family of finite-index subgroups. For $i\in\{1,\dots,k\}$, let
     \[P'_i=\bigcap_{j\in \alpha^{-1}(\{i\})} g^{-1}_j Q_j g_j.\] 
     Then the family $\{P'_i\}^k_{i=1}$ satisfies the requirement.
\end{proof}

\section{Quantifying L\"{u}ck's approximation}\label{sec. approximation}

We will need an approximation result for $L^2$-Betti numbers for the space of marked groups, proved by Jaikin-Zapirain--L\'{o}pez-\'{A}lvarez \cite{jaikin2020strong}. 

Let $F$ be a group freely generated by a finite set $S$. The space of marked groups $\mg(F)$ can be identified with the set of normal subgroups of $F$ with the metric $d(M_1,M_2)=e^{-n}$, where $M_1\neq M_2\lhd F$ are two normal subgroups and $n$ is the largest number such that the balls of radius $n$ in the Cayley graphs $\Gamma(F/M_1,S)$ and $\Gamma(F/M_2,S)$ are isomorphic as labeled graphs. The metric $d$ induces a topology on $\mg(F)$. Below, the notion of convergence is discussed with respect to this topology. 

\begin{theorem}[Jaikin-Zapirain--L\'{o}pez-\'{A}lvarez, 2020]\label{thm. marked group luck}
    Let $F$ be a finitely generated free group and assume that a sequence $\{M_i\}_{i\geqslant 1}$ of normal subgroups of $F$ converges to a normal subgroup $M\lhd F$ in $\mg(F)$. Suppose that $F/M$ is virtually locally indicable. Then for every finite matrix $A$ over $\C F$, we have
    \[\lim_{i\rightarrow \infty} \rk_{F/M_i} (A)=\rk_{F/M} (A).\]
\end{theorem}

\noindent For applications to  Dehn fillings, we need to quantify the above theorem.

\begin{theorem}\label{thm. matrix approximation ver 2}
    Let $G$ be a finitely generated, virtually locally indicable group. Then for every finite matrix $A$ over $\C G$ and every $\delta>0$, there exists a finite subset $\F_\delta\subset G\smallsetminus\{1\}$ such that if a normal subgroup $N\lhd G$ satisfies $N\cap \F_\delta=\emptyset$, then 
    \[|\rk_{G/N}(A)-\rk_G(A)|<\delta.\]
\end{theorem}

\begin{proof}
    Suppose, to the contrary, that there exists a matrix $A$ over $\C G$ and $\delta_0>0$ such that for each finite subset $\F\subset G\smallsetminus\{1\}$, there exists a normal subgroup $N\lhd G$ such that $N\cap \F=\emptyset$ but
    \[|\rk_{G/N}(A)-\rk_G(A)|>\delta_0.\]

    Fix a finite generating set $S$ of $G$. Let $F$ be the free group on basis $S$, let $\epsilon\colon F\rightarrow G$ be the natural surjection, and let $M=\ker\epsilon$. Then $\epsilon$ can be naturally seen to be a map from the set of square matrices over $\C F$ to the set of square matrices over $\C G$. Let $B$ be a matrix over $\C F$ such that $A=\epsilon(B)$.

    For every $n\in \mathbb N^+$, let $B_n\subset \Gamma(G,S)$ be the subgraph consisting of all points within distance $n$ from the vertex $1$, and let $\F_n=G\cap B_n\smallsetminus\{1\}$.
    By assumption, there exists a normal subgroup $N_n\lhd G$ such that $N_n\cap \F_n=\emptyset$ and
    \begin{equation}\label{eq. rank 1}
        |\rk_{G/N_n}(A)-\rk_G(A)|>\delta_0.
    \end{equation}

    Let $M_n\lhd F$ be the preimage of $N_n$ under $\epsilon$. We claim that the sequence of normal subgroups $\{M_n\}_{n\geqslant 1}$ converges to $M$ in $\mg(F)$. Indeed, for every $R\in\mathbb N^+$, let $n>2R$. Consider the subgraph $\overline B_R\subset\Gamma(F/M_n,S)=\Gamma(G/N_n,S)$ consisting of all points within distance $R$ from $1$, and the subgraph $B_R\subset \Gamma(F/M,S)=\Gamma(G,S)$. Let $ g\neq  h$ be two vertices of $ B_R$, where we think of $ g, h$ as elements of $G$, and let $\bar g,\bar h$ be their natural images in $G/N_n$. Then $\bar g, \bar h$ are vertices in $\overline B_R$. Note that $g$ and $h$ are a distance less than $2R$ apart from each other. As $n>2R$, we have $g^{-1}h\not\in N_n$, and thus $\bar g\neq \bar h$. Therefore, the natural map from $B_R$ to $\overline B_R$ is injective. As this natural map is obviously surjective, the labeled graphs $ B_R$ and $\overline B_R$ are isomorphic, whenever $n>2R$. 

    By \Cref{thm. marked group luck}, we have
    \begin{equation}\label{eq. rank 2}
        \lim_{n\rightarrow \infty}\rk_{F/M_n}(B)=\rk_{F/M}(B).
    \end{equation}

    On the other hand, we have $\rk_{F/M_n}(B)=\rk_{G/N_n}(A)$ and $\rk_{F/M}(B)=\rk_G(A)$. Thus, \eqref{eq. rank 1} and \eqref{eq. rank 2} contradict each other.
\end{proof}

In the above theorem, if $A$ is not finite but just a tamed matrix, one can still obtain a lower bound on $\rk_{G/N}(A)$.

\begin{corollary}\label{lem. infinite matrix approximation}
    Let $G$ be a finitely generated, virtually locally indicable group. Then for each tamed matrix $A$ over $\C G$ such that $\rk_G(A)<\infty$ and each $\delta>0$, there exists a finite subset $\F_\delta\subset G\smallsetminus\{1\}$ such that if a normal subgroup $N\lhd G$ satisfies $N\cap \F_\delta=\emptyset$, then 
    \[\rk_{G/N}(A)>\rk_G(A)-\delta.\]
\end{corollary}

\begin{proof}
    Write $A$ as $A=(a_{ij})_{i\in I,j\in J}$ and let $r_A\colon (\N(G))^I\rightarrow (\N(G))^J$ be the $\N(G)$-module homomorphism given by the right-multiplication of $A$. By definition,
    \begin{equation}\label{eq. definition of rank}
        \rk_G(A)=\dim_G \im(r_A).
    \end{equation}
    For each $i\in I$, let $e_i=(x_{i'})_{i'\in I}\in (\N(G))^I$ be the element such that
    \[
    x_{i'}=
    \begin{cases}
        1,&\text{if } i'=i\\
        0,&\text{if } i'\neq i.
    \end{cases}
    \]
    Let $\delta>0$, and let $\fS$ be the family of non-empty finite subsets of $I$. For every $S\in\fS$, let $M_S$ be the $\N(G)$-submodule of $(\N(G))^I$ generated by $e_i,i\in S$. Then $\{r_A(M_S)\mid S\in\fS\}$ is a cofinal family of $\im(r_A)$, and thus
    \[\dim_G\im(r_A)=\sup\{\dim_G r_A(M_S)\mid S\in \fS\}\]
    by \Cref{prop. properties of dimension} \ref{item. cofinal}. So there exists $S\in \fS$ such that
    \begin{equation}\label{eq. connect to a finite submodule}
        \dim_G\im(r_A)<\dim_G r_A(M_S)+\dfrac{\delta}{2}.
    \end{equation}
    For each $i\in S$, there are only finitely many $j\in J$ such that $a_{ij}\neq 0$. Let $T\subset J$ be the subset consisting of $j\in J$ such that there exists $i\in S$ with $a_{ij}\neq 0$. As $|S|<\infty$, we also have $|T|<\infty$. Let $B=(a_{ij})_{i\in S,j\in T}$, and let $r_B\colon (\N(G))^S\rightarrow (\N(G))^T$ be the $\N(G)$-module homomorphism induced by the right-multiplication of $B$. Then $\im(r_B)\cong r_A(M_S)$ as $\N(G)$-modules, and thus
    \begin{equation}\label{eq. connect to a finite matrix}
        \dim_G r_A(M_S)=\dim_G \im(r_B)=\rk_G(B).
    \end{equation}
    By \Cref{thm. matrix approximation ver 2}, there exists a finite subset $\F_\delta\subset G\smallsetminus\{1\}$ such that if a normal subgroup $N\lhd G$ satisfies $N\cap\F_\delta=\emptyset$, then
    \begin{equation}\label{eq. estimate for the finite submatrix}
        \rk_{G/N}(B)>\rk_G(B)-\dfrac{\delta}{2}.
    \end{equation}
    For every $N\lhd G$, let $\epsilon_N\colon G\rightarrow G/N$ be the natural homomorphism, and let $r_{\epsilon_N(A)}\colon (\N(G/N))^I\rightarrow (\N(G/N))^J$ (resp. $r_{\epsilon_N(B)}\colon (\N(G/N))^S\rightarrow (\N(G/N))^T$) be the $\N(G/N)$-module homomorphism induced by the right-multiplication of $\epsilon_N(A)$ (resp. $\epsilon_N(B)$). 

    For each $i\in I$, let $\bar e_i=(\bar x_{i'})_{i'\in I}\in (\N(G/N))^I$ be the element such that
    \[
    \bar x_{i'}=
    \begin{cases}
        1,&\text{if } i'=i\\
        0,&\text{if } i'\neq i.
    \end{cases}
    \]
    And let $\overline M_S$ be the submodule of $(\N(G/N))^I$ generated by $\bar e_i, i\in S$. Then $\im( r_{\epsilon_N(B)})\cong r_{\epsilon_N(A)}(\overline M_S)$ as $\N(G)$-modules, and thus
    \begin{equation}\label{eq. transform back}
    \begin{aligned}
        \rk_{G/N}(B)&=\dim_{G/N}\im( r_{\epsilon_N(B)})=\dim_{G/N}\bar r_{\epsilon_N(A)}(\overline M_S)\\
        &\leqslant \dim_{G/N}\im( r_{\epsilon_N(A)})=\rk_{G/N}(A).
        \end{aligned}
    \end{equation}
    The desired result follows from \eqref{eq. definition of rank}, \eqref{eq. connect to a finite submodule}, \eqref{eq. connect to a finite matrix}, \eqref{eq. estimate for the finite submatrix} and \eqref{eq. transform back}.
\end{proof}

\begin{corollary}\label{cor. l2 betti approximation ver 2}
    Let $G$ be a finitely generated virtually locally indicable group and let $X$ be a free $G$-CW complex of type $F_{k+1}$ for some $k\in\mathbb N$. Then for every $\delta>0$, there exists a finite subset $\F_\delta\subset G\smallsetminus\{1\}$ such that if a normal subgroup $N\lhd G$ satisfies $N\cap \F_\delta=\emptyset$, then 
    \[b^{(2)}_{k+1}(X ;G)+\delta>b^{(2)}_{k+1}(N\backslash X ; G/N),\]
    and for all $n\leqslant k$, we have
    \[|b^{(2)}_n(X ;G)-b^{(2)}_n(N\backslash X ; G/N)|<\delta.\]
\end{corollary}

\begin{proof}
Let $\partial_\ast$ be the boundary map of $C_\ast(X)$. For each $n\in \Z$ and each $G$-orbit of the $n$-cells of $X$, pick one representative. The chosen representatives form a $\Z G$-basis $B_n$ of the module $C_n(X)$ of $X$, yielding an isomorphism $C_n(X)\cong (\Z G)^{B_n}$.

Let $A_n$ be the tamed matrix over $\Z G$ such that 
\[\partial_n\colon C_n(X)\cong (\Z G)^{B_n}\rightarrow C_{n-1}(X)\cong (\Z G)^{B_{n-1}}\]
is given by the right-multiplication of $A_n$. Then for $n\leqslant k+1$, we have
\[b^{(2)}_n(X ;G)=|B_n|-\rk_G (A_n)-\rk_G (A_{n+1}),\]
\[b^{(2)}_n(N\backslash X ; G/N)=|B_n|-\rk_{G/N} (A_n)-\rk_{G/N} (A_{n+1}).\]
For $n\leqslant k$, let $\F_{n1}\subset G\smallsetminus\{1\}$ (resp. $\F_{n2}\subset G\smallsetminus\{1\}$) be the finite subset given by \Cref{thm. matrix approximation ver 2} with respect to $A_n$ (resp. $A_{n+1}$) and $\delta/2$. Let $\F_{k+1}\subset G\smallsetminus\{1\}$ be the finite subset given by \Cref{lem. infinite matrix approximation} with respect to $A_{k+2}$ and $\delta/2$. Let 
\[\F_\delta=\F_{k+1}\cup\left(\bigcup^{k+1}_{n=0}(\F_{n1}\cup F_{n2})\right),\]
and let $N\lhd G$ be a normal subgroup such that $N\cap\F_\delta=\emptyset$. By \Cref{thm. matrix approximation ver 2} and \Cref{lem. infinite matrix approximation}, we have
\[\rk_{G/N}(A_{k+2})>\rk_G(A_{k+2})-\dfrac{\delta}{2},\]
and
\[|\rk_G(A_n)-\rk_{G/N}(A_n)|<\delta/2\]
for $n\leqslant k+1$.  Therefore,

\begin{align*}
    b^{(2)}_{k+1}(N\backslash X\; ; G/N)&=|B_{k+1}|-\rk_{G/N} (A_{k+1})-\rk_{G/N} (A_{k+2})\\
    &<|B_{k+1}|-\rk_G (A_{k+1})-\rk_G (A_{k+2})+\delta\\
    &=b^{(2)}_{k+1}(X\; ;G)+\delta,
\end{align*}
and
\[|b^{(2)}_n(X\; ;G)-b^{(2)}_n(N\backslash X\; ; G/N)|\leqslant |\rk_G A_n-\rk_{G/N}A_n|+|\rk_G A_{n+1}-\rk_{G/N}A_{n+1}|<\delta\]
for $n\leqslant k$.
\end{proof}

The above results provide estimates of $L^2$-Betti numbers. To obtain actual computations rather than just estimation, we will combine them with the Atiyah Conjecture. We will need the following  result:

\begin{proposition}\label{lem. det and atiyah ver 2}
    Let $G$ be a finitely generated virtually locally indicable group. Suppose that there exists $C\in\mathbb{N}$ such that for all finite subsets $\fS\subset G\smallsetminus\{1\}$, there exists a normal subgroup $N_{\fS}\lhd G$ such that 
    \begin{enumerate}
        \item[(i)] $N_{\fS}\cap \fS=\emptyset$;
        \item[(ii)] $G/N_{\fS}$ satisfies the Atiyah Conjecture; and
        \item[(iii)] $\max\fin(G/N_{\fS})\leqslant C$, where $\max$ means taking the maximum of a set.
    \end{enumerate}
    Let $X$ be a free $G$-CW complex of type $F_{k+1}$ for some $k\in\mathbb N$. Then for all $n\leqslant k$, we have
    \[b^{(2)}_n(X\; ;G)\in \dfrac{1}{C!}\cdot\Z.\]
\end{proposition}

\begin{proof}
    For each finite subset $\fS\subset G\smallsetminus\{1\}$, let $N_{\fS}\lhd G$ be a normal subgroup such that the above (i) through (iii) hold. The family
    \[\N=\{N_{\fS}\mid \fS \text{ is a finite subset of }G\}\]
    is a directed set under the partial order $N_{\fS}\prec N_{\fS'}$ if and only if $\fS\subset \fS'$.
     For each finite matrix $A$ over $\Z G$, we have
    \[\rk_G(A)=\varinjlim \rk_{G/N_{\fS}}(A)\in\mathbb{N}\in \dfrac{1}{(\max_{N_{\fS}\in\N}\fin(G/N_{\fS}))!}\cdot\Z\leqslant \dfrac{1}{C!}\cdot\Z\]
    by \Cref{thm. matrix approximation ver 2}. 
    
    Now fix $n\leqslant k$. Let $B_n$ (resp. $B_{n+1},B_{n-1}$) be the set obtained by choosing one representative of each $G$-orbit of the $n$-cells (resp. $(n+1)$-cells, $(n-1)$-cells) of $X$. Then $B_n$ (resp. $B_{n+1},B_{n-1}$) is a $\Z G$-basis of $C_n(X)$ (resp. $C_{n+1}(X),C_{n-1}(X)$). Let $\partial_\ast$ be the boundary map of $C_\ast(X)$, and let $A_n$ (resp. $A_{n+1}$) be the matrix representative of $\partial_n$ (resp. $\partial_{n+1}$) under the bases $B_n$ and $B_{n-1}$ (resp. $B_{n+1}$ and $B_n$). Then
    \[b^{(2)}_n(X\; ;G)=|B_n|-\rk_G A_n-\rk_G A_{n+1}\in\dfrac{1}{C!}\cdot \Z.\qedhere\]
\end{proof}

\subsection{The first $L^2$-Betti number}

Our next results yield an upper bound for the first $L^2$-Betti number of Dehn filling (see \Cref{cor. locally indicable dehn filling}). Estimation of the other $L^2$-Betti numbers needs an additional ingredient called a Dehn filling space, which will be studied in \Cref{sec. df and l2 betti}.

\begin{lemma}\label{lem. first l2 inequality}
    Let $G$ be any group and let $N\lhd G$ be a normal subgroup. Then
    \[b^{(2)}_1(N\backslash EG ; G/N)\geqslant b^{(2)}_1(G/N).\]
\end{lemma}

\begin{proof}
    Let $EG$ (resp. $E(G/N)$) be a free contractible $G$-CW complex (resp. $(G/N)$-CW complex). The quotient map $G\twoheadrightarrow G/N$ induces a $G$-equivariant cellular map $f\colon EG \rightarrow E(G/N)$, which factors through a $G/N$-equivariant cellular map $\tilde f\colon N\backslash EG\rightarrow E(G/N)$.

    We will use \cite[Theorem 6.54 (1)]{luck2002l2}. To this end, let us consider a subgroup $H\leqslant G/N$ and let $(N\backslash EG)^H$ (resp. $(E(G/N))^H$) be the subcomplex of $N\backslash EG$ (resp. $E(G/N)$) consisting of the global fixed points of $H$. If $H\neq\{1\}$, then $(N\backslash EG)^H=(E(G/N))^H=\emptyset$ and the assumptions of \cite[Theorem 6.54 (1)]{luck2002l2} are trivially satisfied. If $H=\{1\}$, then $(N\backslash EG)^H=N\backslash EG$ and $(E(G/N))^H=E(G/N)$. Clearly, $\tilde f$ induces an isomorphism $H_0(N\backslash EG ; \C)\rightarrow H_0(E(G/N) ; \C)$. As $H_1(E(G/N) ;\C)=0$, $\tilde f$ induces an epimorphism $H_1(N\backslash EG ;\C)\rightarrow H_1(E(G/N) ;\C)$. So the assumptions of \cite[Theorem 6.54 (1)]{luck2002l2} are met, which yields
    \[b^{(2)}_1(N\backslash EG ; G/N)\geqslant b^{(2)}_1(E(G/N) ;G/N)=b^{(2)}_1(G/N).\]
\end{proof}

\begin{corollary}\label{cor. upper bound ver 2}
    Let $G$ be a finitely generated virtually locally indicable group. Then for every $\delta>0$, there exists a finite subset $\F_\delta\subset G\smallsetminus \{1\}$ such that if a normal subgroup $N\lhd G$ satisfies $N\cap \F_\delta=\emptyset$, then
    \[b^{(2)}_1(G/N)< b^{(2)}_1 (G)+\delta\]
\end{corollary}

\begin{proof}
    Let $X$ be the Cayley $2$-complex of $G$ with respect to any presentation of $G$ with finitely many generators, and let $\F_\delta\subset G\smallsetminus\{1\}$ be the finite subset provided by \Cref{cor. l2 betti approximation ver 2}  with respect to $k=0$ and $X$. Then if a normal subgroup $N\lhd G$ satisfies $N\cap \F_\delta=\emptyset$, we have
    \[b^{(2)}_1(N\backslash X ;G/N)<b^{(2)}_1(X ;G)+\delta.\] 
    By \Cref{lem. first l2 inequality}, we have
    \[b^{(2)}_1(N\backslash EG ; G/N)\geqslant b^{(2)}_1(G/N).\]
    Note that
\[b^{(2)}_1(N\backslash EG ; G/N)=b^{(2)}_1(N\backslash X ;G/N),\]
    \[b^{(2)}_1(G)=b^{(2)}_1(X ;G),\]
    because there exists a model for $EG$ with $X$ as its $2$-skeleton (see e.g., \cite[p.~89-90]{hatcher2002algebraic}). The desired result follows from the above equalities and inequalities.
\end{proof}

\begin{corollary}\label{cor. locally indicable dehn filling}
    Let $G$ be a finitely generated virtually locally indicable group, and let $H\hookrightarrow_h G$ be a hyperbolically embedded subgroup. Then for every $\delta>0$, there exists a finite subset $\F_\delta\subset H\smallsetminus\{1\}$ such that if a normal subgroup $N\lhd H$ satisfies $N\cap \F_\delta=\emptyset$, then
    \begin{equation}\label{eq. first l2 betti inequality}
        b^{(2)}_1(G/\ll N \rr)< b^{(2)}_1(G)+\delta.
    \end{equation}
\end{corollary}

\begin{proof}
    By \Cref{cor. upper bound ver 2}, there exists a finite subset $\fS_\delta\subset G\smallsetminus\{1\}$ such that if a normal subgroup $M\lhd G$ satisfies $M\cap \fS_\delta=\emptyset$, then 
    \begin{equation}\label{eq. li dehn 8}
        b^{(2)}_1(G/M)<b^{(2)}_1(G)+\delta.
    \end{equation}
    By \cite[Theorem 7.19 (c)]{dahmani2017hyperbolically},  there exists a finite subset $\F_\delta\subset H\smallsetminus\{1\}$ such that if a normal subgroup $N\lhd H$ satisfies $ N\cap \F_\delta=\emptyset$, then $\ll N \rr\cap \fS_\delta=\emptyset$, which implies the desired result.
\end{proof}

\section{Dehn filling space and $L^2$-Betti numbers}\label{sec. df and l2 betti}

\subsection{Dehn filling space}\label{subsec: Dehn filling space}

Let $G$ be a group, $H\leqslant G$ a subgroup, and $N\lhd H$ a normal subgroup such that $(G,H,N)$ is a Cohen--Lyndon triple, i.e., there exists a left transversal $T$ of $H\ll N \rr$ in $G$ such that 
\[\ll N \rr = \Asterisk_{t\in T}tNt^{-1}.\]
As in the introduction, we denote $\bg=G/\ll N \rr$ and $\bh=H/N$. The goal of this subsection is to build a specific $K(\bg,1)$-space. 

Note that the natural map $\bh\rightarrow \bg$ is injective (see e.g., \cite[Lemma 6.4]{sun2018cohomologyi}), and below we will view $\bh$ as a subgroup of $\bg$. Let $BG$ (resp. $BH, B\bh$) be a $K(G,1)$ (resp. $K(H,1),K(\bh,1)$) CW-complex.

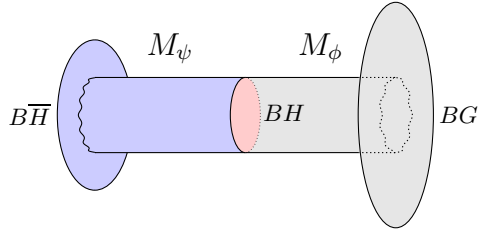
\begin{figure}[ht!]
    \centering
    \begin{tikzpicture}
        \begin{scope}
            \clip (-0.7,-1.2) rectangle (2.3,1.2);
            \fill[blue!20] (0,0) ellipse (0.5 and 1);
            \fill[blue!20] (0,-0.5) rectangle (2,0.5); 
        \end{scope}
        
        \begin{scope}
            \clip (1.7,-1.7) rectangle (4.7,1.7);
            \fill[gray!20] (4,0) ellipse (0.5 and 1.5);
            \fill[gray!20] (2,-0.5) rectangle (4,0.5); 
        \end{scope}

        \fill[red!20] (2,0) ellipse (0.2 and 0.5); 

        \draw (1,0.5) circle (0pt) node[anchor=south, minimum size=0.3in] {\large $M_{\psi}$};
        \draw (3,0.5) circle (0pt) node[anchor=south, minimum size=0.3in] {\large $M_{\phi}$};
        \draw (0,0.5) -- (2,0.5);
        \draw (0,-0.5) -- (2,-0.5);
        \draw[dashed,dash pattern=on 0.2pt off 1pt]  (2,.5) arc (90:-90:0.2 and 0.5) node[anchor=west, yshift=15pt, minimum size=0.4in] {\footnotesize $BH$};
       \draw (2,.5) arc (90:270:0.2 and 0.5);
             \draw (0,1) arc (90:30:0.5 and 1);
                                       \draw (0.43,-0.5) arc (-30:-90:0.5 and 1);
                                       
       \draw (0,1) arc (90:270:0.5 and 1) node[anchor=south east, yshift=21pt, xshift=-13pt] {\footnotesize $B\bh$}; 
        \draw[domain=90:270, samples=200, smooth, variable=\t]
        plot ({0.2*cos(\t) + 0.02*cos(15*\t)}, {0.5*sin(\t)});
           \draw[dashed, dash pattern=on 0.5pt off 1pt, domain=0:360, samples=200, smooth, variable=\t]
        plot ({4 + 0.2*cos(\t) + 0.02*cos(10*\t)}, {0.5*sin(\t)});
        \draw (2,0.5) -- (3.5,0.5);
            \draw (3.5,0.5)[dashed,dash pattern=on 0.5pt off 1pt] -- (4,0.5);
    \draw (2,-0.5) -- (3.5,-0.5);
            \draw (3.5,-0.5)[dashed,dash pattern=on 0.5pt off 1pt] -- (4,-0.5);
        \draw (4,0) ellipse (0.5 and 1.5) node[anchor=west, minimum size=0.65in] {\footnotesize $BG$};
    \end{tikzpicture}
    
    \alt{A schematic diagram of a Dehn filling space. Two shaded cylindrical regions on the left and right represent the different mapping cylinders. A small red ellipse in the centre represents a common subspace connecting the two regions. The bases of the left and right regions are shown as larger disks, and dashed lines indicate connections between the subspaces.}
    \caption{Dehn filling space associated to the triple $(G,H,N)$. The mappings cylinders $M_{\psi}$ and $M_{\phi}$ are shaded blue and grey, respectively. Their bases $B\bh$ and $BG$ are represented by larger discs. The common subspace $BH$ is shaded red.}
    \label{fig:Dehn_space}
\end{figure}

\begin{definition}\label{Dehn_space}
There is a classifying map $\phi\colon BH\rightarrow BG$ (resp. $\psi\colon BH\rightarrow B\bh$) induced by the inclusion $H\hookrightarrow G$ (resp. the quotient map $q\colon H\rightarrow \bh$). Let $X$  be the complex obtained by gluing the mapping cylinders $M_{\phi},M_{\psi}$ along their common subcomplex $BH$ (see Figure \ref{fig:Dehn_space}). We call the complex $X$ a {\it Dehn filling space} associated to the triple  $(G,H,N)$.  
\end{definition}
For simplicity, we will use the following notation. Let
\[\tilde\phi\colon \bg\times_{\bh}(N\backslash EH)\cong \ll N \rr\backslash  \big(G\times_{H} EH\big)\rightarrow \ll N \rr\backslash EG\]
and
\[\tilde\psi\colon \bg\times_{\bh}(N\backslash EH)\rightarrow \bg\times_{\bh}E\bh\]
be the $\bg$-equivariant maps induced by $\phi$ and $\psi$, respectively. Also let $M_{\tilde\phi}$ and $M_{\tilde\psi}$ be the mapping cylinder of $\tilde\phi$ and $\tilde\psi$, respectively.

\begin{theorem}\label{thm. topology}
Let $(G,H,N)$ be a Cohen--Lyndon triple. Then every Dehn filling space associated to the triple $(G, H, N)$ is a $K(\bg,1)$-space.
\end{theorem}

By combining the above theorem with \Cref{thm. cl} we get the following:

\begin{corollary}\label{cor. topology}
Let $G$ be a group with a hyperbolically embedded subgroup $H\hookrightarrow_h G$. Then for sufficiently deep $N\lhd H$, every Dehn filling space associated to the triple $(G, H, N)$ is a $K(\bg,1)$-space.
\end{corollary}

\begin{proof}[Proof of \Cref{thm. topology}]
We will use the notation from \Cref{Dehn_space}. We first prove that $\pi_1(X)=\bg$. By the Seifert--van Kampen theorem, $\pi_1(X)$ has the presentation
\begin{equation}\label{eq. pre1}
    \pi_1(X)=\langle G,\bh \mid h=q(h),\forall h\in H\leqslant G \rangle.
\end{equation}
In particular, there is an epimorphism $\tau\colon G\twoheadrightarrow \pi_1(X)$ that maps every $n\in N$ to $1$. Thus, $\tau$ descends to an epimorphism (still denoted by) $\tau\colon \bg\twoheadrightarrow\pi_1(X)$.

Let $p\colon G\rightarrow \bg$ be the natural quotient map. Then from the presentation \eqref{eq. pre1} we see that there is a homomorphism $\eta\colon \pi_1(X)\rightarrow \bg$ given by $\eta(g)=p(g)$ for all $g\in G$ and $\eta(\bar h)=\bar h$ for all $\bar h\in\bh$. It is easy to see that $\tau$ and $\eta$ are mutual inverses. Thus, $\pi_1(X)=\bg$.

It remains to prove that the universal cover $\widetilde X$ of $X$ is contractible. Since $\widetilde X$ is a simply connected CW-complex,  it suffices to prove that $H_n(\widetilde X)=0$ for all $n\geqslant 2$. For simplicity, all homology groups have integral coefficients (where the group action on $\Z$ is trivial) for the rest of this proof.

The space $\widetilde X$ can be obtained by gluing the mapping cylinders $M_{\widetilde\phi},M_{\widetilde\psi}$ along their common subcomplex $\bg\times_{\bh} (N\backslash EH)$. We thus have a Mayer--Vietoris sequence:
\[\cdots\rightarrow H_{n+1}(\widetilde X)\rightarrow H_n(\bg\times_{\bh} (N\backslash EH))\xrightarrow{\lambda} H_n(M_{\widetilde\phi})\oplus H_n(M_{\widetilde\psi})\rightarrow H_n(\widetilde X)\rightarrow \cdots\]
It is enough to prove that for all $n\geqslant 2$, $\lambda\colon H_n(\bg\times_{\bh} (N\backslash EH))\rightarrow H_n(M_{\widetilde\phi})\oplus H_n(M_{\widetilde\psi})$ is an isomorphism. Note that $M_{\tilde\psi}$ deformation retracts onto $\bg\times_{\bh}E\bh$ and $H_n(\bg\times_{\bh}E\bh)=0$. Thus, $H_n(M_{\widetilde\psi})=0$. So it suffices to prove that the homomorphism \[\theta\colon H_n(\bg\times_{\bh} (N\backslash EH))\rightarrow H_n(M_{\widetilde\phi})\]
induced by $\lambda$ is an isomorphism. Note that $M_{\widetilde\phi}$ deformation retracts onto $\ll N \rr \backslash EG$, and thus is a $K(\ll N \rr, 1)$-space. By the Cohen--Lyndon property,
\[H_n(\ll N \rr \backslash EG)=\bigoplus_{t\in T}H_n(tNt^{-1}),\]
\noindent where $T$ is a left transversal of $H\ll N \rr$ in $G$.  The space $\bg\times_{\bh} (N\backslash EH)$ is a disjoint union of copies of $N\backslash EH$. Each of these copies is labeled by a unique element $p(t)$ for some $t\in T$. The inclusion $\bg\times_{\bh} (N\backslash EH)\hookrightarrow M_{\widetilde\phi}$, which induces the homology map $\theta$, can be seen as a disjoint union of inclusions of each copy of $N\backslash EH$. Consider one such copy $(N\backslash EH)_{p(t)}$ labeled by $p(t)$. The copy $(N\backslash EH)_{p(t)}$ is a $K(tNt^{-1},1)$-space and the inclusion $i_t\colon (N\backslash EH)_{p(t)}\hookrightarrow M_{\widetilde\phi}$ is the classifying map. So the induced map on homology maps $H_n((N\backslash EH)_{p(t)})=H_n(tNt^{-1})$ isomorphically onto the summand $H_n(tNt^{-1})$ of $H_n(\ll N \rr \backslash EG)$. Since this holds for all $t\in T$ we conclude that $\theta$ is an isomorphism, as desired.
\end{proof}

\begin{corollary}\label{cor. finiteness conditions}
Let $G$ be a group with a hyperbolically embedded subgroup $H\hookrightarrow_h G$.
\begin{enumerate}
    \item[(i)] If $G$ is of type $F_n$ for some $n\in \mathbb N^{+}\cup \{\infty\}$, then for all sufficiently deep normal  $N\lhd H$, we have that $\bg$ is of type $F_n$  if and only if $\bh$ is of type $F_n$. 

    \item[(ii)] If both $G$ and $H$ are of type $F$, then for all sufficiently deep normal  $N\lhd H$ such that $\bh$ is of type $F$, we have that $\bg$ is of type $F$.
\end{enumerate}

\end{corollary}
\begin{proof} \hspace{1cm}

\noindent{\it (i).} Suppose $\bg$ is of type $F_n$. By Theorem \ref{thm. simple Dehn filling}, $\bh$ is hyperbolically embedded in $\bg$. Applying, Corollary 4.32 of \cite{dahmani2017hyperbolically}, we obtain that $\bh$ is of type $F_n$.

Now, suppose $\bh$ is of type $F_n$. We will show that the Dehn filling space has finite $n$-skeleton. It follows that there is a model for $B\bh$ with finite $n$-skeleton. Since $G$ is of type $F_n$, by Corollary 4.32 of \cite{dahmani2017hyperbolically}, $H$  is of type $F_n$. Thus, there are  models for $BG$ and $BH$ with finite $n$-skeleta. Using the notation of \Cref{Dehn_space}, it follows that the mapping cylinders $M_{\phi}$ and $M_{\psi}$ have finite $n$-skeleta. Since they are glued along $BH$, the resulting Dehn filling space has the desired finiteness properties. By \Cref{cor. topology}, this Dehn filling space is a $K(\bg, 1)$ space provided $N$ is sufficiently deep.\\

\noindent{\it (ii).}  The proof is analogous to part (i). There are models for $BG,BH$ and $B\bh$ that are finite complexes. Using the notation of \Cref{Dehn_space}, it follows that the mapping cylinders $M_{\phi}$ and $M_{\psi}$ are finite complexes. Since they are glued along $BH$, the resulting Dehn filling space is a finite complex. By \Cref{cor. topology}, this Dehn filling space is a $K(\bg, 1)$ space provided $N$ is sufficiently deep.
\end{proof}

In \cite{sun2019cohomologyii}, we gave a rather elaborate argument based on the spectral sequence derived in \cite[Theorem 4.3]{sun2019cohomologyii} to obtain the next result, which now follows easily from excision on Dehn filling space.

\begin{corollary}[{\cite[Theorem A (ii)]{sun2019cohomologyii}}]\label{cor. alg excision} Let $(G,H,N)$ be a Cohen--Lyndon triple and $A$ be a $\bg$-module. Then, there is natural isomorphism induced by the quotient maps $G\to \bg$ and  $H\to \bh$,
    $$H^n(G, H ; A)\cong H^n(\bg, \bh ; A),  \; \; \forall n\geq 0.$$
\end{corollary}
\begin{proof} Recall that the space $\widetilde X$ of \Cref{thm. topology} is obtained by gluing the mapping cylinders $M_{\widetilde\phi},M_{\widetilde\psi}$ along $\bg\times_{\bh} (N\backslash EH)$. Set $Z=\bg\times_{\bh}(N\backslash EH)$, so that
$\widetilde X=M_{\widetilde\phi}\cup_Z M_{\widetilde\psi}$.
The $\bg$-module $A$ determines a local coefficient system on
$\bg\backslash\widetilde X$.  We equip all subspaces appearing
below with the restrictions of this coefficient system. Applying excision for cohomology
with local coefficients \cite[Theorem 5.13]{davis2001lecture},
we obtain the natural restriction isomorphism
\[
H^n_{\bg}(\widetilde X,M_{\widetilde\psi};A)
\xrightarrow{\cong}
H^n_{\bg}(M_{\widetilde\phi},Z;A),
\qquad n\geq 0.
\]
  Since $\widetilde X$ is a model for $E\bg$, by definition, 
  $$H^n(\bg,  \bh  ; A)=H^n_{\bg}(\widetilde X,  M_{\widetilde\psi}  ; A)=H^n_G(\widetilde X,  M_{\widetilde\psi} ; A).$$
  Since $\ll N\rr$ acts trivially on $A$,   the pair $(M_{\widetilde\phi},\; Z)$ and its $\ll N\rr$-cover $(\widetilde M_{\widetilde\phi},\; G \times_{H} EH)$ have the same cohomology, i.e.
  $$H^n_G(\widetilde  M_{\widetilde\phi},\; G \times_{H} EH\; ; A)=H^n_{\bg}(M_{\widetilde\phi},\; Z\; ; A).$$
  By definition,  $H^n(G, H  ; A) = H^n_{G}(\widetilde  M_{\widetilde\phi}, G \times_{H} EH\; ; A).$
\end{proof}

\begin{lemma}\label{lem:L2-diagram}   Let $G$ be a group with a hyperbolically embedded subgroup $H\hookrightarrow_h G$. Then for all sufficiently deep normal  $N\lhd H$, we have the following commutative diagram with exact rows
\begin{footnotesize}
\begin{equation}\label{eq. l2-diagram}
\begin{tikzcd}
  \dots \;  \N(\bg)\otimes_{\N(\bh)}\cH_n(N\backslash EH; \N(\bh)) \arrow[d] \arrow[r] & \cH_n(\ll N\rr\backslash EG; \N(\bg)) \arrow[r] \arrow[d] & \cH_n(\bg, \bh;\N(\bg)) \arrow[d, "="] \; \dots \\
  \dots \;  \N(\bg)\otimes_{\N(\bh)}\cH_n(\bh;\N(\bh)) \arrow[r] & \cH_n(\bg;\N(\bg)) \arrow[r] & \cH_n(\bg, \bh;\N(\bg))  \; \dots 
\end{tikzcd}
\end{equation}
\end{footnotesize}
Moreover, if both $G$ and $\bh$ are of type $F_k$ for some $k\in\mathbb N$, then for $n\leqslant k$ we have the following inequalities
\begin{equation}\label{eq. li dehn 3}
\begin{aligned}
    b^{(2)}_n(\ll N \rr\backslash EG; \bg)-b^{(2)}_n(N\backslash EH; \bh)
    &\leqslant b^{(2)}_n(\bg,\bh)\\
    &\leqslant b^{(2)}_n(\ll N \rr\backslash EG; \bg)+b^{(2)}_{n-1}(N\backslash EH; \bh),
 \end{aligned}  
\end{equation}
\begin{equation}\label{eq. li dehn 4}
  b^{(2)}_n(\bg)-b^{(2)}_n(\bh)
    \leqslant b^{(2)}_n(\bg,\bh)
    \leqslant b^{(2)}_n(\bg)+b^{(2)}_{n-1}(\bh).
    \end{equation}
\end{lemma}
\begin{proof}
We use the notation from \Cref{Dehn_space}. 
The map of CW pairs 
$$(M_{\tilde\phi},\bg\times_{\bh} N\backslash(EH))\rightarrow(\widetilde X, M_{\tilde \psi})$$
gives rise to a commutative diagram with exact rows
\begin{footnotesize}
\begin{equation}\label{eq. pre diagram}
\begin{tikzcd}
\dots \; \cH_n(\bg\times_{\bh}(N\backslash EH); \N(\bg)) \arrow[d] \arrow[r] & \cH_n(M_{\tilde\phi}; \N(\bg)) \arrow[r] \arrow[d] & \cH_n(M_{\tilde\phi}, \bg\times_{\bh}(N\backslash EH); \N(\bg)) \arrow[d] \; \dots \\
  \dots \;  \cH_n(M_{\tilde\psi}; \N(\bg)) \arrow[r] & \cH_n(\widetilde X; \N(\bg)) \arrow[r] & \cH_n(\widetilde X, M_{\tilde\psi}; \N(\bg))  \; \dots
\end{tikzcd}
\end{equation}
\end{footnotesize}
By \cite[Theorem 6.54 (7)]{luck2002l2}, we have natural isomorphisms
\begin{align*}\cH_n(\bg\times_{\bh} (N\backslash EH); \N(\bg))&\cong \N(\bg)\otimes_{\N(\bh)}\cH_n(N\backslash EH; \N(\bh)),\\
\cH_n(M_{\tilde\psi}; \N(\bg))&\cong \N(\bg)\otimes_{\N(\bh)}\cH_n(\bh;\N(\bh)).
\end{align*}
By \Cref{Dehn_space} and \Cref{cor. topology}, we also have
\begin{align*}\cH_n(M_{\tilde\phi}; \N(\bg))&\cong \cH_n(\ll N\rr\backslash EG; \N(\bg)),\\
\cH_n(\widetilde X; \N(\bg))&\cong \cH_n(\bg; \N(\bg)).
\end{align*}
By excision, we have the natural isomorphism
\[\cH_n(M_{\tilde\phi}, \bg\times_{\bh}(N\backslash EH); \N(\bg))\cong \cH_n(\widetilde X, M_{\tilde\psi}; \N(\bg))\cong \cH_n(\bg, \bh; \N(\bg)).\]
The above isomorphisms transform \eqref{eq. pre diagram} to \eqref{eq. l2-diagram}.

If both $G$ and $\bh$ are of type $F_k$ for some $k\in\mathbb N$, then from \eqref{eq. pre diagram} and \Cref{prop. properties of dimension} \ref{item. additivity} we deduce the inequalities for all $n\leqslant k$,
\begin{align}\label{eq. pre li dehn 3}
    b_n^{(2)}(M_{\tilde\phi} ; \bg)-b_n^{(2)}(\bg\times_{\bh}(N\backslash EH) ; \bg) & \leqslant  b_n^{(2)}(M_{\tilde\phi}, \bg\times_{\bh}(N\backslash EH)  ;\bg)\\\nonumber
   & \leqslant  b_n^{(2)}(M_{\tilde\phi} ; \bg)+b_{n-1}^{(2)}(\bg\times_{\bh}(N\backslash EH) ; \bg),\nonumber
\end{align}
\begin{align}\label{eq. pre li dehn 4}
    b_n^{(2)}(\widetilde X\; ;\bg)-b^{(2)}_n(M_{\tilde\psi}\; ;\bg) & \leqslant b_n^{(2)}(\widetilde X, M_{\tilde\psi}\; ;\bg)\\\nonumber
   &  \leqslant b_n^{(2)}(\widetilde X\; ;\bg)+b^{(2)}_{n-1}(M_{\tilde\psi}\; ;\bg),\nonumber
\end{align}
as these numbers are all finite by \Cref{prop. properties of dimension} \ref{item. finite dimension}.

By \cite[Theorem 6.54 (7)]{luck2002l2}, we have
\begin{align*}
b_n^{(2)}(\bg\times_{\bh}(N\backslash EH);  \bg)&=b_n^{(2)}(N\backslash EH; \bh),\\
b_{n-1}^{(2)}(\bg\times_{\bh}(N\backslash EH);  \bg)&=b_{n-1}^{(2)}(N\backslash EH ;\bh),\\
b^{(2)}_n(M_{\tilde\psi}; \bg)&=b^{(2)}_n(\bh),\\
b^{(2)}_{n-1}(M_{\tilde\psi}; \bg)&=b^{(2)}_{n-1}(\bh).
\end{align*}

The above equalities turn \eqref{eq. pre li dehn 3} and \eqref{eq. pre li dehn 4} into \eqref{eq. li dehn 3} and \eqref{eq. li dehn 4}, respectively.
\end{proof}

We are now ready to prove our first general result on $L^2$-Betti numbers of groups, providing bounds for the $L^2$-Betti numbers of Dehn fillings of virtually locally indicable groups.
\begin{theorem}\label{thm. li dehn filling}
    Let $G$ be a virtually locally indicable group of type $F_{k+1}$ for some $k\geqslant 0$, and let $H\hookrightarrow_h G$ be a hyperbolically embedded subgroup. Then for every $\delta>0$, there exists a finite subset $\F_\delta\subset H\smallsetminus\{1\}$ such that if a normal subgroup $N\lhd H$ satisfies $N\cap \F_\delta=\emptyset$ and $H/N$ is of type $F_{k+1}$, then 
    \begin{equation}\label{eq. highest dim l2 betti inequality}
        b^{(2)}_{k+1}(G/\ll N \rr)<b^{(2)}_{k+1}(G)+b^{(2)}_k(H)+b^{(2)}_{k+1}(H/N)+\delta,
    \end{equation}
    and for $n\leqslant k$, we have      
    \begin{equation}\label{eq. l2 betti inequality}
  -b^{(2)}_n(H)-b^{(2)}_{n-1}(H/N)-\delta< b^{(2)}_n(G/\ll N \rr)- b^{(2)}_n(G) < b^{(2)}_{n-1}(H)+b^{(2)}_n(H/N)+\delta.
    \end{equation}
\end{theorem}
\begin{proof}
By \Cref{cor. l2 betti approximation ver 2}, there exists a finite subset $\fS_\delta\subset G\smallsetminus\{1\}$ such that if a normal subgroup $M\lhd G$ satisfies $M\cap \fS_\delta=\emptyset$, then we have
    \begin{equation}\label{eq. highest li dehn for G}
        b^{(2)}_{k+1}(M\backslash EG\; ; G/M)<b^{(2)}_{k+1}(G)+\delta/2,
    \end{equation}
and
    \begin{equation}\label{eq. li dehn 1}
        |b^{(2)}_n(G)-b^{(2)}_n(M\backslash EG\; ; G/M)|<\delta/2 \text{ for } n\leqslant k.
    \end{equation}
By \cite[Theorem 7.19 (c)]{dahmani2017hyperbolically}, there exists a finite subset $\F_1\subset H\smallsetminus\{1\}$ such that if a normal subgroup $N\lhd H$ satisfies $N\cap \F_1=\emptyset$, then $\ll N \rr\cap \fS_\delta=\emptyset$.

By \cite[Corollary 4.32 (b)]{dahmani2017hyperbolically}, $H$ is of type $F_{k+1}$. Therefore, by \Cref{cor. l2 betti approximation ver 2}, there exists a finite subset $\F_2\subset G\smallsetminus\{1\}$ such that if a normal subgroup $N\lhd H$ satisfies that $N\cap \F_2=\emptyset$, then we have
    \begin{equation}\label{eq. highest li dehn for H/N}
    b^{(2)}_{k+1}(N\backslash EH\; ; H/N)<b^{(2)}_{k+1}(H)+\delta/2,
    \end{equation}
and
    \begin{equation}\label{eq. li dehn 2}
        |b^{(2)}_n(H)-b^{(2)}_n(N\backslash EH\; ; H/N)|<\delta/2 \text{ for } n\leqslant k.
    \end{equation}

By \Cref{lem:L2-diagram}, there exists a finite subset $\F_3\subset H\smallsetminus\{1\}$ such that if a normal subgroup $N\lhd H$ satisfies $N\cap \F_3=\emptyset$, then \eqref{eq. li dehn 3} and \eqref{eq. li dehn 4} hold for all $n\leqslant k+1$.

Let $\F_\delta=\F_1\cup \F_2\cup\F_3$ and fix a normal subgroup $N\lhd H$ such that $N\cap \F_\delta=\emptyset$ and $H/N$ is of type $F_{k+1}$.
    
By \eqref{eq. highest li dehn for G}, \eqref{eq. li dehn 1} and the choice of $N$, we have
    \begin{equation}\label{eq. highest li dehn for quotient space}
        b^{(2)}_{k+1}(\ll N \rr\backslash EG\; ; G/\ll N \rr)<b^{(2)}_{k+1}(G)+\delta
    \end{equation}
and
    \begin{equation}\label{eq. li dehn 5}
        |b^{(2)}_n(G)-b^{(2)}_n(\ll N \rr\backslash EG\; ; G/\ll N \rr)|<\delta/2 \text{ for } n\leqslant k.
    \end{equation}

Inequalities \eqref{eq. highest dim l2 betti inequality} and \eqref{eq. l2 betti inequality} follow by combining \eqref{eq. li dehn 3}, \eqref{eq. li dehn 4}, \eqref{eq. highest li dehn for H/N} \eqref{eq. li dehn 2}, \eqref{eq. highest li dehn for quotient space}, and \eqref{eq. li dehn 5}.
\end{proof}

We will prove \Cref{thm. intro li} using \Cref{thm. li dehn filling}. In preparation, we first prove the following lemma:

\begin{lemma}\label{lem. large quotient}
    Let $G$ be a finitely generated infinite group. Then for every $n>0$, there exists a finite subset $\F\subset G\smallsetminus\{1\}$ such that if a normal subgroup $N\lhd G$ satisfies $N\cap \F=\emptyset$, then $|G/N|>n$, and thus $b^{(2)}_0(G/N)<1/n$.
\end{lemma}

\begin{proof}
    Let $S$ be a finite generating set of $G$. For $r\in\mathbb N^+$ we denote by $B_r$ the subgraph of the Cayley graph $\Gamma(G,S)$ consisting of all points within distance $r$ from the vertex $1$. Choose $R\in\mathbb N^+$ such that $B_R$ has more than $n$ vertices. Let $\F=G\cap B_{2R}\smallsetminus\{1\}$. Then for every two vertices $u\neq v\in B_R$, we have $uv^{-1}\in \F$. So if $N\lhd G$ is a normal subgroup such that $N\cap \F=\emptyset$, then $u$ and $v$ have different images in $G/N$, whence $|G/N|>|G\cap B_R|>n$. The last inequality of the lemma follows from \cite[Theorem 6.54 (8b)]{luck2002l2}.
\end{proof}

\begin{proof}[Proof of \Cref{thm. intro li}] First consider the case where $|H|<\infty$. Fix $n\geqslant 0$ and let $\F_{n,\delta}=H\smallsetminus\{1\}$. Then the only normal subgroup $N\lhd H$ that satisfies $N\cap \F_{n,\delta}=\emptyset$ is $\{1\}$. So $\bg=G$ and the result follows.

Now consider the case where $|H|=\infty$. Since $H$ is amenable, then so is $H/N$, being a quotient of an amenable group. By \cite[Theorem 6.37]{luck2002l2}, we have 
\begin{equation}\label{eq. vanishing l2 betti 1}
    b^{(2)}_n(H)=b^{(2)}_n(H/N)=0\text{ for all }n\geqslant 1.
\end{equation}
and by \cite[Theorem 6.54 (8b)]{luck2002l2}, we have
\begin{equation}\label{eq. vanishing l2 betti 2}
    b^{(2)}_0(H)=0.
\end{equation}
By \Cref{lem. large quotient}, there exists a finite subset $\F'_\delta\subset H\smallsetminus\{1\}$ such that 
\begin{equation}\label{eq. small l2 betti}
    b^{(2)}_0(\bh)<\delta/2
\end{equation}
whenever $N\cap \F'_\delta=\emptyset$. Now fix $n\geqslant 0$. \Cref{thm. li dehn filling} provides us a finite subset $\F'_{n,\delta}\subset H\smallsetminus\{1\}$ such that if $N\cap \F'_{n,\delta}=\emptyset$, then
\begin{align}\label{eq. bound l2 betti}
    -b^{(2)}_n(H)-b^{(2)}_{n-1}(H/N)-\delta/2 &< b^{(2)}_n(G/\ll N \rr)- b^{(2)}_n(G)\\
    &< b^{(2)}_{n-1}(H)+b^{(2)}_n(H/N)+\delta/2.\nonumber
\end{align}
Let $\F_{n,\delta}=\F'_\delta\cup \F'_{n,\delta}$. So, if $N\cap \F_{n,\delta}=\emptyset$, then all of \eqref{eq. vanishing l2 betti 1}, \eqref{eq. vanishing l2 betti 2}, \eqref{eq. small l2 betti} and \eqref{eq. bound l2 betti} hold, which yields the desired result.
\end{proof}

\subsection{Bounding finite subgroups}

Next, we refine the above approximation result to an equality by applying the Atiyah Conjecture. To achieve this, we proceed as follows:

\begin{enumerate}
    \item[(i)] Bound the order of finite subgroups arising in Dehn fillings (this will be addressed in the current subsection), and
    \item[(ii)] Apply the theory of virtually compact special groups to show that certain Dehn fillings satisfy the Atiyah Conjecture (this will be discussed in the following subsection).
\end{enumerate}

\begin{lemma}\label{lem. fin}
    Let $G$ be a group and let $N\lhd G$ be a normal subgroup. Then 
    \[\max\fin(G)\leqslant \max\fin(N)\cdot\max\fin(G/N).\]
\end{lemma}

\begin{proof}
    Let $K\leqslant G$ be a finite subgroup and let $\phi\colon G\rightarrow G/N$ be the natural quotient map. Then $|K|=|K\cap\ker \phi|\cdot |\phi(K)|\leqslant \max\fin(N)\cdot\max\fin(G/N)$.
\end{proof}

\begin{lemma}\label{lem. pre torsion-free}
    Let $G$ be a group that is hyperbolic relative to a subgroup $H$. Then there exists a finite subset $\F\subset H\smallsetminus\{1\}$ such that if a normal subgroup $N\lhd H$ satisfies $N\cap \F=\emptyset$, then 
    \[\max\fin(G/\ll N \rr)\leqslant \max\big(\fin(G)\cup\fin(H/N)\big).\]
\end{lemma}

\begin{proof}
    For every $N\lhd H$, let $\phi_N\colon G\rightarrow G/\ll N \rr$ be the natural quotient map. By \cite[Lemma 4.3]{dahmani2018recognizing}, there exists a finite subset $\fS\subset G\smallsetminus\{1\}$ and a collection $\mathcal{C}$ of finite subgroups of $G$ such that if a normal subgroup $N\lhd H$ satisfies $\ll N \rr\cap \fS=\emptyset$, then the following holds:

    \begin{enumerate}[label=($\ast$)]
        \item\label{item. control finite subgroup} for every finite subgroup $K\leqslant G/\ll N \rr$, there exists a subgroup $L\leqslant G$ such that $K\leqslant \phi_N(L)$ and $L$ either conjugates to $H$ or belongs to $\mathcal{C}$.
    \end{enumerate}

     By \cite[Theorem 1.1]{osin2007peripheral}, there exists a finite subset $\F\subset H\smallsetminus\{1\}$ such that if a normal subgroup $N\lhd H$ satisfies $N\cap \F=\emptyset$, then $\ll N \rr\cap \fS=\emptyset$ and $H\cap\ll N \rr=N$.

    Now let $N\lhd H$ be a normal subgroup such that $N\cap \F=\emptyset$, and let $K\leqslant G/\ll N \rr$ be a finite subgroup. Let $L$ and $\mathcal C$ be given by \ref{item. control finite subgroup}. If $L$ conjugates to $H$, then $K$ conjugates to a finite subgroup of $\phi_N(H)\cong H/N$, and thus $|K|\leqslant \max\fin(H/N)$. If $L\in\mathcal{C}$, then $|K|\leqslant |L|\leqslant \max\fin(G)$.
\end{proof}

\subsection{Compact special Dehn filling}

\begin{definition}
    A group $G$ is \textit{virtually compact cubical} if there is a finite-index subgroup $H\leqslant G$ such that $H$ is the fundamental group of a compact locally CAT(0) cube complex $X$.
    
    The group $G$ is \textit{virtually compact special} if in addition $X$ is special. Note that \cite{agol2016alternate, groves2022specializing} call such groups virtually special, while \cite{schreve2015l2} calls such groups virtually cocompact special.

    A group $G$ is \textit{virtually special} if there is a finite-index subgroup $H\leqslant G$ such that $H$ is the fundamental group of a special locally CAT(0) cube complex.
\end{definition}

Below, we collect some known facts about virtually (compact) special groups that will be used later.

\begin{proposition}\label{prop. properties of vir special groups} The following hold. 

\begin{enumerate}[label=(\roman*)]
\item \label{item: special} Any cover of a special cube complex is special. In particular, any subgroup of a virtually special group is virtually special.

    \item Every virtually compact cubical group is of type $F_\infty$ and has finite virtual cohomological dimension.

    \item Every finitely generated virtually special group is virtually locally indicable.

    \item Every finitely generated virtually special group is virtually residually finite rationally solvable (RFRS).

    \item (\cite{groves2022specializing}) Let $G$ be a virtually compact cubical group that is hyperbolic relative to a family of virtually abelian groups. Then $G$ is virtually compact special. In particular, virtually compact cubical hyperbolic groups are virtually compact special, which was first proved by \cite[Theorem 1.1]{agol2013virtual}.
\end{enumerate}
\end{proposition}

\begin{proof}\hspace{1cm}
    \begin{enumerate}[label=(\roman*)] 
\item Let $X$ be a special cube complex. By   \cite[Proposition 6.2]{wise2021structure},  $X$ admits a local-isometry to a Salvetti complex $S$. Let $Y$ be a cover of $X$. Then, the covering $Y\to X$, followed by the local-isometry $X\to S$, induces a local-isometry $Y\to S$. Thus, by \cite[Proposition 6.2]{wise2021structure}, $Y$ is special.

Let $G$ be a special group and $H$ a subgroup. Consider a finite-index subgroup $G_0\leq G$ that is a fundamental group of special cube complex $X$.   Let $Y$ be the cover of $X$ whose fundamental group is $H_0:=H\cap G_0$, which is a finite-index subgroup of $H$. Then, by the above, $Y$ is special.\\   
    
        \item Let $G$ be a virtually compact special group. Then $G$ has a finite-index subgroup $G_1$ that is the fundamental group of some compact locally CAT(0) cube complex. In particular, $G_1$ is of type $F$.\\

        \item Let $G$ be a finitely generated virtually special group. By \cite[Theorem 1.1]{haglund2008special}, $G$ has a finite-index subgroup $G_1$ that embeds into a finitely generated right-angled Artin group. By \cite[p. 387 and p.389]{duchamp1992lower}, every finitely generated right-angled Artin group is residually torsion-free nilpotent, and thus is residually $p$ for any prime $p$ by \cite[Theorem 2.1]{gruenberg1957residual}, and thus is bi-orderable by \cite{rhemtulla1973residually}, and thus is locally indicable by \cite[Corollary 1.4]{rhemtulla2002local}. So $G_1$ is locally indicable. \\

        \item Let $G$ be a finitely generated virtually special group. By (iii), $G$ has a finite-index subgroup $G_1$ that embeds into a finitely generated right-angled Artin group. By \cite[Corollary 2.3]{agol2008criteria}, finitely generated right-angled Artin groups are virtually RFRS, and thus so is the subgroup $G_1$.\\

        \item Let $\{H_i\}^n_{i=1}$ be a family of virtually abelian subgroups of $G$ such that $G$ is hyperbolic relative to $\{H_i\}^n_{i=1}$, let $X$ be a CAT(0) cube complex with a free geometric action of a finite-index normal subgroup $G_0\leqslant G$, and let $\{K_j\}^m_{j=1}$ be the peripheral structure of $G_0$ induced by $\{H_i\}^n_{i=1}$. As each $H_i$ is virtually abelian, so is each $K_j$, and thus by \cite[Remark 1.4]{groves2022specializing} the relatively hyperbolic pair $(G_0,\{K_j\}^\ell_{j=1})$ and the action $G_0\curvearrowright X$ satisfy Assumptions 1.1 and 1.2 of \cite{groves2022specializing}. Hence, by \cite[Theorem A]{groves2022specializing}, the group $G_0$, and therefore the group $G$, is virtually compact special.

        The last statement follows by noting that every hyperbolic group is hyperbolic relative to the trivial subgroup $\{1\}$.
    \end{enumerate}
\end{proof}

\begin{theorem}\label{thm. hyp.}
    Let $G$ be a virtually compact cubical hyperbolic group, and let $H\leqslant G$ be an almost malnormal quasi-convex subgroup. Then there exists a finite-index torsion-free normal subgroup $K\lhd H$ such that for every sufficiently deep normal subgroup $N\lhd H$ that satisfies 
    \begin{itemize}
        \item $N\leqslant K$, and
        \item $K/N$ is torsion-free virtually compact cubical hyperbolic,
    \end{itemize}
    we have
    \[b^{(2)}_1(G/\ll N \rr)\leqslant b^{(2)}_1(G).\]
\end{theorem}

\begin{remark}\label{rm. assumption differ}
    Note that in \Cref{thm. intro hyp} we assume that $G$ is hyperbolic relative to $H$, whereas here we assume that $H$ is an almost malnormal quasi-convex subgroup. These assumptions are equivalent, by \cite[Theorem 7.11]{bowditch2012relatively}.
\end{remark}

\begin{proof}
If $H$ is finite we can simply take $K=\{1\}$, and the theorem holds trivially. Below, we assume $|H|=\infty$. Then, by \cite[Theorem 7.11]{bowditch2012relatively}, $G$ is hyperbolic relative to $H$. For every $N\lhd H$, let $\phi_N\colon G\rightarrow G/\ll N \rr$ be the natural quotient map. 

\begin{claim}\label{cl. special} 
    There is a finite-index normal subgroup $K\lhd H$ such that if a normal subgroup $N\lhd H$ satisfies that $N\leqslant K$ and $\phi_N(H)$ is virtually compact cubical hyperbolic, then $G/\ll N \rr$ is virtually compact special hyperbolic.
\end{claim}

\begin{claimproof}
    The claim follows from \cite[Theorem 2.7]{agol2016alternate}. Note that \cite[Theorem 2.7]{agol2016alternate} requires both $G$ and $\phi_N(H)$ to be virtually compact special hyperbolic, but by \Cref{prop. properties of vir special groups}, this is equivalent to being virtually compact cubical and hyperbolic.
\end{claimproof}

Fix a $K\lhd H$ that satisfies \Cref{cl. special}. As $G$ is virtually compact cubical, it is virtually torsion-free. Let $G_1\lhd G$ be a torsion-free finite-index normal subgroup. By \Cref{prop. properties of vir special groups}, $G$ is virtually compact special, and thus is virtually locally indicable, again by \Cref{prop. properties of vir special groups}. By \Cref{cor. locally indicable dehn filling}, there exists a finite subset $\F_1\subset H\smallsetminus\{1\}$ such that if a normal subgroup $N\lhd H$ satisfies $N\cap \F_1=\emptyset$, then 
\begin{equation}\label{eq. hyp. 1}
    b^{(2)}_1(G/\ll N \rr)<b^{(2)}_1(G)+\dfrac{1}{\max\{|G/G_1|,|H/K|\}!}.
\end{equation}

By \cite[Theorem 1.1]{osin2007peripheral}, there exists a finite subset $\F_2\subset H\smallsetminus\{1\}$ such that if a normal subgroup $N\lhd H$ satisfies $N\cap \F_2=\emptyset$, then $\phi_N(H)\cong H/N$. Let $\F_3\subset H\smallsetminus\{1\}$ be the finite subset given by \Cref{lem. pre torsion-free}. Let 
\[\F_\delta=\F_1\cup \F_2\cup \F_3\]
and consider a normal subgroup $N\lhd H$ such that $N\leqslant K, N\cap \F_\delta=\emptyset$ and $K/N$ is torsion-free virtually compact cubical hyperbolic. As $\phi_N(K)\cong K/N$ is a finite-index subgroup of $\phi_N(H)\cong H/N$, we have that $\phi_N(H)$ is virtually compact cubical hyperbolic. \Cref{cl. special} implies that $G/\ll N \rr$ is virtually compact special, and thus satisfies the Atiyah Conjecture \cite{schreve2015l2}. Therefore, by \Cref{lem. pre torsion-free}
\begin{equation}\label{eq. 4}
    b^{(2)}_1(G/\ll N \rr)\in \dfrac{1}{\fin(G/\ll N \rr)}\cdot \Z\leqslant \dfrac{1}{(\max (\fin(G)\cup \fin(H/N)))!}\cdot \Z.
\end{equation}
As $G_1$ is a torsion-free normal subgroup of $G$, \Cref{lem. fin} implies 
\begin{equation}\label{eq. 5}
    \max\fin(G)\leqslant \max\fin(G_1)\cdot \max\fin(G/G_1)=|G/G_1|.
\end{equation}
Similarly, by considering $K/N$ in $H/N$ we have
\[\fin(H/N)\leqslant \fin(K/N)\cdot\fin(H/K)=|H/K|.\]
Combining these with \eqref{eq. 4}, we get
\begin{equation}\label{eq. hyp. 2}
    b^{(2)}_1(G/\ll N \rr)\in \dfrac{1}{\max\{|G/G_1|,|H/K|\}!}\cdot \Z.
\end{equation}
As $G$ is virtually compact special, it satisfies the Atiyah Conjecture \cite{schreve2015l2}, and thus

\begin{equation}\label{eq. l2 of big group}
    b^{(2)}_1(G)\in \dfrac{1}{\fin(G)}\cdot\Z\leqslant \dfrac{1}{\max\{|G/G_1|,|H/K|\}!}\cdot \Z
\end{equation}
by \eqref{eq. 5}.

The theorem follows from \eqref{eq. hyp. 1}, \eqref{eq. hyp. 2}, and \eqref{eq. l2 of big group}.
\end{proof}

\begin{remark}\label{rm. why finite-index}
    In the proof of \Cref{thm. hyp.}, the passage from $H$ to the finite-index subgroup $K$ is a consequence of the Malnormal Special Quotient Theorem, which guarantees that the Dehn filling quotient is virtually compact special, which in turn guarantees that the quotient satisfies the Atiyah Conjecture. If the Atiyah Conjecture is preserved by sufficiently deep Dehn fillings (i.e., whenever $G$ is a group that satisfies the Atiyah Conjecture and has a hyperbolically embedded subgroup $H\hookrightarrow_h G$, the quotient $G/\ll N \rr$ satisfies the Atiyah Conjecture for sufficiently deep $N\lhd H$), one can assume $K=H$.
\end{remark}

Our next result concerns groups hyperbolic relative to abelian subgroups. The lemma below follows by the work of Groves--Manning \cite{groves2022specializing} and \cite{wise2021structure}.

\begin{lemma}\label{lem. special quotient}
    Let $G$ be a virtually compact cubical group that is hyperbolic relative to a finite family of abelian subgroups $\{H_i\}^k_{i=1}$. Then there exists a family of finite-index torsion-free subgroups $\{P_i\lhd H_i\}^k_{i=1}$ such that for every family of subgroups $\{N_i\leqslant P_i\}^k_{i=1}$ that satisfies $P_i/N_i$ is virtually $\Z$ for all $i$, the group $G/\ll \bigcup^k_{i=1} N_i \rr$ is virtually compact special.
\end{lemma}

\begin{proof}
By \Cref{prop. properties of vir special groups}, the group $G$ is virtually compact special. The desired result then follows from \cite[Theorem 15.6]{wise2021structure}. Note that \cite[Theorem 15.6]{wise2021structure} is stated for virtually sparse special groups, but virtually compact special groups are virtually sparse special (see, e.g., \Cref{rem:sparse compact}).
\end{proof}

\begin{lemma}\label{lem. zn}
    Let $n\geqslant 2$. Then for any finite subset $\F\subset \Z^n\smallsetminus\{0\}$, there exists a subgroup $N\leqslant \Z^n$ such that $N\cap \F=\emptyset$ and $\Z^n/N\cong \Z$.
\end{lemma}
\begin{proof}
    Let $e_1,\dots,e_n$ be a basis of the free abelian group $\Z^n$. Let $p_1,\dots,p_n$ be mutually distinct primes such that the subgroup $G$ of $\Z^n$ generated by $\{p_ie_i\mid i=1,\dots,n\}$
    satisfies $G\cap \F=\emptyset$.  Consider the natural quotient homomorphisms
    \[\alpha\colon \Z^n\rightarrow \Z^n/G\cong \prod^n_{i=1}(\Z/p_i\Z)\cong \Z/(p_1\cdots p_n\cdot\Z)\]
    and
    \[\beta\colon \Z\rightarrow \Z/(p_1 \cdots p_n\cdot\Z).\]
    As $\Z^n$ is free abelian, there exists a homomorphism $\gamma\colon \Z^n\rightarrow \Z$ such that $\beta\circ\gamma=\alpha$.
    Let $N=\ker\gamma$. Then $N\cap \F\subset \ker\alpha\cap \F=\emptyset$.
\end{proof}

We now state our next main result.

\begin{theorem}\label{thm. rel hyp.}
    Let $G$ be a virtually compact cubical group that is hyperbolic relative to an  abelian subgroup $H\leqslant G$. Then there exists a finite-index torsion-free subgroup $K\leqslant H$ such that for every sufficiently deep subgroup $N\leqslant H$ that satisfies $N\leqslant K$ and $K/N$ is torsion-free, we have for all $n$,
    \begin{equation}\label{eq. rel hyp.}
        b^{(2)}_n( G/\ll N \rr)=b^{(2)}_n (G).
    \end{equation}
\end{theorem}

\begin{remark}\label{rm. trivial}
    In \Cref{thm. rel hyp.}, if $H$ is finite, we can choose $K$ to be $\{1\}$. If $H$ is virtually $\Z$, we can choose the group $K$ to be an infinite cyclic subgroup of $H$ and the finite set $\F$ to consist of just a generator of $K$. In this case, the only subgroup $N\leqslant K$, that satisfies both $N\cap \F=\emptyset$ and $K/N$ is torsion-free, is $\{1\}$. The desired conclusion clearly follows.
\end{remark}

\begin{proof}
    Let $G_1\lhd G$ be a finite-index normal subgroup that is the fundamental group of some compact locally CAT(0) cube complex, and let $\{H_i\}_{i\in I}$ be the induced peripheral structure on $G_1$. The main idea of the proof is to apply \Cref{thm. li dehn filling}, which, by itself, only gives \eqref{eq. rel hyp.} for finitely many $n$. The claim below takes care of all other values of $n$.

    \begin{claim}\label{cl. beyong vcd}
        There exists a finite subset $\F_1\subset H\smallsetminus\{1\}$ such that if a subgroup $N\leqslant G_1\cap H$ satisfies that $N\cap\F_1=\emptyset$ and $|H/N|=\infty$, then
        \[b^{(2)}_n(G/\ll N \rr_G)=0,\]
        for all $n>\vcd(G)+1$.
    \end{claim}

    \begin{claimproof}
    By \cite[Theorem 1 (1)]{osin2007peripheral} and \cite[Corollary 4.6]{sun2019cohomologyii}, there exists a family of finite subsets $\{\F_{1i}\subset H_i\smallsetminus\{1\}\}_{i\in I}$ such that whenever a family of groups $\{N_i\leqslant H_i\}_{i\in I}$ satisfies that  $N_i\cap \F_{1i}=\emptyset$ for all $i\in I$, then the natural maps $H_i/N_i\rightarrow \bg_1$ are all injective, where 
    \[\bg_1:=G_1/\ll \bigcup_{i\in I}N_i \rr_{G_1},\]
   and for all $n>\cd(G_1)+1=\vcd(G)+1$, by \cite[Corollary 4.6 (a)]{sun2019cohomologyii}, we have 

 \begin{align*}
     \cH_n\big(\bg_1\; ;\N(\bg_1)\big) & \cong \bigoplus_{i\in I} H_n\big(H_i/N_i\; ;\N(\bg_1)\big)\\
     & \cong \bigoplus_{i\in I}\cH_n\big({\bg_1}\times_{H_i/N_i} E(H_i/N_i)\; ;\N(\bg_1)\big).
 \end{align*}
as $\N(\bg_1)$-modules. By \cite[Theorem 6.54 (7)]{luck2002l2}, we get
\begin{equation}\label{eq. above vcd}
        b^{(2)}_n(\bg_1)=\sum_{i\in I}b^{(2)}_n \big({\bg_1}\times_{H_i/N_i} E(H_i/N_i)\; ;\N(\bg_1)\big)=\sum_{i\in I}b^{(2)}_n(H_i/N_i)
    \end{equation}
    for $n>\vcd(G)+1$.

    We now choose the family $\{N_i\leqslant H_i\}_{i\in I}$. Let $\F_1\subset H\smallsetminus\{1\}$
    be the subset given by \Cref{lem. df in fi subgroup} \ref{item. deep} with respect to the family $\{\F_{1i}\}_{i\in I}$, let $N\leqslant G_1\cap H$ be a subgroup that satisfies $N\cap \F_1=\emptyset$ and $|H/N|=\infty$, and let $\{N_i\leqslant H_i\}_{i\in I}=\varphi^{-1}(\{N\})$, where $\varphi$ is the map given by \Cref{lem. df in fi subgroup}. Then
    \[b^{(2)}_n\big(G_1/\ll \bigcup_{i\in I}N_i \rr_{G_1}\big)=\sum_{i\in I}b^{(2)}_n(H_i/N_i)=\sum_{i\in I}b^{(2)}_n\big((G_1\cap H)/N\big)=0\]
    for all $n>\vcd(G)+1$, by \cite[Theorem 6.37]{luck2002l2} and the fact that $(G_1\cap H)/N$ is infinite abelian,  being a finite-index subgroup of $H/N$.

    By \Cref{lem. df in fi subgroup} \ref{item. embedding}, $G_1/\ll \bigcup_{i\in I}N_i \rr_{G_1}$ is a finite-index subgroup of $G/\ll N \rr_G$, and thus by \cite[Theorem 6.54 (6)]{luck2002l2} and \Cref{lem. df in fi subgroup}, 
    \begin{align*}
        &b^{(2)}_n\big(G/\ll N \rr_G\big)\\
        =&b^{(2)}_n\Big(E(G/\ll N \rr_G)\; ;G/\ll N \rr_G\Big)\\
        =&\dfrac{1}{[G/\ll N \rr_G:G_1/\ll \bigcup_{i\in I}N_i \rr_{G_1}]}\cdot b^{(2)}_n\Big(E(G/\ll N \rr_G)\; ;G_1/\ll \bigcup_{i\in I}N_i \rr_{G_1}\Big)\\
        =&\dfrac{1}{[G/\ll N \rr_G:G_1/\ll \bigcup_{i\in I}N_i \rr_{G_1}]}\cdot b^{(2)}_n\big(G_1/\ll \bigcup_{i\in I}N_i \rr_{G_1}\big)\\
        =&0.\qedhere
    \end{align*}  
    \end{claimproof}
    
    We now apply \Cref{thm. li dehn filling} along with the Atiyah Conjecture. By \Cref{prop. properties of vir special groups}, the group $G$ is virtually compact special, and thus is virtually locally indicable, again by \Cref{prop. properties of vir special groups}. By \Cref{thm. li dehn filling}, we have the following:

    \begin{claim}\label{cl. rel. hyp. approximation}
        There exists a finite subset $\F_2\subset H\smallsetminus\{1\}$ such that if a subgroup $N\lhd H$ satisfies $N\cap \F_2=\emptyset$, then for all $n\leqslant \vcd(G)+1$, 
    \begin{align*}
        & b^{(2)}_n(G)-b^{(2)}_n(H)-b^{(2)}_{n-1}(H/N)-\dfrac{1}{\max\{|G/G_1|,|H/K|\}!}\\
        & < b^{{2}}_n\big(G/\ll N \rr_G\big)\\
        & < b^{(2)}_n(G)+b^{(2)}_{n-1}(H)+b^{(2)}_n(H/N)+\dfrac{1}{\max\{|G/G_1|,|H/K|\}!}.
    \end{align*}
    \end{claim}

    Now, let $K\leqslant H\cap G_1$ be the finite-index torsion-free subgroup given by \Cref{lem. special quotient}. If $|K/N|=\infty$, then $H/N$ is infinite abelian. So by \cite[Theorem 1.44]{luck2002l2}, we have the following:
    \begin{claim}\label{cl. rel. hyp. vanishing}
        If $|K/N|=\infty$, then for all $n\in\mathbb N$, 
        \[b^{(2)}_n(H)=b^{(2)}_{n-1}(H)=b^{(2)}_n(H/N)=b^{(2)}_{n-1}(H/N)=0.\]
    Therefore,
    \begin{equation*}
        |b^{{2}}_n(G/\ll N \rr_G)-b^{(2)}_n(G)|<\dfrac{1}{\max\{|G/G_1|,|H/K|\}!}
    \end{equation*}
    whenever $N\cap \F_2=\emptyset$ and $n\leqslant \vcd(G)+1$.
    \end{claim}

    By \Cref{lem. fin}, we have the following:

    \begin{claim}\label{cl. rel. hyp. finite subgroup H}
        If a subgroup $N\leqslant K$ satisfies that $K/N$ is torsion-free, then
        \[\max\fin(H/N)\leqslant \max\fin(K/N)\cdot\max\fin(H/K)=|H/K|.\]
    \end{claim}

    We will also need the following:
    
    \begin{claim}\label{cl. rel. hyp. finite subgroup G}
        There exists a finite subset $\F_3\subset H\smallsetminus\{1\}$ such that if a subgroup $N\leqslant K$ satisfies that $N\cap \F_3=\emptyset$ and $K/N$ is torsion-free, then
        \[\max\fin(G/\ll N \rr_G)\leqslant \max\{|G/G_1|,|H/K|\}.\]
    \end{claim}

    \begin{claimproof}
    By \Cref{lem. fin}, we have
    \begin{equation}\label{eq. 6}
        \max\fin(G)\leqslant \max\fin(G_1)\cdot \max\fin(G/G_1)=|G/G_1|.
    \end{equation}
    
    By \Cref{lem. pre torsion-free}, there exists a finite subset $\F_3\subset H\smallsetminus\{1\}$ such that if a subgroup $N\leqslant K$ satisfies that $N\cap \F_3=\emptyset$ and $K/N$ is torsion-free, then
    \begin{equation}\label{eq. 61}
        \max\fin(G/\ll N \rr_G)\leqslant \max(\fin(G)\cup\fin(H/N)).
    \end{equation}

    The claim follows from \Cref{cl. rel. hyp. finite subgroup H}, \eqref{eq. 6} and \eqref{eq. 61}.
    \end{claimproof}

    For any subgroup $N\leqslant H$, let $\phi_N\colon G\rightarrow G/\ll N \rr_G$ be the natural quotient map. The following claim follows from \cite[Theorem 1.1]{osin2007peripheral}.
    
    \begin{claim}\label{cl. rel. hyp. rh}
        There exists a finite subset $\F_4\subset H\smallsetminus\{1\}$ such that if a subgroup $N\leqslant H$ satisfies $N\cap \F_4=\emptyset$, then $\phi_N(H)\cong H/N$ and $G/\ll N \rr_G$ is hyperbolic relative to $\phi_N(H)$.
    \end{claim}

    Since by  \Cref{prop. properties of vir special groups} (ii), $G$ is of type $F_\infty$, by \Cref{cor. finiteness conditions} (i), there is a finite subset $\F_5\subset H\smallsetminus\{1\}$ such that if a subgroup $N\leqslant H$ satisfies $N\cap \F_5=\emptyset$, then  $G/\ll N \rr_G$ is of type $F_\infty$. Let
    \[\F=\F_1\cup \F_2\cup \F_3\cup \F_4\cup \F_5.\]
    By \Cref{rm. trivial}, we may assume that $|H|=\infty$. So, by enlarging $\F_1$, we may assume that $\F\cap K\neq \emptyset$. We have the following:

    \begin{enumerate}[label=($\ast$)]
        \item\label{item. K/N neq 1} If a subgroup $N\leqslant K$ satisfies $N\cap \F=\emptyset$, then $K/N\neq\{1\}$.
    \end{enumerate}

    Now consider a subgroup $N\leqslant K$ such that $N\cap \F=\emptyset$ and $K/N$ is torsion-free.

    \begin{claim}\label{cl. K/N 2}
        \[b^{(2)}_i(G/\ll N \rr_G)\in \dfrac{1}{\max\{|G/G_1|,|H/K|\}!}\cdot \Z.\]
    \end{claim}

    \begin{claimproof}
        By \ref{item. K/N neq 1}, the group $K/N$ is infinite finitely generated free abelian. If $K/N \cong \Z$, then \Cref{lem. special quotient} implies that $G/\ll N \rr_G$ is virtually compact special, and thus satisfies the Atiyah Conjecture \cite{schreve2015l2}. In particular,
        \[b^{(2)}_i(G/\ll N \rr_G)\in \dfrac{1}{\fin(G/\ll N \rr_G)}\cdot \Z \subset \dfrac{1}{\max\{|G/G_1|,|H/K|\}!}\cdot \Z,\]
        by \Cref{cl. rel. hyp. finite subgroup G}.
    
        Below, we assume that $K/N$ is free abelian of finite rank at least $2$. We will apply \Cref{lem. det and atiyah ver 2} to $G/\ll N \rr_G$, with $C=\max\{|G/G_1|,|H/K|\}$ and $X$ being a model for $E(G/\ll N \rr_G)$ of finite type. Such a model exists by our choice of $\F_5$. Let $\fS\subset G/\ll N \rr_G\smallsetminus\{1\}$ be any finite subset. We need to find a normal subgroup of $G/\ll N \rr_G$ satisfying the assumptions in \Cref{lem. det and atiyah ver 2}. By \Cref{cl. rel. hyp. rh}, \Cref{thm. cl} and \cite[Theorem 1.1]{osin2007peripheral}, there is a finite subset $\F^{\prime}\subset \phi_N(H)\smallsetminus\{1\}$ such that if a subgroup $\overline M\leqslant \phi_N(H)$ satisfies $\overline M\cap \F^{\prime}=\emptyset$, then $\ll \overline M \rr_{G/\ll N \rr_G}\cap \fS=\emptyset$ and $(G/\ll N \rr_G, \phi_N(H), \overline M)$ is a Cohen--Lyndon triple. By enlarging $\F^{\prime}$, we may assume 
        \begin{equation}\label{eq. F' contain F}
            \phi_N(\F)\subset \F^{\prime}
        \end{equation}
        
        By assumption, $\phi_N(K)\cong K/N$ is free abelian of rank at least $2$. Use \Cref{lem. zn} to choose a subgroup $\overline M_{\fS}\leqslant \phi_N(K)$ such that $\overline M_{\fS}\cap \F^{\prime}=\emptyset$ and $\phi_N(K)/\overline M_{\fS}\cong \Z$. Then $\ll \overline M_{\fS} \rr_{G/\ll N \rr_G}\lhd G/\ll N \rr_G$ and $\ll \overline M_{\fS} \rr_{G/\ll N \rr_G}\cap \fS=\emptyset$.

        Let $ M_{\fS}\leqslant K$ be the preimage of $\overline M_{\fS}\leqslant \phi_N(K)$ under the restriction $\phi_N|_K$. Then $M_{\fS}\cap \F=\emptyset$ and thus all other claims in the current proof can be applied to $M_{\fS}$. Note that $K/ M_{\fS}\cong \phi_N(K)/\overline M_{\fS} \cong \Z$. By \Cref{lem. special quotient}, we have that 
        \[(G/\ll N \rr_G)/\ll \overline M_{\fS}\rr_{G/\ll N \rr_G}\cong G/\ll M_{\fS} \rr_G\]
        is virtually compact special. In particular, $(G/\ll N \rr_G)/\ll M_{\fS}\rr_{G/\ll N \rr_G}$ satisfies the Atiyah Conjecture \cite{schreve2015l2}. Another consequence of the virtual compact specialness of $(G/\ll N \rr_G)/\ll \overline M_{\fS}\rr_{G/\ll N \rr_G}$ is that $(G/\ll N \rr_G)/\ll \overline M_{\fS}\rr_{G/\ll N \rr_G}$ is virtually locally indicable by \Cref{prop. properties of vir special groups}. By \Cref{thm. cl} and our choice of $\F'$, the triple $(G/\ll N \rr_G, \phi_N(H),\overline M_{\fS})$ has the Cohen--Lyndon property, which implies that $\ll \overline M_{\fS}\rr_{G/\ll N \rr_G}$ is a free product of groups isomorphic to $\overline M_{\fS}$. In particular, $\ll \overline M_{\fS}\rr_{G/\ll N \rr_G}$ is locally indicable. From the short exact sequence 
        \[1\rightarrow \ll \overline M_{\fS}\rr_{G/\ll N \rr_G} \rightarrow G/\ll N \rr_G \rightarrow (G/\ll N \rr_G)/\ll \overline M_{\fS}\rr_{G/\ll N \rr_G}\rightarrow 1\]
        we conclude that $G/\ll N \rr_G$ is virtually locally indicable.

        By \Cref{cl. rel. hyp. finite subgroup G}, we have
        \[\max\fin(G/\ll M_{\fS} \rr_G)\leqslant \max\{|G/G_1|,|H/K|\}.\]
        The above shows that for every finite subset $\fS\subset G/\ll N \rr_G\smallsetminus\{1\}$, there is a normal subgroup $\ll \overline M_{\fS} \rr_{G/\ll N \rr_G}\lhd G/\ll N \rr_G$ satisfying the conditions in \Cref{lem. det and atiyah ver 2} with $C=\max\{|G/G_1|,|H/K|\}$, which then yields the desired result.
    \end{claimproof}
    
    The theorem then follows from \Cref{cl. rel. hyp. vanishing} and \Cref{cl. K/N 2}.
\end{proof}

\begin{remark}\label{rm. torsion-free necessary}
    The torsion-free condition in \Cref{thm. intro hyp} and \Cref{thm. intro rel hyp} cannot be dropped, as shown by the following example. Let $G'$ be the fundamental group of a hyperbolic $3$-manifold and let $H$ be a maximal cyclic subgroup of $G'$. Let $G=G'\ast_{H=H}G'$ be the amalgam of two copies of $G'$ glued via the identity map on $H$. Then by \cite[Theorem 7.11]{bowditch2012relatively}, $G'$ is hyperbolic relative to $H$. So by \cite[Theorem 0.1 (3)]{dahmani2003combination}, $G$ is hyperbolic relative to $H$. As $H$ is hyperbolic, so is $G$ by \cite[Corollary 2.41]{osin2006relatively}. Note that $G'$ is virtually compact cubical by \cite[Theorem 5.4]{afw2015three}. Thus, by \cite[Theorem 13.1]{wise2021structure}, $G$ is virtually compact cubical. So the pair of groups $G$ and $H$ satisfy the assumptions of both \Cref{thm. intro hyp} and \Cref{thm. intro rel hyp}.

    Using a Mayer--Vietoris sequence argument, we get that $b^{(2)}_1(G)=0$. Let $h$ be a generator of $H$ and let $n\in\mathbb N$ be a large integer. By \cite[Theorem 1.1]{osin2007peripheral}, the group $G'/\ll t^n \rr_{G'}$ is infinite and contains $H/\langle t^n \rangle$ as a subgroup. Then $G/\ll t^n\rr_G$ is the amalgam of two copies of $G'/\ll t^n \rr_{G'}$ along their common subgroup $H/\langle t^n\rangle$. By \cite[Lemma 3.1]{peterson2011group}, we have $b^{(2)}_1(G/\ll t^n \rr_G)\geqslant 1/n$, and so $G/\ll t^n \rr_G$ does not satisfy the conclusions of \Cref{thm. intro hyp} and \Cref{thm. intro rel hyp}.
\end{remark}

\section{Main results for multiple peripherals}\label{sec. multiple peripherals}

    None of the results involved in the proof of \Cref{thm. hyp.} and \Cref{thm. rel hyp.} are confined to the case of a single peripheral subgroup. Below, we state their general versions. Their proofs are completely analogous and therefore left to the reader.

    \subsection{Dehn filling space}

    Let $G$ be a group, let $\{H_\lambda\}_{\lambda\in\Lambda}$ be a family of subgroups of $G$, and let $\{N_\lambda\lhd H_\lambda\}_{\lambda\in\Lambda}$ be a family of normal subgroups. For the rest of this subsection, we assume that the triple $(G,\{H_\lambda\}_{\lambda\in\Lambda},\{N_\lambda\}_{\lambda\in\Lambda})$ is a \textit{Cohen--Lyndon triple}, i.e., for each $\lambda$ there exists a family of left transversal of the subgroup $T_\lambda$ of $H_\lambda\ll \bigcup_{\lambda\in\Lambda} N_\lambda \rr$ in $G$ such that 
\[\ll \bigcup_{\lambda\in\Lambda} N_\lambda \rr = \Asterisk_{\lambda\in\Lambda,t\in T_\lambda}tN_\lambda t^{-1}.\]
Let $\bg:=G/\ll \bigcup_{\lambda\in\Lambda} N_\lambda \rr$, and let $\bh_\lambda:=H_\lambda/N_\lambda$ for each $\lambda$. Note that the natural maps $\bh_\lambda\rightarrow \bg$ are injective (see e.g., \cite[Lemma 6.4]{sun2018cohomologyi}), and below we will view each $\bh_\lambda$ as a subgroup of $\bg$. Let $BG$ (resp. $BH_\lambda, B\bh_\lambda$) be a $K(G,1)$ (resp. $K(H_\lambda,1),K(\bh_\lambda,1)$) space that has the homotopy type of a CW complex.

For each $\lambda$, there is a classifying map $\phi_\lambda\colon BH_\lambda\rightarrow BG$ (resp. $\psi_\lambda\colon BH_\lambda\rightarrow B\bh_\lambda$) induced by the inclusion $H_\lambda\hookrightarrow G$ (resp. the quotient map $q\colon H_\lambda\rightarrow \bh_\lambda$). Let
\[\phi\colon \bigsqcup_{\lambda\in \Lambda} BH_\lambda\rightarrow BG\]
be the map induced by $\phi_\lambda,\lambda\in\Lambda$, and let
\[\psi\colon \bigsqcup_{\lambda\in \Lambda} BH_\lambda\rightarrow \bigsqcup_{\lambda\in \Lambda} B\bh_\lambda\]
be the disjoint union of $\psi_\lambda,\lambda\in \Lambda$. Let $X$  be the complex obtained by gluing all the mapping cylinders $M_\phi$ and $M_\psi$ along the common subcomplex $\bigsqcup_{\lambda\in \Lambda} BH_\lambda$. We call $X$ a {\it Dehn filling space} associated to the triple  $(G,\{H_\lambda\}_{\lambda\in\Lambda},\{N_\lambda\}_{\lambda\in\Lambda})$.

\begin{theorem}
Let $(G,\{H_\lambda\}_{\lambda\in\Lambda},\{N_\lambda\}_{\lambda\in\Lambda})$ be a Cohen--Lyndon triple. Then every Dehn filling space associated to the triple $(G,\{H_\lambda\}_{\lambda\in\Lambda},\{N_\lambda\}_{\lambda\in\Lambda})$ is a $K(G/\ll \bigcup_{\lambda\in \Lambda} N_\lambda\rr,1)$-space.
\end{theorem}

\begin{corollary}\label{cor. gen topology}
Let $G$ be a group with a family of hyperbolically embedded subgroups $\{H_\lambda\}_{\lambda\in\Lambda}\hookrightarrow_h G$. Then for sufficiently deep $\{N_\lambda\lhd H_\lambda\}_{\lambda\in\Lambda}$, every Dehn filling space associated to the triple $(G,\{H_\lambda\}_{\lambda\in\Lambda},\{N_\lambda\}_{\lambda\in\Lambda})$ is a $K(G/\ll \bigcup_{\lambda\in \Lambda} N_\lambda\rr,1)$-space.
\end{corollary}

\begin{corollary}\label{cor. gen finiteness conditions}
Let $G$ be a group with a finite family hyperbolically embedded subgroups $\{H_i\}^k_{i=1}\hookrightarrow_h G$.
\begin{enumerate}
    \item[(i)] If $G$ is of type $F_n$ for some $n\in \mathbb N^{+}\cup \{\infty\}$, then there exists a family of finite subsets $\{\F_i\subset H_i\smallsetminus\{1\}\}^k_{i=1}$ such that if, for every $i$, we have $N_i\cap\F_i=\emptyset$, then $G/\ll \bigcup^k_{i=1} N_i\rr$ is of type $F_n$  if and only if all $H_i/N_i$ are of type $F_n$. 

    \item[(ii)] If $G$ and every $H_i$ are of type $F$, then  then there exists a family of finite subsets $\{\F_i\subset H_i\smallsetminus\{1\}\}^k_{i=1}$ such that if for every $i$, we have $N_i\cap\F_i=\emptyset$ and $H_i/N_i$ is of type $F$, then $G/\ll \bigcup^k_{i=1} N_i\rr$ is of type $F$.
\end{enumerate}

\end{corollary}

    \subsection{$L^2$-Betti numbers} The following results are generalizations of our main results on $L^2$-Betti numbers in the presence of multiple peripheral subgroups.

    \begin{theorem}\label{thm. gen li}
 Let $k\geqslant 0$ and let $G$ be a type $F_{k+1}$ virtually locally indicable group with a finite family of hyperbolically embedded subgroups $\{H_i\}^\ell_{i=1}\hookrightarrow_h G$. Then for every $\delta>0$, there exists a family of finite subsets $\{\F_{i,\delta}\subset H_i\smallsetminus\{1\}\}^\ell_{i=1}$ such that if a family of normal subgroups $\{N_i\lhd H_i\}^\ell_{i=1}$ satisfies $N_i\cap \F_{i,\delta}=\emptyset$ and $H_i/N_i$ is of type $F_{k+1}$ for all $i$, then $\bg:=G/\ll \bigcup^\ell_{i=1} N_i \rr$ satisfies
    \[b^{(2)}_1(\bg)<b^{(2)}_1(G)+\delta,\]
    \begin{equation*}
        b^{(2)}_{k+1}(\bg)<b^{(2)}_{k+1}(G)+\sum^\ell_{i=1} b^{(2)}_k(H_i)+ \sum^\ell_{i=1}b^{(2)}_{k+1}(H_i/N_i)+\delta,
    \end{equation*}
    and
    \begin{equation*}
       -\sum^\ell_{i=1}b^{(2)}_n(H_i)-\sum^\ell_{i=1}b^{(2)}_{n-1}(H_i/N_i)-\delta< b^{(2)}_n(\bg)- b^{(2)}_n(G)< \sum^\ell_{i=1}b^{(2)}_{n-1}(H_i)+\sum^\ell_{i=1}b^{(2)}_n(H_i/N_i)+\delta
    \end{equation*}
    for all $n\leqslant k$.
    \end{theorem}

    \begin{corollary}\label{cor. gen li}
    Let $G$ be a virtually locally indicable group of type $F_\infty$ that has a finite family of hyperbolically embedded amenable subgroups $\{H_i\}^m_{i=1}\hookrightarrow_h G$. Then for every $\delta>0$ and $n\geq 0$, there exists a family of finite subsets $\{\F_{n,\delta,i}\subset H_i\smallsetminus\{1\}\}^m_{i=1}$ such that if a family of normal subgroups $\{N_i\lhd H_i\}^m_{i=1}$ satisfies that, for all $i$, $N_i\cap\F_{n,\delta,i}=\emptyset$ and $H_i/N_i$ is of type $F_\infty$, then
    \[\left| b^{(2)}_n(G/\ll \bigcup^m_{i=1} N_i \rr)-b^{(2)}_n(G)\right|< \delta.\]
    \end{corollary}

    \begin{remark}
        One can prove \Cref{cor. gen li} by proceeding as follows: First, prove \Cref{cor. gen li} under the additional assumption that each $H_i$ is infinite. The proof will be a generalization of and completely analogous to the proof of \Cref{thm. intro li}. Now consider the general case. Suppose that there exists $\ell<k$ such that $|H_i|<\infty$ for $i=\ell+1,\dots, k$ and $|H_i|=\infty$ for $i=1,\dots,\ell$. We can then let $\F_i=H_i\smallsetminus\{1\}$ for $i=\ell+1,\dots,k$, from which it follows that $N_i=\{1\}$ for $i=\ell+1,\dots,k$. Moreover, it follows from the definition that, if $X\subset G$ is a subset such that $\{H_i\}^k_{i=1}\hookrightarrow_h (G,X)$, then $\{H_i\}^\ell_{i=1}\hookrightarrow_h (G,X\cup \bigcup^k_{i=\ell+1}H_i)$. In particular, $\{H_i\}^\ell_{i=1}\hookrightarrow_h G$. So we are reduced to the previously solved special case where the hyperbolically embedded family consists of only infinite amenable groups.
    \end{remark}

    \begin{theorem}\label{thm. gen hyp.}
    Let $G$ be a virtually compact cubical hyperbolic group that is hyperbolic relative to a finite family of subgroups $\{H_i\}^k_{i=1}<G$. Then there exists a family of finite-index torsion-free normal subgroups $\{K_i\lhd H_i\}^k_{i=1}$ such that the following ($\bigstar$) holds.

    \noindent
    ($\bigstar$) For each $i$, let $L_i\leqslant K_i$ be any finite-index subgroup. Then there exists a family of finite subsets $\{\F_i\subset L_i\smallsetminus\{1\}\}^k_{i=1}$ such that if a family of normal subgroups $\{N_i\lhd H_i\}^k_{i=1}$ satisfies for all $i$,
    \begin{itemize}
        \item $N_i\leqslant L_i$, 
        \item $N_i\cap\F_i=\emptyset$, and
        \item $L_i/N_i$ is torsion-free virtually compact cubical hyperbolic,
    \end{itemize}
    then
    \[b^{(2)}_1(G/\ll \bigcup^k_{i=1} N_i \rr)\leqslant b^{(2)}_1(G).\]
\end{theorem}

    \begin{theorem}\label{thm. gen rel hyp.}
    Let $G$ be a virtually compact cubical group that is hyperbolic relative to a finite family of abelian  subgroups $\{H_i\}^k_{i=1}$. Then there exists a family of finite-index torsion-free subgroups $\{K_i\lhd H_i\}^k_{i=1}$ such that the following ($\bigstar$) holds.

    \noindent
    ($\bigstar$) For each $i$, let $L_i\leqslant K_i$ be any finite-index subgroup. Then there exist finite subsets $\{\F_i\subset L_i\smallsetminus\{1\}\}^k_{i=1}$ such that if a family of normal subgroups $\{N_i\lhd H_i\}^k_{i=1}$ satisfies for all $i$,
    \begin{enumerate}
        \item[(i)] $N_i\leqslant L_i$,
        \item[(ii)] $N_i\cap \F_i=\emptyset$, and
        \item[(iii)] $L_i/N_i$ is torsion-free,
    \end{enumerate}
    then, for all $n\in\mathbb N$,
    \[b^{(2)}_n\big( G/\ll \bigcup^k_{i=1} N_i \rr\big)=b^{(2)}_n (G).\]
\end{theorem}

\begin{remark}
    \Cref{thm. gen rel hyp.} can be proved by proceeding along the following line: First prove the theorem under the additional assumption that each $H_i$ is virtually $\Z^m$ for $m\geqslant 2$. The proof will be a generalization of and completely analogous to the proof of \Cref{thm. rel hyp.}. Now consider the general case. Suppose that there exists $\ell<k$ such that, for $i=\ell+1,\dots,k$, the group $H_i$ is virtually cyclic; and for $i=1,\dots,\ell$, the group $H_i$ is virtually $\Z^m$ for some $m\geqslant 2$. Then by \cite[Theorem 2.40]{osin2006relatively}, $G$ is hyperbolic relative to the family $\{H_i\}^\ell_{i=1}$. Moreover, if for some $i$, the group $H_i$ is finite, then we can take $K_i=\{1\}$, from which it follows that $N_i=\{1\}$. And if for some $i$, the group $H_i$ is virtually $\Z$, then we can choose $K_i$ to be isomorphic to $\Z$. In this case any finite-index subgroup $L_i\leqslant K_i$ will be isomorphic to $\Z$ and we can let $\F_i$ consist of the generator of $L_i$. The only $N_i$ that satisfies conditions (i), (ii) and (iii) in \Cref{thm. gen rel hyp.} is $\{1\}$.  Therefore, $N_i=\{1\}$ for $i=\ell+1,\dots,k$ and we are reduced to the previously solved special case where $G$ is hyperbolic relative to the family $\{H_i\}^\ell_{i=1}$ with each $H_i$  virtually $\Z^m$ for some $m\geqslant 2$.
\end{remark}

\section{Virtually sparse special groups}\label{sec:sparse}

We will study $L^2$-Betti numbers of Dehn fillings of cusped hyperbolic manifolds. For this purpose, \Cref{thm. rel hyp.} is not satisfactory because it only deals with virtually compact cubical groups, and, to the best of the authors' knowledge, only a finite number of commensurability classes of cusped hyperbolic manifolds of dimension at least 4 with virtually compact cubical fundamental group have been constructed so far. So, in this section, we will generalize \Cref{thm. rel hyp.} to deal with multiple cusps and virtually sparse special  groups.  Notably, Bergeron--Wise \cite{bergeron2012boundary} proved that all cusped arithmetic hyperbolic manifolds have virtually  sparse special fundamental group (see \Cref{thm. lattice sparse}). We start with the definition of a sparse group. For more details on this topic, we refer to \cite{wise2021structure}.

\begin{definition}
    A \textit{quasiflat} is a pair $(X,A)$, where $A$ is a finitely generated virtually abelian group and $X$ is a locally finite CAT(0) cube complex with a proper action by $A$ such that there are finitely many $A$-orbits of hyperplanes.
\end{definition}

\begin{definition}
    A quasiflat $(X,A)$ is \textit{closable} if there exists a quasiflat $(Y,B)$ such that $B$ acts properly and cocompactly on $Y$, and there is an inclusion $A\hookrightarrow B$ and an equivariant convex embedding $X\rightarrow Y$. 
    
    The quasiflat $(X,A)$ is called \textit{abelian closable} if additionally $(Y,B)$ can be chosen such that $B$ is free abelian. Note that if $(X,A)$ is a closable quasiflat, then there is a finite-index normal abelian subgroup $A_0\lhd A$ such that the quasiflat $(X,A_0)$ is abelian closable. Moreover, for every finite-index subgroup $A_1\leqslant A_0$, the quasiflat $(X,A_1)$ is abelian closable.
\end{definition}

\begin{definition}\label{def:cosparse}
    Let $G$ be a group that is hyperbolic relative to a finite family of virtually abelian groups $\{H_i\}^n_{i=1}$. An action $G\curvearrowright X$ on a CAT(0) cube complex $X$ is called \textit{cosparse} if there exists a compact subcomplex $K\subset X$ and a family of quasiflats $\{(F_i, H_i)\}^n_{i=1}$  with $H_i=\mathrm{Stab}_G(F_i)$ for each $i$ such that

    \begin{enumerate}[label=(\roman*)]
        \item\label{sparse_1} $X=GK\cup (\bigcup^n_{i=1}GF_i)$;
        \item\label{sparse_2}  for each $i$ we have $(F_i\cap GK)\subset H_iK_i$ for some compact subcomplex $K_i\subset X$; and
        \item\label{sparse_3}  for $i,j\in\{1,\dots,n\}$ and $g\in G$, either $(F_i\cap gF_j)\subset GK$ or $i=j$ and $F_i=gF_j$.
    \end{enumerate}

    The group $G$ is called \textit{sparse special} if additionally the action $G\curvearrowright X$ is free and the quotient $G\backslash X$ is special.
\end{definition}

\begin{remark}\label{rem:sparse compact} We note that any compact special group $G$ that is hyperbolic relative to a family of virtually abelian groups $\{H_i\}^n_{i=1}$  is sparse special. To see this, consider the special CAT(0) cube complex $X$ on which $G$ acts properly and cocompactly. Let $K$ be a compact subcomplex so that $X=GK$. Then, the assumptions \ref{sparse_1} and \ref{sparse_3} of \Cref{def:cosparse} are immediately satisfied. By Theorem 3.6 of \cite{Wise_Woodhouse}, for each $i$, there is an $H_i$-invariant convex subcomplex $F_i$ on which $H_i$ acts cocompactly, which gives \ref{sparse_2}.
\end{remark}

\begin{lemma}\label{lem. sparse finite index}
    Let $G$ be a group that is hyperbolic relative to a finite family of virtually abelian groups $\{H_i\}^n_{i=1}$. Suppose $G$ acts  cosparsely on a CAT(0) cube complex $X$. Then, the restriction of the action to any finite index subgroup  $G'\leq G$ is cosparse. 
    
    In particular, any finite index subgroup of a sparse special group is sparse special.
\end{lemma}
\begin{proof} Let $T$ be a right transversal of $G'$ in $G$. For $t\in T$ and $i\in \{1,\dots,n\}$, let $H'_{i,t}=tH_it^{-1}\cap G'$. Then $G'$ hyperbolic relative to the family of virtually abelian groups $\{H'_{i,t}\}_{\{1\leq i\leq n,\, t\in T\}}$ by \cite[Theorem 9.1]{hruska2010relative}. Following the notation of \Cref{def:cosparse}, we define $K'=TK$ and the family of quasiflats $\{(tF_i, H'_{i,t})\}_{\{1\leq i\leq n,\, t\in T\}}$. Note that $H'_{i,t}=\mathrm{Stab}_{G'}(tF_i)$ for each $i$ and $t$. Next, we verify the conditions of \Cref{def:cosparse}.

\noindent {\it (i).} $X=GK\cup (\bigcup^n_{i=1}GF_i)= G'K'\cup (\bigcup_{\{1\leq i\leq n,\, t\in T\}}G'tF_i)$.\\
\noindent {\it (ii).} For each $i$ and $t$, we have $(tF_i\cap G'K')\subset tH_iK_i\subset H'_{i,t}K'_i$, for some compact subcomplex $K'_i\subset X$.\\ 
\noindent {\it (iii).} For $i,j\in\{1,\dots,n\}$, $t_1,t_2\in T$ and $g\in G$, either $(t_1F_i\cap gt_2F_j)=t_1(F_i\cap t_1^{-1}gt_2F_j)\subset G'K'$ or $i=j$ and $F_i=t_1^{-1}gt_2F_j$ implying $t_1F_i=gt_2F_j$. 

 For the last claim of the lemma suppose $\Gamma'$ is finite index subgroup of a sparse special group $\Gamma$. Then $\Gamma$ acts cosparsely on a CAT(0) cube complex $X$ such that the quotient is special. By the above, the group $\Gamma'$ acts cosparsely on $X$ and the quotient is a finite cover of the special complex $X/{\Gamma}$ and hence, it is special by \Cref{prop. properties of vir special groups} \ref{item: special}.
\end{proof}

\begin{lemma}\label{lem. get rid of finite index}
    Let $G$ be a finitely generated virtually special group that is hyperbolic relative to a finite family of virtually abelian subgroups $\{H_i\}^k_{i=1}$. For each $i$, let $K_i\leqslant H_i$ be any finite-index abelian subgroup. Then there exists a finite-index special normal subgroup $G_0\leqslant G$ such that $G_0\cap H_i\leqslant K_i$ for all $i$. If $G$ is virtually sparse (resp. compact) special, then we can take $G_0$ to be sparse (resp. compact) special.
\end{lemma}

\begin{proof}
    Consider one of the $K_i$. By \cite[Corollary 6.8]{wise2021structure}, there exists a nested sequence of finite-index subgroups of $G$
    \[G_{i0}>G_{i1}>\dots\]
    such that $\bigcap_{j\geqslant 0} G_{ij}=K_i$. The sequence of numbers $\{[H_i: H_i\cap G_{ij}]\}_{j\geqslant 0}$ is a monotone increasing sequence of integers, with upper bound $[H_i:K_i]$. Since
    \[\bigcap_{j\geqslant 0} (H_i\cap G_{ij})=H_i\cap\left(\bigcap_{j\geqslant 0} G_{ij}\right)=K_i,\]
    this sequence of numbers will eventually stabilize, so there exists $j_i$ such that $[H_i: H_i\cap G_{ij_i}]=[H_i:K_i]$, i.e., $H_i\cap G_{ij_i}=K_i$.
    
    Let $G'_0\leqslant G$ be a finite-index special subgroup, and let $G_0$ be the normal core of $G'_0\cap\left( \bigcap^k_{i=1} G_{ij_i} \right)$. Then by \Cref{prop. properties of vir special groups} \ref{item: special}, $G_0$ satisfies all requirements. If $G$ is virtually sparse (resp. compact) special, then we can take $G'_0$ to be sparse (resp. compact) special, in which case $G_0$ will also be sparse (resp. compact) special by \Cref{lem. sparse finite index}.
\end{proof}

\begin{lemma}\label{lem. sparse implies amalgam}
    Let $G$ be a finitely generated virtually sparse special group that is hyperbolic relative to a finite family of infinite virtually abelian groups $\{H_i\}^n_{i=1}$. Then there exist a finite-index normal subgroup $G_0\lhd G$ and a compact cubical group $G'_0$ such that the following hold:

    \begin{enumerate}
        \item[(i)] $G'_0$ splits as a graph of groups $\mathcal{T}$ whose underlying  graph is a star, i.e., obtained by gluing a finite number of edges to a central vertex.
        \item[(ii)] The central vertex group of $\mathcal{T}$ is $G_0$. 
        \item[(iii)] Every non-central vertex group is infinite finitely generated free abelian.
        \item[(iv)] Let $\{K_j\}^m_{j=1}$ be the induced peripheral structure on $G_0$. Then each edge group of $\mathcal{T}$ is isomorphic to some $K_j$, and different edge groups correspond to different $K_j$.
        \item[(v)] For all $i$, the group $G_0\cap H_i$ is infinite finitely generated free abelian.
        \item[(vi)] $G_0$ is hyperbolic relative to the family of edge groups of $\mathcal{T}$.
    \end{enumerate}
\end{lemma}

\begin{proof}
    Let $G_1\lhd G$ be a finite-index normal sparse special subgroup. So there exists a CAT(0) cube complex $X$ with a cosparse cospecial action of $G_1$. By \cite[Lemma 7.51]{wise2021structure}, each quasiflat of $X$ is closable.

    Let $\{L_k\}^p_{k=1}$ be the induced peripheral structure of $G_1$. Each $L_k$ corresponds to a quasiflat $Y_k$ of $X$. There exists a finite-index free abelian normal subgroup $N_k\lhd L_k$  such that the quasiflat $(Y_k,N_k)$ is abelian closable. By \Cref{lem. get rid of finite index}, there exists a finite-index normal subgroup $G_0\lhd G$ such that $G_0\cap L_k\leqslant N_k$ for all $k$. Then, by \Cref{lem. sparse finite index}, the action $G_0\curvearrowright X$ is cosparse, the quotient $G_0\backslash X$ is special, and every quasiflat of the action $G_0\backslash X$ is abelian closable. The desired result follows from \cite[Theorem 7.54]{wise2021structure}. Note that \cite[Theorem 7.54]{wise2021structure} only asserts that the non-central vertex groups of $\mathcal{T}$ are virtually abelian, but a close inspection of its proof shows that, if each quasiflat is abelian closable, then $G'_0$ can be constructed in such a way that the non-central vertex groups of $\mathcal{T}$ are abelian: There is a CAT(0) cube complex $X$ with a free cosparse action of $G_0$ such that the quotient $G_0\backslash X$ is special and every quasiflat of $X$ is abelian closable. Each quasiflat of $X$ has the form $(F_j, K_j)$ for some $j\in \{1,\dots,m\}$,  and there exists a quasiflat $(F'_j, K'_j)$ such that $K'_j$ is finitely generated free abelian and there is an inclusion $K_j\hookrightarrow K'_j$ and an equivariant embedding $F_j\hookrightarrow F'_j$. The group $G'_0$ is the quotient of $G_0\ast (\Asterisk^m_{j=1}K'_j)$ by identifying, for each $j$, the two embeddings $K'_j\hookleftarrow K_j \hookrightarrow G_0$ of $K_j$.
\end{proof}

\begin{lemma}\label{lem. extend abelian}
    Let $A\leqslant B$ be finitely generated free abelian groups. Then there exists a finite-index subgroup $C\leqslant B$ such that $A\leqslant C$ and for every subgroup $D\leqslant A$ such that $A/D$ is torsion-free, we have that $C/D$ is also torsion-free.
\end{lemma}

\begin{proof} Considering $A'$ as a $\Z$-module, $A'\otimes_{\Z}\Q$ is a $\Q$-vector space. Let $v_1,\dots,v_n$ be elements of $B\subset B\otimes_{\Z}\Q$ that together with a basis of $A\otimes_{\Z}\Q$  give a basis of $B\otimes_{\Z}\Q$. Let $C\leqslant B$ be the subgroup generated by $A$ and $v_1, \dots,v_n$. Then $C$ has finite-index in $B$. Moreover, $A$ is a direct factor of $C$ and thus if $D\leqslant A$ is a subgroup then $A/D$ is a direct factor of $C/D$. Thus, if $A/D$ is torsion-free then so is $C/D$.
\end{proof}

\begin{theorem}\label{thm. sparse rel. hyp.}
Let $G$ be a finitely generated virtually sparse special group that is hyperbolic relative to a finite family of virtually abelian subgroups $\{H_i\}^n_{i=1}$. Then there exists a family of finite-index normal subgroups $\{K_i\leqslant H_i\}^n_{i=1}$ such that the following holds:

\noindent
($\bigstar$) For each $i$, let $L_i\lhd H_i$ be any finite-index normal subgroup such that $L_i\leqslant K_i$. Then there exists a family of finite subsets $\{\F_i\subset L_i\smallsetminus\{1\}\}^n_{i=1}$ such that if a family of normal subgroups $\{N_i\lhd H_i\}$ satisfies that, for all $i$,
\begin{itemize}
    \item $N_i\leqslant L_i$,
    \item $N_i\cap \F_i=\emptyset$, and
    \item $L_i/N_i$ is torsion-free,
\end{itemize}
then
\[b^{(2)}_\ast(G/\ll \bigcup^n_{i=1}N_i \rr)=b^{(2)}_\ast(G).\]
Moreover, the family $\{K_i\}_{i=1}^n$ may be chosen so that there exists a finite-index sparse special normal subgroup $G_0\lhd G$ such that $K_i=G_0\cap H_i$ for every $i$. With this choice, if $L_i/N_i\cong\mathbb Z$ for every $i$, then $G/\ll \bigcup^n_{i=1}N_i \rr$ is virtually compact special.
\end{theorem}

\begin{proof}
    As we will take normal closures in different groups, we will use $\ll\cdot\rr_G$ instead of $\ll\cdot \rr$.

    If some $H_i$ is finite, for each choice of finite-index normal subgroup $L_i\lhd H_i$, we can let $\F_i=L_i\smallsetminus\{1\}$. Moreover, by \cite[Theorem 2.40]{osin2006relatively}, we can exclude $H_i$ from the peripheral structure of $G$. Therefore, we may assume that every $H_i$ is infinite.

    By \Cref{lem. sparse implies amalgam}, there exists a finite-index normal subgroup $G_0\lhd G$ and a compact special group $G'_0$ which splits as a graph of groups $\mathcal{T}$ such that $G_0,G'_0$ and $\mathcal{T}$ satisfy the conclusions of \Cref{lem. sparse implies amalgam}.

    Let $\{H_{0j}\}^m_{j=1}$ be the peripheral structure of $G_0$ induced by $\{H_i\}^n_{i=1}$. Then each $H_{0j}$ is an edge group of $\mathcal{T}$. The edge of $\mathcal{T}$ corresponding to $H_{0j}$ is incident to a unique non-central vertex, whose vertex group will be denoted by $H'_{0j}$. By \cite[Theorem 0.1 (2)]{dahmani2003combination}, $G'_0$ is hyperbolic relative to the family of abelian subgroups $\{H'_{0j}\}^m_{j=1}$.

    \begin{claim}\label{cl. equal l2 betti}
        $b^{(2)}_\ast(G'_0)=b^{(2)}_\ast(G_0)$.
    \end{claim}

    \begin{claimproof}
    Consider the Mayer--Vietoris sequence:
    {\small
\[\dots \rightarrow \bigoplus^m_{j=1} H_r(H_{0j};\N(G'_0))\rightarrow \left(\bigoplus^m_{j=1} H_r(H'_{0j};\N(G'_0))\right)\oplus H_r(G_0;\N(G'_0)) \rightarrow H_r(G'_0;\N(G'_0))\rightarrow \dots\]}
 \noindent Since  $H_{0j}$ and $H'_{0j}$ are infinite abelian, the desired result follows by taking $\dim_{G'_0}$ and using Theorem 6.37 and Theorem 6.54 (7) of \cite{luck2002l2}.
    \end{claimproof}

    Let $\{K'_{0j}\lhd H'_{0j}\}^m_{j=1}$ be the family of finite-index normal subgroups given by \Cref{thm. gen rel hyp.} such that the ($\bigstar$) property of \Cref{thm. gen rel hyp.} holds. 
    
    For the rest of this proof, we will apply \Cref{lem. df in fi subgroup} multiple times. In each of these applications, we are always applying \Cref{lem. df in fi subgroup} with respect to the pair of groups $G_0\lhd G$ and the peripheral structures $\{H_i\}^n_{i=1}$ and $\{H_{0j}\}^m_{j=1}$.
    
    Let $\{\overline K_i\leqslant H_i\}^n_{i=1}$ be the family of finite-index subgroups given by \Cref{lem. df in fi subgroup} \ref{item. finite index} with respect to the family of finite-index subgroups $\{K'_{0j}\cap H_{0j}\}^m_{j=1}$. For each $i$, let $K_i$ be the normal core of $\overline K_i$ in $H_i$. Then $K_i$ has finite-index in $H_i$.
    
    Now consider a family of finite-index subgroups $\{L_i\leqslant K_i\}^n_{i=1}$ such that each $L_i$ is normal in $H_i$. Let $\{L_{0j}\lhd H_{0j}\}^m_{j=1}$ be the family of normal subgroups given by \Cref{lem. df in fi subgroup} \ref{item. bijection} with respect to the family $\{L_i\}^n_{i=1}$. Then by \Cref{lem. df in fi subgroup} \ref{item. peripheral}, for each $j$, the group $H_{0j}/L_{0j}$ is isomorphic to some $(G_0\cap H_i)/L_i$, and thus $L_{0j}$ has finite-index in $H_{0j}$. Let $\{K_{0j}\lhd H_{0j}\}^m_{j=1}$ be the family of normal subgroups given by \Cref{lem. df in fi subgroup} \ref{item. bijection} with respect to the family $\{K_i\lhd H_i\}^n_{i=1}$. Then \Cref{lem. df in fi subgroup} \ref{item. finite index} implies that $K_{0j}\leqslant K'_{0j}\cap H_{0j}$ for all $j$, and thus $L_{0j}\leqslant K'_{0j}$.

    For each $j$, \Cref{lem. extend abelian}, applied to the pair of finitely generated free abelian groups $L_{0j}\leqslant K'_{0j}$, gives rise to a finite-index subgroup $L'_{0j}\leqslant K'_{0j}$ such that for every subgroup $N_{0j}\leqslant L_{0j}$ with $L_{0j}/N_{0j}$ torsion-free, the group $L'_{0j}/N_{0j}$ is torsion-free as well.

    By \cite[Theorem 1.1]{osin2007peripheral}, \Cref{thm. gen rel hyp.} and our choice of $\{K'_{0j}\}^m_{j=1}$, there exists a family of finite subsets $\{F'_{0j}\subset L'_{0j}\smallsetminus\{1\}\}^m_{j=1}$ such that if a family of subgroups $\{N_{0j}\leqslant L'_{0j}\}^m_{j=1}$ satisfies that for all $j$,
    \begin{itemize}
        \item $N_{0j}\cap \F'_{0j}=\emptyset$, and
        \item $L'_{0j}/N_{0j}$ is torsion-free,
    \end{itemize}
    then
        \begin{equation}\label{eq. sparse 1}
        b^{(2)}_\ast(G'_0/\ll \bigcup^m_{j=1} N_{0j} \rr_{G'_0})=b^{(2)}_\ast(G'_0),
    \end{equation}
    and
    \begin{equation}\label{eq. intersection 1}
        \ll \bigcup^m_{j=1} N_{0j} \rr_{G'_0}\cap K'_{0j}=N_{0j} \text{ for all } j.
    \end{equation}
    We will specify our choice of $\{N_{0j}\leqslant L'_{0j}\}^m_{j=1}$ later.
    
    Let $\{\F_i\subset H_i\smallsetminus\{1\}\}^n_{i=1}$ be the family of finite subsets given by \Cref{lem. df in fi subgroup} \ref{item. deep} with respect to the family $\{F'_{0j}\cap H_{0j}\}^m_{j=1}$. Now fix a family of normal subgroups $\{N_i\lhd H_i\}^n_{i=1}$ such that for all $i$,
    \begin{itemize}
        \item $N_i\leqslant L_i$,
        \item $N_i\cap \F_i=\emptyset$, and
        \item $L_i/N_i$ is torsion-free.
    \end{itemize}
    Let $\{N_{0j}\lhd H_{0j}\}^m_{j=1}$ be the family of normal subgroups given by \Cref{lem. df in fi subgroup} \ref{item. bijection} with respect to the family $\{N_i\}^n_{i=1}$. Then $N_{0j}\cap \F'_{0j}=\emptyset$ for all $j$. Moreover, by \Cref{lem. df in fi subgroup} \ref{item. peripheral}, we have that, for every $j$, the group $L_{0j}/N_{0j}$ is isomorphic to some $L_i/N_i$, and thus is torsion-free. Therefore, so is $L'_{0j}/N_{0j}$, by \Cref{lem. extend abelian} and our choice of $L'_{0j}$. Hence, \eqref{eq. sparse 1} and \eqref{eq. intersection 1} hold.

    Note that, as $N_{0j}\leqslant K_{0j}\leqslant K'_{0j}$, \eqref{eq. intersection 1} implies
    \[N_{0j}\leqslant \ll \bigcup^m_{j=1} N_{0j} \rr_{G_0}\cap K_{0j}\leqslant \ll \bigcup^m_{j=1} N_{0j} \rr_{G'_0}\cap K'_{0j}=N_{0j}.\]
    In particular,
    \begin{equation}\label{eq. intersection 2}
        \ll \bigcup^m_{j=1} N_{0j} \rr_{G_0}\cap K_{0j}=N_{0j} \text{ for all } j.
    \end{equation}
    By enlarging $\F_i$ if necessary, we may assume that 
    \begin{equation}\label{eq. infinite abelian 1}
        |K_i/N_i|=\infty \text{ for all } i,
    \end{equation}
    and thus
    \begin{equation}\label{eq. infinite abelian 2}
        |K'_{0j}/N_{0j}|=|K_{0j}/N_{0j}|=\infty \text{ for all } j.
    \end{equation}
    \begin{claim}
        $G'_0/\ll \bigcup^m_{j=1} N_{0j} \rr_{G'_0}$ splits as a graph of groups $\overline{\mathcal{T}}$ with the following properties:
        
        \begin{itemize}
            \item The underlying graph of $\overline{\mathcal{T}}$ is the same as that of $\mathcal{T}$.

            \item The central vertex group is $G_0/\ll \bigcup^m_{j=1} N_{0j}\rr_{G_0}$.

            \item The non-central vertex groups, as well as the edge groups, are infinite abelian.
        \end{itemize}
    \end{claim}

    \begin{claimproof}
    Construct a graph of groups $\overline{\mathcal{T}}$ from $\mathcal{T}$ by the following procedure:

    \begin{itemize}
            \item Let $v$ be a non-central vertex of $\mathcal{T}$. Then $G_v=K'_{0j}$ for some $j$. Replace $G_v$ by $K'_{0j}/N_{0j}$.

            \item Let $v$ be the central vertex of $\mathcal{T}$. Replace $G_v$ by $G_0/\ll \bigcup^m_{j=1}N_{0j}\rr_{G_0}$.

            \item Let $e$ be an edge of $\mathcal{T}$. Then $G_e=K_{0j}$ for some $j$. Replace $G_e$ by $K_{0j}/N_{0j}$.
    \end{itemize}
    
    By \eqref{eq. intersection 2}, for each $j$, the group $K_{0j}/N_{0j}$ is a subgroup of both $G_0/\ll \bigcup^m_{j=1}N_{0j}\rr_{G_0}$ and $K'_{0j}/N_{0j}$. Therefore, the above procedure does produce a graph of groups $\overline{\mathcal{T}}$. By \eqref{eq. infinite abelian 1} and \eqref{eq. infinite abelian 2}, the non-central vertex groups and the edge groups of $\mathcal{T}$ are infinite abelian. It is now easy to check that the fundamental group of $\overline{\mathcal{T}}$ is $G'_0/\ll \bigcup^m_{j=1} N_{0j} \rr_{G'_0}$. 
    \end{claimproof}

    Similar to \Cref{cl. equal l2 betti}, we have
    \begin{equation}\label{eq. equal l2 betti}
        b^{(2)}_\ast(G'_0/\ll \bigcup^m_{j=1} N_{0j} \rr_{G'_0})=b^{(2)}_\ast(G_0/\ll \bigcup^m_{j=1} N_{0j}\rr_{G_0}).
    \end{equation}
    By \Cref{lem. df in fi subgroup} \ref{item. embedding}, we have
    \begin{equation}\label{eq. index 2}
        [G/\ll \bigcup^n_{i=1} N_i \rr_G : G_0/\ll \bigcup^m_{j=1} N_{0j} \rr]=[G:G_0].
    \end{equation}
 The desired result follows from  \Cref{cl. equal l2 betti}, \eqref{eq. sparse 1}, \eqref{eq. equal l2 betti}, \eqref{eq. index 2} and \cite[Theorem 6.54 (6)]{luck2002l2}.    

It remains to establish the final assertions concerning the choice of the subgroups $K_i$.  By \cite[Theorem 15.6]{wise2021structure}, there exists a family of finite-index normal subgroups $\{K'_i\lhd H_i\}^n_{i=1}$ such that if a family of normal subgroups $\{N_i\lhd H_i\}^n_{i=1}$ satisfies that, for all $i$,

     \begin{itemize}
         \item $N_i\leqslant K'_i$, and
         \item $K'_i/N_i$ is virtually cyclic,
     \end{itemize}
     then the quotient $G/\ll \bigcup^n_{i=1} N_i \rr$ is virtually compact special.

     By \Cref{lem. get rid of finite index}, there exists a finite-index sparse special normal subgroup $G_0\lhd G$ such that $G_0\cap H_i\leqslant K_i\cap K'_i$ for all $i$. Now simply note that the ($\bigstar$) condition continues to hold with $G_0\cap H_i$ in place of $K_i$.
\end{proof}

\section{Applications}\label{sec. applications}

Next, we present some applications of our main results.

\subsection{Algebraic fibering}

For $n\in\mathbb N^+\cup\{\infty\}$, we say that a group $G$ \textit{virtually $FP_n(\Q)$-fibers} if $G$ has a finite-index subgroup $G_0$ that has a homomorphism onto $\Z$ with kernel of type $FP_k(\Q)$. Algebraic fibrations of groups are intimately related to $L^2$-Betti numbers. The following theorem is proved for $n=1$ by Kielak, and then extended to all $n$ by Fisher.

\begin{theorem}[\cite{kielak2020RFRS,fisher2021improved}]\label{thm. fiber and betti}
    For a virtually RFRS group $G$ of type $FP_n(\Q)$, the following are equivalent:
    \begin{enumerate}
        \item[(i)] $G$ virtually $FP_n(\Q)$-fibers.
        \item[(ii)] $b^{(2)}_i(G)=0$ for $i\leqslant n$.
    \end{enumerate}
\end{theorem}

By combining the above theorem with our computations on $L^2$-Betti numbers, we obtain: 
\begin{corollary}\label{cor. 1 fiber}
    Let $G$ be a virtually compact cubical hyperbolic group. Suppose that $G$ virtually $FP_1(\Q)$-fibers and is hyperbolic relative to a family of subgroups $\{H_i\}^n_{i=1}$. Then there exists a family of finite-index torsion-free normal subgroups $\{K_i\lhd H_i\}^n_{i=1}$ such that for every family of sufficiently deep normal subgroups $\{N_i\lhd H_i\}^n_{i=1}$ that satisfies, for all $i$, 

    \begin{itemize}
        \item $N_i\leqslant K_i$, and
        \item $K_i/N_i$ is torsion-free virtually compact cubical hyperbolic,
    \end{itemize}
    then the group $G/\ll \bigcup^n_{i=1} N_i \rr$ virtually $FP_1(\Q)$-fibers.
\end{corollary}

The assumption of \Cref{cor. 1 fiber} seemingly differs from that of \Cref{cor. intro 1 fiber}, but these assumptions are in fact equivalent (see \Cref{rm. assumption differ}).

\begin{proof}
    By \Cref{prop. properties of vir special groups}, $G$ is of type $F_\infty$, and thus is of type $FP_\infty(\Q)$. First consider the case where $G$ is virtually $\Z$. For each $i$, let $\F_i=\emptyset$. It remains to choose a finite-index normal subgroup $K_i\lhd H_i$ for each $i$. If $|H_i|<\infty$, then we let $K_i=\{1\}$ and $\F_i=\emptyset$. If $|H_i|=\infty$, then $H_i=G$ and we will let $K_i=G$. 
    
    Now suppose that a family of normal subgroups $\{N_i\lhd H_i\}^n_{i=1}$ is given such that, for all $i$, 

    \begin{itemize}
        \item $N_i\leqslant K_i$, and
        \item $K_i/N_i$ is torsion-free.
    \end{itemize}

    If $|H_i|<\infty$ for all $i$, then by construction $N_i=\{1\}$ for all $i$ and thus the group $G/\ll \bigcup^n_{i=1} N_i \rr=G$ virtually $FP_1(\Q)$-fibers. If some $H_{i_0}$ is infinite, then for $i\neq i_0$ we will have $|H_{i_0}|<\infty$. Therefore, $G/\ll \bigcup^n_{i=1} N_i \rr=G/\ll N_{i_0} \rr=K_{i_0}/ N_{i_0}$. As the group $K_{i_0}/ N_{i_0}$ is torsion-free virtually cyclic, it is either $\{1\}$ or $\Z$. In either case, the group $K_{i_0}/ N_{i_0}$, and thus the group $G/\ll \bigcup^n_{i=1} N_i \rr$ virtually $FP_1(\Q)$-fibers.

    It remains to consider the case where $G$ is not virtually $\Z$. As $G$ virtually $FP_1(\Q)$-fibers, \cite[Theorem 6.63]{luck2002l2} implies that $b^{(2)}_1(G)=0$.

    By \Cref{thm. hyp.}, \cite[Theorem 1.1]{osin2007peripheral} and \cite[Theorem 2.7]{agol2016alternate}, there exists a family of finite-index torsion-free normal subgroup $\{K_i\lhd H_i\}^n_{i=1}$ such that for every family of sufficiently deep normal subgroup $\{N_i\lhd H_i\}^n_{i=1}$ that satisfies, for all $i$, $N_i\leqslant K_i$ and $K_i/N_i$ is torsion-free virtually compact cubical hyperbolic, we have that 
    
    \begin{enumerate}
        \item[(i)] $G/\ll \bigcup^n_{i=1} N_i \rr$ is virtually compact special (and thus is of type $FP_\infty(\Q)$ and virtually RFRS by \Cref{prop. properties of vir special groups});
        \item[(ii)] $G/\ll \bigcup^n_{i=1} N_i \rr$ is non-elementary hyperbolic, in particular, $|G/\ll \bigcup^n_{i=1} N_i \rr|=\infty$, and thus $b^{(2)}_0(G/\ll \bigcup^n_{i=1} N_i \rr)=0$ by \cite[Theorem 6.54 (8b)]{luck2002l2}; and
        \item[(iii)] $b^{(2)}_1(G/\ll \bigcup^n_{i=1} N_i \rr)=0$.
    \end{enumerate}
    The desired result then follows from \Cref{thm. fiber and betti}.
\end{proof}

\begin{corollary}\label{cor. k fiber}
    Let $G$ be a finitely generated virtually sparse special group that is hyperbolic relative to a family of virtually abelian subgroups $\{H_i\}^n_{i=1}$, and let $k\in\mathbb N^+$. Then there exists a family of finite-index torsion-free normal subgroups $\{K_i\lhd H_i\}^n_{i=1}$ such that for every family of sufficiently deep normal subgroups $\{N_i\lhd H_i\}^n_{i=1}$ such that for all $i$,
    
    \begin{itemize}
        \item $N_i\leqslant K_i$, and
        \item $K_i/N_i\cong \Z$,
    \end{itemize}
    we have that the group $G/\ll \bigcup^n_{i=1} N_i \rr$ virtually $FP_k(\Q)$-fibers if and only if so does $G$.
\end{corollary}

\begin{proof}
        By Theorem 2.11 and Remark 2.13 of \cite{patil2023finiteness}, each $H_i$ is finitely generated virtually abelian, and thus is of type $F_\infty$. By Theorem 2.11 and Remark 2.13 of \cite{patil2023finiteness} again, the group $G$ is of type $F_\infty$ (in particular, $FP_\infty(\Q)$). By \Cref{prop. properties of vir special groups}, the group $G$ is virtually RFRS. Therefore, $G$ virtually $FP_n(\Q)$-fibers if and only if $b^{(2)}_j(G)=0$ for $i\leqslant k$ by \Cref{thm. fiber and betti}.

        By \Cref{cor. finiteness conditions}, \Cref{lem. special quotient}, and  \Cref{thm. sparse rel. hyp.}, there exists a family of finite-index normal subgroups $\{K_i\lhd H_i\}^n_{i=1}$ such that for every family of sufficiently deep normal subgroups $\{N_i\lhd H_i\}^n_{i=1}$ such that for all $i$, $N_i\leqslant K_i$ and $K_i/N_i\cong \Z$, we have the following: 
        
        \begin{itemize}
            \item The group $G/\ll \bigcup^n_{i=1} N_i \rr$ is virtually compact special, and thus is virtually RFRS and of type $F_\infty$ by \Cref{prop. properties of vir special groups}.
            \item For all $j$, we have $b^{(2)}_j(G/\ll N \rr)=0$ if and only if $b^{(2)}_j(G)=0$.
        \end{itemize}
        Note that by \Cref{thm. sparse rel. hyp.}, we may assume that the subgroups $K_i$ are contained in a finite-index sparse special subgroup of $G$. In particular, we may assume that the $K_i$ are torsion-free.
        
        Now, \Cref{thm. fiber and betti} implies that $G/\ll N \rr$ virtually $FP_k(\Q)$-fibers if and only if $b^{(2)}_j(G/\ll N \rr_G)=0$ for $j\leqslant k$, which is equivalent to $b^{(2)}_j(G)=0$ for $j\leqslant k$, which in turn is equivalent to that $G$ virtually $FP_k(\Q)$-fibers.
\end{proof}

\subsection{Arithmetic lattices in $SO(n,1)$} We apply the results of Section \ref{sec:sparse} to hyperbolic arithmetic lattices.

\begin{theorem}[Bergeron--Wise \cite{bergeron2012boundary}]\label{thm. lattice sparse}
    Let $G<SO(n,1)$ be a non-uniform arithmetic lattice. Then $G$ is virtually sparse special.
\end{theorem}

\begin{proof}
    In \cite[Theorem 6.2]{bergeron2012boundary}, Bergeron--Wise  proved that $G$ is virtually special by the following argument: The group $G$ has a quasi-isometrically embedded codimension-1 subgroup that intersects every peripheral subgroup of $G$ along an infinite subgroup. The authors applied \cite[Proposition 5.2]{bergeron2012boundary} to obtain an action $G\curvearrowright X$ of $G$ on a CAT(0) cube complex $X$ and argue that the action is proper and that the quotient $G_0\backslash X$ is special, where $G_0$ is a finite-index torsion-free subgroup given by the Selberg's Lemma. Note that \cite[Proposition 5.2]{bergeron2012boundary} is strengthened by \cite[Theorem 7.9]{wise2021structure} which asserts that the action $G\curvearrowright X$ is cosparse as well.
\end{proof}

 The following theorem is a direct consequence of \cite[Theorem 2.3]{jost2000vanishing} and \cite[Theorem 6.54 (6b)]{luck2002l2}. We provide the proof for the convenience of the reader.

\begin{theorem}\label{thm. l2 betti hyp mnfl}
    Let $G<SO(n,1)$ be a lattice. Then 
\[
b^{(2)}_i(G)=
\begin{cases}
    0,&\text{if }i\neq n/2\\
    (-1)^{n\over 2}\chi(G)=\dfrac{2\cdot vol(\mathbb H^n/G)}{\Omega_n},&\text{if }i=n/2,
\end{cases}
\]
where $\Omega_n$ is the volume of the unit $n$-sphere and $\chi(G)$ is the rational-valued Euler characteristic of $G$.
\end{theorem}

\begin{proof}
    First suppose that $G$ is a torsion-free cocompact lattice. Then $\mathbb H^n/G$ is a hyperbolic $n$-manifold. By \cite[Theorem 2.3]{jost2000vanishing},

    \[
b^{(2)}_i(H)=
\begin{cases}
    0,&\text{if }i\neq n/2\\
    (-1)^{n\over 2}\chi(\mathbb H^n/G),&\text{if }i=n/2.
\end{cases}
\]

\noindent When $n$ is even, by \cite[Theorem 3.3]{KZ01},  $(-1)^{n\over 2}\chi(\mathbb H^n/G)=2\cdot vol(\mathbb H^n/G)/\Omega_n$, which concludes the proof in this case.

The general case follows from the above special case and \cite[Corollary 0.2]{gaboriau2002invariants}.
\end{proof}

The theorem below summarizes our results in the setting of hyperbolic lattices. 

\begin{theorem}\label{thm. lattice summarize}
Let $G<SO(n,1)$ be any non-uniform  lattice with $n\geqslant 2$, and let $\{H_i\}^k_{i=1}$ be the peripheral subgroups of $G$. 

\begin{enumerate}[label=(\roman*)]

    \item\label{item. lattice li} Assume $G$ is virtually locally indicable (e.g.~$G$ is arithmetic). For every $\delta>0$, there exists a family of finite subsets $\{\F_{\delta,i}\subset H_i\smallsetminus\{1\}\}^k_{i=1}$ such that if a family of normal subgroups $\{N_i\lhd H_i\}^k_{i=1}$ satisfies that $N_i\cap \F_{\delta,i}=\emptyset$ for all $i$, then for $\ell\leqslant n$ we have

    \begin{equation}\label{eq. lattice l2 betti inequality}
        \left| b^{(2)}_\ell(G/\ll \bigcup^k_{i=1} N_i\rr_G) - b^{(2)}_\ell(G) \right|<\delta.
    \end{equation}
    
\noindent  In particular, if additionally $n$ is even and $\delta$ is sufficiently small, then $b^{(2)}_{n/2}(G/\ll \bigcup^k_{i=1} N_i\rr_G)>0$.\\

    \item\label{item. lattice df} Assume $G$ is arithmetic. There exists a finite-index sparse special normal subgroup $G_0\lhd G$ such that the following holds: For each $i$, let $H'_i:=G_0\cap H_i$. Then there exists a family of finite subsets $\{\F'_i\subset H'_i\smallsetminus\{1\}\}^k_{i=1}$ such that if a family of normal subgroups $\{N_i\lhd H_i\}^k_{i=1}$ satisfies that, for all $i$,

    \begin{itemize}
        \item $N_i\cap \F'_i=\emptyset$,
        \item $N_i\leqslant H'_i$, and
        \item $H'_i/N_i$ is torsion-free,
    \end{itemize}
    then
    \begin{equation*}
        b^{(2)}_\ast(G/\ll \bigcup^k_{i=1} N_i\rr_G)=b^{(2)}_\ast(G).
    \end{equation*}
    Moreover, if additionally $H'_i/N_i\cong \Z$ for all $i$, then the quotient $G/\ll \bigcup^k_{i=1} N_i\rr_G$ is virtually compact special. In particular, when $n$ is odd, the group $G/\ll \bigcup^k_{i=1} N_i\rr_G$ virtually $FP_\infty(\Q)$-fibers; and when $n$ is even, the group $G/\ll \bigcup^k_{i=1} N_i\rr_G$ virtually $FP_{n/2-1}(\Q)$-fibers but does not virtually $FP_{n/2}(\Q)$-fiber.\\
\end{enumerate}
\end{theorem}

\begin{proof}\hspace{1mm}
   
      \noindent {\it (i)}. 
      For $\delta>0$ and each $\ell\in \{0,1,\dots,n\}$, \Cref{cor. gen li} provides a family of finite subsets $\{\F_{\delta,i,\ell}\subset H_i\smallsetminus\{1\}\}^k_{i=1}$ such that if a family of normal subgroups $\{N_i\lhd H_i\}^k_{i=1}$ satisfies that $N_i\cap \F_{\delta,i,\ell}=\emptyset$ for all $i$, then
      \[\left| b^{(2)}_\ell(G/\ll \bigcup^k_{i=1} N_i\rr_G) - b^{(2)}_\ell(G) \right|<\delta.\]
      Inequality \eqref{eq. lattice l2 betti inequality} follows by letting $\F_{\delta,i}=\bigcup^n_{\ell=0} \F_{\delta,i,\ell}$. The last statement of the theorem follows from \eqref{eq. lattice l2 betti inequality} and \Cref{thm. l2 betti hyp mnfl}.

In the case when  $G$ is arithmetic, by  \Cref{thm. lattice sparse}, it is finitely generated virtually special, and thus it is virtually locally indicable by \Cref{prop. properties of vir special groups}.

    \noindent {\it (ii)}. By \Cref{thm. lattice sparse}, the group $G$ is finitely generated virtually sparse special. So the desired conclusion follows from \Cref{thm. sparse rel. hyp.}, \Cref{prop. properties of vir special groups}, and \Cref{thm. fiber and betti}.
\end{proof}

\begin{remark}\label{rm. further finite index}
    Assume the setting of \Cref{thm. lattice summarize} \ref{item. lattice df} and let $G_1\lhd G$ be any finite-index normal subgroup that is contained in $G_0$. Then the same proof shows that the conclusion of \ref{item. lattice df} still holds if $H'_i:=G_0\cap H_i$ is replaced by $H'_i:=G_1\cap H_i$.
\end{remark}

\subsection{Einstein manifolds}

In this and the next subsections, we will apply our result to cusped arithmetic hyperbolic manifolds. There are two notions of Dehn filling, one for manifolds and the other for groups. To distinguish them, we will call the former {\it manifold filling}, as defined below.

\begin{definition}
    Let $M$ be a compact topological $n$-manifold such that $\partial M$ can be decomposed as a disjoint union $\partial M=\bigsqcup^k_{i=1} C_i$,
    with each $C_i$ homeomorphic to the $(n-1)$-torus $(S^1)^{n-1}$.
    
    Let $T_1,\dots,T_k$ be $k$-copies of the solid $n$-torus $D^2\times (S^1)^{n-2}$. A \textit{manifold filling} $\overline M$ of $M$ is a closed $n$-manifold obtained by gluing $T_1,\dots,T_k$ to $M$ by a homeomorphism identifying $\partial T_i$ with $C_i$ for each $i$.
    For each $i$, $\partial D^2$ represents an element $x_i\in \pi_1(\partial T_i)$, via the homeomorphism $T_i\cong D^2\times (S^1)^{n-2}$. Let $H_i$ be the image of $\pi_1(C_i)$ in $\pi_1(M)$ under the inclusion $C_i\hookrightarrow M$. The gluing sends $x_i$ to an element $y_i\in H_i$.

    We say that a property $\mathcal P$ \textit{holds for all sufficiently deep manifold fillings} $\overline M$ of $M$ if there exists a family of finite subsets $\{\F_i\subset H_i\smallsetminus\{1\}\}^k_{i=1}$ such that $\mathcal P$ holds whenever $\langle y_i\rangle\cap \F_i=\emptyset$ for all $i$.
\end{definition}

\begin{definition}
    We say that a Dehn filling $\overline M$ of $M$ is a \textit{$G$-equivariant manifold filling} if $\overline M$ is a manifold filling of $M$ and $\pi_1(\overline M)$ is a $G$-equivariant filling of $G_0$ (see \Cref{def. equivariant filling}).
\end{definition}

The following  theorem is a consequence of \Cref{thm. lattice summarize} \ref{item. lattice li}.

\begin{theorem}\label{thm. li hyp mnfl}
    Let $M$ be a cusped hyperbolic manifold with torus cusps and virtually locally indicable fundamental group, then for every $\delta>0$, every sufficiently deep manifold filling $\overline{M}$ of $M$ satisfies
    \begin{enumerate}
            \item[(a)] if $n\neq \dim(M)/2$, then $b^{(2)}_n(\pi_1(\overline{M}))<\delta$;

            \item[(b)] if $n=\dim M/2$, then $|b^{(2)}_n(\pi_1(\overline{M}))-|\chi(M)||<\delta$.
    \end{enumerate}
\end{theorem}

Suppose $M=\mathbb H^n/{\Gamma}$ is a finite volume cusped hyperbolic manifold. Then $M$ is called {\it arithmetic} if $\Gamma$ is commensurable with a non-uniform arithmetic lattice in $\mathrm{SO}(n,1)$.

\begin{theorem}\label{thm. einstein}
    Let $M$ be a cusped arithmetic hyperbolic manifold, then there is a finite-sheeted cover $W$ of $M$ such that $W$ has torus cusps and sufficiently deep $\pi_1(M)$-equivariant manifold fillings $\overline{W}$ of $W$ satisfy
        \[b^{(2)}_\ast(\pi_1(\overline{W}))=[\pi_1(M):\pi_1(W)]\cdot b^{(2)}_\ast(\pi_1(M)).\]
\end{theorem}

\begin{proof}
    Let $G_0\lhd \pi_1(M)$ be the finite-index normal subgroup provided by \Cref{thm. lattice summarize} \ref{item. lattice df}, and let $W$ be the cover of $M$ with $\pi_1(W)=G_0$. Then $W$ satisfies the requirements.
\end{proof}

\begin{corollary}\label{cor. constructing Einstein}
For every cusped arithmetic hyperbolic manifold $M$ of dimension at least four, there exists a sequence of pairwise non-homeomorphic Anderson-type closed Einstein manifolds $\{W_n\}_{n\geqslant 1}$ such that each $W_n$ satisfies the Singer Conjecture and $\{W_n\}_{n\geqslant 1}$ converges to a finite-sheeted cover $W$ of $M$ under the pointed Gromov--Hausdorff metric with based point any $y\in W$.
\end{corollary}

\begin{proof}
    Let $W$ be the finite-sheeted cover of $M$ given by \Cref{thm. einstein}, and let $H_1,\dots,H_m$ be the fundamental groups of the cusps of $M$, viewed as subgroups of $\pi_1(M)$. For $i=1,\dots, m$, let $a_i,b_i\in \pi_1(W)\cap H_i$ be two elements that can be extended to a basis of the free abelian group $\pi_1(W)\cap H_i$. For $n\in\mathbb N^+$, let $x_{i,n}=a_i+n\cdot b_i$. Then the sequence of elements $\{x_{i,n}\}_{n\in\mathbb N^+}$ satisfies the following:

    \begin{enumerate}
        \item[(i)] Each $x_{i,n}$ is not a proper power of another element of $\pi_1(W)\cap H_i$.
        \item[(ii)] For each finite subset $\fS\subset H_i$, there exists $m>0$ such that for all $n>m$ we have $\langle x_{i,n} \rangle\cap \fS=\emptyset$.
    \end{enumerate}
    
    Let $\{K_j\}^\ell_{j=1}$ be the peripheral structure on $\pi_1(W)$ induced by $\{H_i\}^m_{i=1}$. For each $n$, let $\{Q_{j,n}\leqslant K_j\}^\ell_{j=1}$ be the family given by \Cref{lem. df in fi subgroup} \ref{item. bijection} with respect to the family $\{\langle x_{i,n}\rangle \leqslant H_i\}^k_{i=1}$. Then for all $j$ and $n$, we have that $Q_{j,n}\cong \Z$ and $K_j/Q_{j,n}$ is torsion-free, by \Cref{lem. df in fi subgroup} \ref{item. peripheral}. Therefore, each of the quotient groups $\pi_1(W)/\ll \bigcup^\ell_{j=1} Q_{j,n} \rr_{\pi_1(W)}$ is the fundamental group of a $\pi_1(M)$-equivariant manifold filling of $W$, which will be denoted by $W_n$.
    
    By the main result of \cite{anderson2006dehn} and \cite[Theorem 2.7]{fujmann2010}, there exists a family of finite subsets $\{\F_j\subset K_j\smallsetminus\{1\}\}^\ell_{j=1}$ such that if a family of subgroups $\{Q_j\leqslant K_j\}^\ell_{j=1}$ induces a manifold filling $\overline{W}$ of $W$ and satisfies that for each $j$, we have $Q_j\cap \F_j=\emptyset$, then $\overline{W}$ admits both an Einstein metric and a locally CAT(0) metric. The latter implies that $\overline{W}$ is aspherical. 

    By \Cref{lem. df in fi subgroup} \ref{item. deep}, for $n$ large enough we have that $Q_{j,n}\cap \F_j=\emptyset$. By \Cref{thm. einstein} and \Cref{thm. l2 betti hyp mnfl}, for large $n$, we have that each $W_n$ satisfies the Singer Conjecture. By \cite[Proposition 3.9]{anderson2006dehn}, we may assume that the manifolds $W_n$ are pairwise non-homeomorphic. By \cite{anderson2006dehn}, the sequence $\{W_n\}_{n\geqslant 1}$ converges to $W$ under the pointed Gromov--Hausdorff metric with based point any $y\in W$.
\end{proof}

\subsection{Exotic subgroups of hyperbolic groups}  We present a new construction of exotic subgroups of hyperbolic groups that arise from hyperbolic Dehn filling.

\begin{theorem}\label{thm. exotic quotient}
Let $G$ be the fundamental group of a cusped arithmetic hyperbolic manifold of dimension $n$.  Then there is an infinite sequence $\{\bg_k\}^\infty_{k=1}$ of pairwise non-isomorphic hyperbolic Dehn filling quotients of $G$, each of which has the same $L^2$-Betti numbers as $G$. In particular,
\begin{itemize}
  \item[(i)] If $n=2q\geqslant 4$,  then for each $k$, there exist a finite-index subgroup  $\bg'_k\leqslant \bg_k$ and an epimorphism $\psi:\bg'_k\rightarrow \Z$ such that $\ker \psi$ is of type $FP_{q-1}(\Q)$ but not of type $FP_{q}(\Q)$.\\
  \item[(ii)] If $n=2q+1\geqslant 5$, then there exist a finite-index subgroup $\bg'_k\leqslant \bg_k$ and an epimorphism $\psi:\bg'_k\rightarrow \Z$ such that $\ker \psi$ is of type $FP(\Q)$ but not hyperbolic.
\end{itemize}
  \end{theorem}

\begin{proof}
    As we will discuss normal closures in different groups, we will use $\ll\cdot\rr_G$ instead of $\ll \cdot \rr$. Let $H_i,i=1,\dots,\ell$, be the fundamental groups of the cusps of $M$, viewed as subgroups of $G$. By \cite[Theorem 7.15 (a)]{dahmani2017hyperbolically}, there exists a family of finite subsets
    $\{\F_i\subset G\cap H_i\smallsetminus\{1\}\}^\ell_{i=1}$ such that if a family of normal subgroups $\{N_i\lhd H_i\}^\ell_{i=1}$ satisfies that, for all $i$,
    
    \begin{itemize}
        \item $N_i\cap \F_i=\emptyset$, and
        \item $H_i/N_i$ is virtually-$\Z$,
    \end{itemize}
    then the group $G/\ll \bigcup^\ell_{i=1}N_i\rr_G$ is hyperbolic, and each of the natural homomorphisms $H_i/N_i\rightarrow G/\ll \bigcup^\ell_{i=1}N_i\rr_G$ is injective, in which case $H_i/N_i$ can be viewed as a subgroup of $G/\ll \bigcup^\ell_{i=1}N_i\rr_G$.

    \Cref{thm. lattice summarize} \ref{item. lattice df} provides us with a finite-index sparse special normal subgroup $G_0\lhd G$. As $G$ is residually finite, there is an infinite sequence $\{G_k\}^\infty_{k=1}$ of finite-index normal subgroups of $G$ such that $G_k\leqslant G_0$ for all $k$ and $\bigcap_{k\geqslant 1}G_k=\{1\}$. For each $k$, \Cref{rm. further finite index} gives us a family of finite subsets $\{\F_{k,i}\subset G_k\cap H_i\smallsetminus\{1\}\}^\ell_{i=1}$ such that the conclusion of \Cref{thm. lattice summarize} \ref{item. lattice df} holds with $H'_i:=H_i\cap G_k$. For each $k$, we will construct a hyperbolic Dehn filling quotient $\bg_k$ that has the same $L^2$-Betti numbers as $G$ and the desired fibering property, and then prove that the sequence $\{\bg_k\}^\infty_{k=1}$ contains an infinite number of isomorphism types of groups.\\

    \noindent {\it (i)}. Suppose $n=2q\geqslant 4$. For each $i$ and $k$, the group $G_k\cap H_i$ is free abelian of rank at least $3$. For each $i$, use \Cref{lem. zn} to choose a subgroup $N_{k,i}\leqslant (G_k\cap H_i)$ such that $(G_k\cap H_i)/N_{k,i}\cong \Z$ and $N_{k,i}\cap (\F_i\cup\F_{k,i})=\emptyset$. Then the quotient $\bg_k:=G/\ll \bigcup^n_{i=1} N_{k,i} \rr_G$ is hyperbolic, has the same $L^2$-Betti numbers as $G$, and virtually $FP_{q-1}(\Q)$-fibers, i.e., it has a finite-index subgroup $\bg'_k\leqslant \bg_k$ with an epimorphism $\psi\colon \bg'_k\rightarrow \Z$ such that $\ker\psi$ is of type $FP_{q-1}(\Q)$. By \Cref{thm. lattice summarize} \ref{item. lattice df}, the group $\bg_k$ does not virtually $FP_q(\Q)$-fiber. So $\ker\psi$ is not of type $FP_q(\Q)$.\\

    \noindent {\it (ii)}. For each $k$, let $\{H_{k,j}\}^m_{j=1}$ be the induced peripheral structure of $G_k$. By Corollary 4.7, Theorem 4.10 (ii) of \cite{sun2019cohomologyii} and \Cref{cor. gen finiteness conditions}, there exists a family of finite subsets $\{\fS_{k,j}\subset H_{k,j}\smallsetminus\{1\}\}^m_{j=1}$ such that if a family of subgroups $\{M_{k,j}\}^m_{j=1}$ satisfies that $M_{k,j}\cap \fS_{k,j}=\emptyset$ for all $j$, then the group
    \[\bg''_k:= G_k/\ll \bigcup^m_{j=1} M_{k,j} \rr_{G_k}\]
    satisfies:
    \begin{enumerate}[label=(\alph*)]
        \item\label{item. torsion-free} $\cd(\bg''_k)<\infty$, in particular, $\bg''_k$ is torsion-free.
    \end{enumerate}
    Also,
    \begin{equation}\label{eq. dimension}
        H^n(\bg''_k ,\{\overline{H}_{k,j}\}^m_{j=1};\Z/2\Z)\cong H^n(G_k,\{H_{k,j}\}^m_{j=1};\Z/2\Z)\cong \Z/2\Z,
    \end{equation}
    where, for $j=1,\dots,m$, the group $\overline{H}_{k,j}$ is the natural image of $H_{k,j}$ in $\bg''_k$, and the last isomorphism follows from Poincar\'{e}--Lefschetz duality: as $G_k$ is sparse special, it is torsion-free, and thus is the fundamental group of a cusped hyperbolic manifold. We will specify our choice of $\{M_{k,j}\}^m_{j=1}$ later.

    By applying \Cref{lem. df in fi subgroup} \ref{item. deep} to the pair of groups $G_k\lhd G$, we obtain a family of finite subsets $\{\F'_{k,i}\subset H_i\smallsetminus\{1\}\}^\ell_{i=1}$ with respect to the family $\{\fS_{k,j}\}^m_{j=1}$. 
    
    We now specify our choice of $\{N_{k,i}\}^\ell_{i=1}$ and $\{M_{k,j}\}^m_{j=1}$. As $n\geqslant 5$, for each $i$, the group $G_k\cap H_i$ is free abelian of rank at least $4$. For each $i$, use \Cref{lem. zn} to choose a subgroup $N_{k,i}\leqslant G_k\cap H_i$ such that $N_i\cap (\F_{k,i}\cup \F_i\cup \F'_i)=\emptyset$ and $(G_k\cap H_i)/N_i\cong \Z$. Let $\{M_{k,j}\leqslant H_{k,j}\}^m_{j=1}$ be the $G$-equivariant family of subgroups given by \Cref{lem. df in fi subgroup} \ref{item. bijection} with respect to the family $\{N_{k,i}\}^\ell_{i=1}$. Then $M_{k,j}\cap \fS_{k,j}=\emptyset$ for all $j$. Therefore, \ref{item. torsion-free} and \eqref{eq. dimension} hold. Let 
    \[\bg_k:=G/\ll \bigcup^{\ell}_{i=1} N_{k,i} \rr_G,\;\;\bg''_k:=G_k/\ll \bigcup^m_{j=1} M_{k,j}\rr_{G_k}.\] 
    
    By \Cref{thm. lattice summarize} \ref{item. lattice df}, the group $\bg_k$ has the same $L^2$-Betti numbers as $G$ and virtually $FP_\infty(\Q)$-fibers. Being a finite-index subgroup of $\bg_k$, the group $\bg''_k$ also virtually $FP_\infty(\Q)$-fibers. So there is a finite-index subgroup $\bg'_k\leqslant \bg''_k$ and an epimorphism $\psi\colon \bg'_k\rightarrow \Z$ such that $\ker\psi$ is of type $FP_\infty(\Q)$. By \ref{item. torsion-free} and $\ker\psi<\bg''_k$, we have $\cd(\ker\psi)<\infty$ and thus $\ker\psi$ is of type $FP(\Q)$. Assume, for the sake of contradiction, that $\ker\psi$ is hyperbolic. Since $\bg''_k$ is torsion-free by \ref{item. torsion-free}, so is its subgroup $\bg'_k$. It follows from \cite[Proposition 3.3]{kudlinska} that $\cd(\ker\psi)\leqslant 2$, and thus $\cd(\bg'_k)\leqslant 3$. On the other hand, the fact that $\overline H_{k,j}\cong \Z$ for all $j$ and \eqref{eq. dimension} imply that $\cd(\bg''_k)\geqslant 5$ and thus $\cd(\bg'_k)\geqslant 5$, which is a contradiction. So $\ker\psi$ is not hyperbolic.

    It remains to prove that $\{\bg_k\}^\infty_{k=1}$ contains infinitely many isomorphism types of groups. Fix $i\in\{1,\dots\ell\}$. For each $k$, the group $H_i/N_{k,i}$ is abelian virtually-$\Z$. Hence, there exist integers $p_{1,k},\dots,p_{n-2,k}$ such that
    \[H_i/N_{k,i}\cong \Z\times \prod^{n-2}_{r=1} \Z/(p_{r,k}\Z).\]
    Note that $H_i/N_{k,i}$ has a finite-order element of order
    $p_{1,k}$. As $\bg_k$ contains $H_i/N_{k,i}$ as a subgroup, it contains a finite-order element of order $p_{1,k}$ as well.
    
    The group $(H_i/N_{k,i})/((G_k\cap H_i)/N_{k,i})\cong H_i/(G_k\cap H_i)$ is finite abelian. So there exists an integer $p_{n-1,k}$ such that
    \[H_i/(G_k\cap H_i)\cong \prod^{n-1}_{r=1} \Z/(p_{r,k}\Z).\]
    As $\{G_k\cap H_i\}^\infty_{k=1}$ is a nested sequence of groups and $\bigcap^\infty_{k=1} (G_k\cap H_i)=\{1\}$, we have $\lim_{k\rightarrow\infty}p_{r,k}=\infty$ for $r=1,\dots,n-1$. Therefore, for every $C>0$, there exists $k>0$ such that $\bg_k$ contains a finite-order element of order at least $C$.

    On the other hand, each $\bg_k$ is a hyperbolic group, and thus cannot contain finite-order elements of arbitrary large order. So $\{\bg_k\}^\infty_{k=1}$ must contain an infinite number of isomorphism types of groups.
\end{proof}

\subsection{Deficiency} 
Deficiency is an important invariant of a finitely presented group. A positive deficiency suggests that the group may be large (i.e.~a finite index subgroup admits a homomorphism onto a non-cyclic free group)  or possess a certain degree of freedom \cite{BM78}, which often manifests in its homological or geometric properties. In contrast, groups with non-positive deficiency tend to have a more rigid and constrained structure, such as many lattices in simple Lie groups \cite{Lott99}.

\begin{corollary}\label{cor. deficiency}
    Let $G$ be a virtually locally indicable group of type $F_3$ that has a finite family of hyperbolically embedded subgroups $\{H_i\}^n_{i=1}$. Then for sufficiently deep families of normal subgroups $\{N_i\lhd H_i\}^n_{i=1}$ such that $H_i/N_i$ is of type $F_3$ for all $i$, the group $G/\ll \bigcup^n_{i=1} N_i \rr$ is finitely presented and we have
\begin{align*}
     \defi(G/\ll \bigcup^n_{i=1} N_i \rr)
     \leqslant &b^{(2)}_1(G)-b^{(2)}_2(G)+\sum^n_{i=1} b^{(2)}_1(H_i/N_i)+\sum^n_{i=1} b^{(2)}_2(H_i)+1\\
     \leqslant &b^{(2)}_1(G)-b^{(2)}_2(G)+\sum^n_{i=1} b^{(2)}_1(H_i)+\sum^n_{i=1}b^{(2)}_2(H_i)+1.
 \end{align*}

\end{corollary}

\begin{proof}
    For simplicity, denote
    \begin{align*}
        s&:=b^{(2)}_1(G)-b^{(2)}_2(G)+\sum^n_{i=1} b^{(2)}_1(H_i/N_i)+\sum^n_{i=1} b^{(2)}_2(H_i)+1\\
        t&:=b^{(2)}_1(G)-b^{(2)}_2(G)+\sum^n_{i=1} b^{(2)}_1(H_i)+\sum^n_{i=1}b^{(2)}_2(H_i)+1.
    \end{align*}

    Choose $\delta>0$ such that each of the open intervals $(s,s+2\delta)$ and $(t,t+(n+2)\delta)$ does not contain any integer.
    
    By \Cref{cor. gen finiteness conditions} and \Cref{thm. gen li}, there exists a family of finite subsets $\{\F_{1,i}\subset H_i\smallsetminus\{1\}\}^n_{i=1}$ such that if a family of normal subgroups $N_i\lhd H_i$ satisfies that, for all $i$, $N_i\cap \F_{1,i}=\emptyset$ and $H_i/N_i$ is of type $F_3$, then $G/\ll \bigcup^n_{i=1} N_i \rr$ is of type $F_3$ and
    \begin{equation}\label{eq. li dehn 10}
        b^{(2)}_2(G/\ll \bigcup^n_{i=1} N_i \rr)>b^{(2)}_2(G)-\sum^n_{i=1} b^{(2)}_1(H_i/N_i)-\sum^n_{i=1} b^{(2)}_2(H_i)-\delta.
    \end{equation}
    
    For each $i$, by \Cref{cor. upper bound ver 2}, there exists a finite subset $\F_{2,i}\subset H_i\smallsetminus\{1\}$ such that if a normal subgroup $N_i\lhd H_i$ satisfies $N_i\cap \F_{2,i}=\emptyset$, then
    \begin{equation}\label{eq. li dehn 11}
        b^{(2)}_1(H_i/N_i)<b^{(2)}_1(H_i)+\delta.
    \end{equation}
    
    By \Cref{cor. locally indicable dehn filling}, there exists a family of finite subsets $\{\F_{3,i}\subset H_i\smallsetminus\{1\}\}^n_{i=1}$ such that if a family of normal subgroups $\{N_i\lhd H_i\}^n_{i=1}$ satisfies that, for all $i$, $N_i\cap \F_{3,i}=\emptyset$, then
    \begin{equation}\label{eq. li dehn 12}
        b^{(2)}_1(G/\ll \bigcup^n_{i=1} N_i \rr)<b^{(2)}_1(G)+\delta.
    \end{equation}
    
    Let $\{N_i\lhd H_i\}$ be a family of normal subgroups such that, for all $i$,
    \[N_i\cap (\F_{1,i}\cup\F_{2,i}\cup \F_{3,i})=\emptyset\]
    and $H_i/N_i$ is of type $F_3$. Then from \cite[Lemma 7.16]{luck2002l2}, we conclude that
   \begin{align*}
        \defi(G/\ll \bigcup^n_{i=1} N_i \rr) &\leqslant b^{(2)}_1(G/\ll \bigcup^n_{i=1} N_i \rr)-b^{(2)}_2(G/\ll \bigcup^n_{i=1} N_i \rr)+1\\
        &<s+2\delta &\text{by \eqref{eq. li dehn 10} and \eqref{eq. li dehn 12}}\nonumber\\
        &<t+(n+2)\delta &\text{by \eqref{eq. li dehn 11}}\nonumber
   \end{align*}

   The desired result follows as $\defi(G/\ll \bigcup^n_{i=1} N_i \rr)$ is an integer.
\end{proof}

\begin{corollary}\label{cor. defi lattice} 
    Let $G$ be a non-uniform arithmetic lattice in $SO(n,1)$, $n\geq 3$. Denote by $\{H_i\}_{i=1}^k$  the collection of peripheral subgroups. Then for every family of sufficiently deep normal subgroups $\{N_i\lhd H_i\}_{i=1}^k$, we have
    $$\defi(\bg)\leqslant 1.$$
        In addition, if $n=4$, then 
    $$\defi(\bg)\leqslant 1-\frac{3}{4\pi^2}\mbox{vol}(\mathbb H^4/G).$$
\end{corollary}

\begin{proof}  By combining \Cref{thm. lattice sparse} and \Cref{prop. properties of vir special groups}, we see that $G$ is virtually locally indicable. Also, \Cref{thm. l2 betti hyp mnfl} implies that $b^{(2)}_1(G)=0$. Since the peripheral subgroups of $G$ are virtually abelian, they have vanishing $L^2$-Betti numbers in positive degrees. The first claim now follows from applying \Cref{cor. deficiency}.

When $n=4$, one gets a sharper bound as follows:
\begin{equation*}
\begin{aligned}[b]
\defi(\bg)
  &\leq 1-b^{(2)}_2(G)
  &&\text{by \Cref{cor. deficiency},}\\
  &\leq 1-\chi(G)
  &&\text{by \Cref{thm. l2 betti hyp mnfl},}\\
  &\leqslant
    1-\frac{3}{4\pi^2}\operatorname{vol}(\mathbb H^4/G)
  &&\text{by \cite[(2.13) and (2.14)]{Lott99}.}
\end{aligned}
\qedhere
\end{equation*}
\end{proof}

\bibliographystyle{alpha}
\bibliography{bin_refs}

\end{document}